\documentclass[11pt]{amsart}

\usepackage{amssymb, amscd}
\usepackage{epsfig, mathtools, tensor}
\usepackage[vmargin=1in, hmargin=1.25in]{geometry}
\usepackage[font=small,format=plain,labelfont=bf,up,textfont=it,up]{caption}
\usepackage{microtype}
\usepackage[shortlabels]{enumitem}
\setlist[itemize]{nosep}
\usepackage[backref=page, bookmarks, bookmarksdepth=2, colorlinks=true, linkcolor=blue, citecolor=blue, urlcolor=blue]{hyperref}
\usepackage{etoolbox}
\usepackage{tikz-cd}
\usepackage{tabularray}
\apptocmd{\thebibliography}{\raggedright}{}{}
\usepackage[only,llbracket,rrbracket]{stmaryrd}
\usepackage{refcount}

\mathtoolsset{showonlyrefs}

\makeatletter
\patchcmd{\@maketitle}{\global\topskip42\p@\relax}
  {\global\topskip42\p@\relax \vspace*{-38pt}}
  {}{}
\makeatother

\setenumerate[0]{label=(\alph*)}

\renewcommand*{\backref}[1]{}
\renewcommand*{\backrefalt}[4]{%
    \ifcase #1 (Not cited.)%
    \or        (Cited on page~#2.)%
    \else      (Cited on pages~#2.)%
    \fi}

\DeclareSymbolFont{bbold}{U}{bbold}{m}{n}
\DeclareSymbolFontAlphabet{\mathbbold}{bbold}

\DeclareSymbolFont{extraup}{U}{zavm}{m}{n}
\DeclareMathSymbol{\varheartsuit}{\mathalpha}{extraup}{86}

\newcommand{\arxiv}[1]{\href{http://arxiv.org/abs/#1}{{\tt arXiv:#1}}}

\numberwithin{equation}{section}

\theoremstyle{plain}
\newtheorem{theorem}{Theorem}[section]
\newtheorem{maintheorem}{Theorem}
\newtheorem{maincorollary}[maintheorem]{Corollary}
\newtheorem{proposition}[theorem]{Proposition}
\newtheorem{lemma}[theorem]{Lemma}

\newtheorem{corollary}[theorem]{Corollary}
\newtheorem{conjecture}[theorem]{Conjecture}

\newtheorem*{unnumberedclaim}{Claim}

\newtheorem{stepx}{Step}
\newenvironment{step}[1]
 {\renewcommand\thestepx{#1}\stepx}
 {\endstepx}

\newtheorem{casex}{Case}
\newenvironment{case}[1]
 {\renewcommand\thecasex{#1}\casex}
 {\endcasex}

\newtheorem{claimx}{Claim}
\newenvironment{claim}[1]
 {\renewcommand\theclaimx{#1}\claimx}
 {\endclaimx}

\providecommand{\previous}{}
\newtheorem*{primedtheoreminner}{Theorem \previous}
\makeatletter
\newenvironment{primedtheorem}[1]{%
  \edef\previous{\getrefnumber{#1}$'$}%
  \let\@currentlabel\previous
  \primedtheoreminner
}{\endprimedtheoreminner}
\makeatother

\providecommand{\previous}{}
\newtheorem*{numberedtheoreminner}{Theorem \previous}
\makeatletter
\newenvironment{numberedtheorem}[2]{%
  \edef\previous{\getrefnumber{#1}.#2}%
  \let\@currentlabel\previous
  \numberedtheoreminner
}{\endnumberedtheoreminner}
\makeatother

\theoremstyle{definition}
\newtheorem{asm}[theorem]{Assumption}

\newtheorem{defn}[theorem]{Definition}
\newenvironment{definition}[1][]{\begin{defn}[#1]\pushQED{\qed}}{\popQED \end{defn}}
\newtheorem{notn}[theorem]{Notation}

\newtheorem{covn}[theorem]{Convention}
\newenvironment{convention}[1][]{\begin{covn}[#1]\pushQED{\qed}}{\popQED \end{covn}}

\theoremstyle{remark}
\newtheorem{rmk}[theorem]{Remark}
\newenvironment{remark}[1][]{\begin{rmk}[#1] \pushQED{\qed}}{\popQED \end{rmk}}
\newtheorem{eg}[theorem]{Example}
\newenvironment{example}[1][]{\begin{eg}[#1] \pushQED{\qed}}{\popQED \end{eg}}

\DeclareMathOperator{\GL}{GL}

\DeclareMathOperator{\SL}{SL}

\DeclareMathOperator{\Sp}{Sp}

\DeclareMathOperator{\SO}{SO}

\DeclareMathOperator{\fgl}{\mathfrak{gl}}

\newcommand\R{\ensuremath{\mathbb{R}}}
\newcommand\C{\ensuremath{\mathbb{C}}}
\newcommand\Z{\ensuremath{\mathbb{Z}}}
\newcommand\Q{\ensuremath{\mathbb{Q}}}
\newcommand\N{\ensuremath{\mathbb{N}}}

\DeclareMathOperator{\HH}{H}
\newcommand\RH{\ensuremath{\widetilde{\HH}}}
\DeclareMathOperator{\CC}{C}

\newcommand\RC{\ensuremath{\widetilde{\CC}}}
\DeclareMathOperator{\Tor}{Tor}

\DeclareMathOperator{\Aut}{Aut}

\DeclareMathOperator{\Ind}{Ind}
\DeclareMathOperator{\Res}{Res}

\DeclareMathOperator{\Char}{char}

\newcommand\Span[1]{\ensuremath{\langle #1 \rangle}}
\newcommand\SpanSet[2]{\ensuremath{\langle \text{#1 $|$ #2} \rangle}}
\newcommand\Set[2]{\ensuremath{\left\{\text{#1 $|$ #2}\right\}}}

\newcommand\cB{\ensuremath{\mathcal{B}}}

\newcommand\cL{\ensuremath{\mathcal{L}}}

\newcommand\cP{\ensuremath{\mathcal{P}}}

\newcommand\cT{\ensuremath{\mathcal{T}}}

\newcommand\fP{\ensuremath{\mathfrak{P}}}

\newcommand\fR{\ensuremath{\mathfrak{R}}}
\newcommand\fS{\ensuremath{\mathfrak{S}}}

\newcommand\fg{\ensuremath{\mathfrak{g}}}

\newcommand\bA{\ensuremath{\mathbf{A}}}
\newcommand\bB{\ensuremath{\mathbf{B}}}

\newcommand\bD{\ensuremath{\mathbf{D}}}

\newcommand\bG{\ensuremath{\mathbf{G}}}
\newcommand\bH{\ensuremath{\mathbf{H}}}

\newcommand\bL{\ensuremath{\mathbf{L}}}

\newcommand\bP{\ensuremath{\mathbf{P}}}
\newcommand\bQ{\ensuremath{\mathbf{Q}}}
\newcommand\bR{\ensuremath{\mathbf{R}}}
\newcommand\bS{\ensuremath{\mathbf{S}}}
\newcommand\bT{\ensuremath{\mathbf{T}}}
\newcommand\bU{\ensuremath{\mathbf{U}}}
\newcommand\bV{\ensuremath{\mathbf{V}}}
\newcommand\bW{\ensuremath{\mathbf{W}}}
\newcommand\bX{\ensuremath{\mathbf{X}}}
\newcommand\bY{\ensuremath{\mathbf{Y}}}
\newcommand\bZ{\ensuremath{\mathbf{Z}}}

\newcommand\bb{\ensuremath{\mathbf{b}}}

\newcommand\bPhi{\ensuremath{\mathbf{\Phi}}}
\newcommand\bDelta{\ensuremath{\mathbf{\Delta}}}

\newcommand\bbA{\ensuremath{\mathbb{A}}}

\newcommand\bbF{\ensuremath{\mathbb{F}}}
\newcommand\bbG{\ensuremath{\mathbb{G}}}

\newcommand\bbP{\ensuremath{\mathbb{P}}}

\newcommand\bbS{\ensuremath{\mathbb{S}}}

\newcommand\tf{\ensuremath{\widetilde{f}}}

\newcommand\tw{\ensuremath{\widetilde{w}}}

\newcommand\ok{\ensuremath{\overline{k}}}

\newcommand\oDelta{\ensuremath{\overline{\Delta}}}

\newcommand\hbb{\ensuremath{\widehat{\bb}}}

\newcommand\dA{\ensuremath{\operatorname{A}}}
\newcommand\dB{\ensuremath{\operatorname{B}}}
\newcommand\dC{\ensuremath{\operatorname{C}}}
\newcommand\dBC{\ensuremath{\operatorname{BC}}}
\newcommand\dD{\ensuremath{\operatorname{D}}}
\newcommand\dE{\ensuremath{\operatorname{E}}}
\newcommand\dF{\ensuremath{\operatorname{F}}}
\newcommand\dG{\ensuremath{\operatorname{G}}}
\newcommand\dX{\ensuremath{\operatorname{X}}}
\newcommand\dZ{\ensuremath{\operatorname{Z}}}

\newcommand*{\Cdot}[1][1.25]{%
  \mathpalette{\CdotAux{#1}}\cdot%
}
\newdimen\CdotAxis
\newcommand*{\CdotAux}[3]{%
  {%
    \settoheight\CdotAxis{$#2\vcenter{}$}%
    \sbox0{%
      \raisebox\CdotAxis{%
        \scalebox{#1}{%
          \raisebox{-\CdotAxis}{%
            $\mathsurround=0pt #2#3$%
          }%
        }%
      }%
    }%
    \dp0=0pt %
    \sbox2{$#2\bullet$}%
    \ifdim\ht2<\ht0 %
      \ht0=\ht2 %
    \fi
    \sbox2{$\mathsurround=0pt #2#3$}%
    \hbox to \wd2{\hss\usebox{0}\hss}%
  }%
}

\newcommand\St{\ensuremath{\operatorname{St}}}
\newcommand\Tits{\ensuremath{\cT}}
\newcommand\bPhik{\ensuremath{\tensor*[_k]{\bPhi}{}}}
\newcommand\bPhiQ{\ensuremath{\tensor*[_{\Q}]{\bPhi}{}}}
\newcommand\bDeltak{\ensuremath{\tensor*[_k]{\bDelta}{}}}
\newcommand\bDeltaQ{\ensuremath{\tensor*[_{\Q}]{\bDelta}{}}}

\newcommand\bGder{\ensuremath{\bG_{\operatorname{der}}}}

\newcommand\diag{\ensuremath{\operatorname{diag}}}
\newcommand\ssE{\ensuremath{\operatorname{E}}}
\newcommand\rank{\ensuremath{\operatorname{rank}}}
\newcommand\sign{\ensuremath{\operatorname{sign}}}
\newcommand\Apart[1]{\ensuremath{\bbA\llbracket #1 \rrbracket}}
\newcommand\Gen[1]{\ensuremath{\llbracket #1 \rrbracket}}
\newcommand\unit{\ensuremath{\operatorname{\mathbbold{1}}}}

\DeclareRobustCommand\Bind{b}

\title[Homological vanishing for the Steinberg representation II]{Homological vanishing for the Steinberg representation II: reductive groups and integral conjectures}

\author{Jeremy Miller}
\address{Dept of Mathematics; Purdue University; 150 N. University; West Lafayette, IN 47907}
\email{jeremykmiller@purdue.edu}

\author{Peter Patzt}
\address{Dept of Mathematics; University of Oklahoma; 601 Elm Ave; Norman, OK 73019}
\email{ppatzt@ou.edu}

\author{Andrew Putman}
\address{Dept of Mathematics; University of Notre Dame; 255 Hurley Hall; Notre Dame, IN 46556}
\email{andyp@nd.edu}

\thanks{JM was supported by NSF grants DMS-2202943 and DMS-2504473 and a Simons Foundation Travel Support for Mathematicians grant.  
PP was supported by NSF grant DMS-2405310 and
a Simons Foundation Collaboration Grant.  AP was supported by NSF grant DMS-2305183.}

\begin{document}

\begin{abstract}
We prove that the homology groups of any connected reductive group over a field with
coefficients in the Steinberg representation vanish in a range.  The generalizes
work of Ash--Putman--Sam on the classical split groups.
We state a connectivity conjecture 
that would allow us to prove such a vanishing result for
$\SL_n(\Z)$, as was conjectured by Church--Farb--Putman.  We prove some special
cases of this conjecture and use it to refine known results about the first and second
of homology of $\SL_n(\Z)$ with Steinberg coefficients.
\end{abstract}

\maketitle
\thispagestyle{empty}

\tableofcontents

\vspace{-35pt}
\section{Introduction}
\label{section:introduction}

Let\footnote{Here are we using the functorial language of algebraic groups.  For an algebraic
group $\bG$ over a field $k$, there is a group
$\bG(k)$ of $k$-points of $\bG$, and more generally for any commutative $k$-algebra $A$ there is a group $\bG(A)$
of $A$-points.  For $\bG = \GL_n$, these are the groups $\GL_n(k)$ and $\GL_n(A)$.}
$\bG$ be a (connected)\footnote{Here ``connected'' refers to the Zariski topology.  Our convention
is that all reductive groups are connected.} reductive group\footnote{For a reader who has not seen it before, the definition of a reductive
group is not particularly enlightening and we refer them to standard sources for it, e.g., \cite{BorelBook, MilneBook}.
A lot of this paper can be understood by assuming that $\bG$ is $\GL_n$ or $\SL_n$, or more generally one
of the classical split groups $\{\GL_n, \SL_n, \Sp_{2n}, \SO_{n,n}, \SO_{n,n+1}\}$.} 
 over a field $k$.  For example, $\bG$ might
be $\GL_n$ or $\Sp_{2n}$.  The Steinberg representation of $\bG$ encodes the combinatorial
structure of the parabolic\footnote{A $k$-subgroup $\bP$ of $\bG$ is parabolic if $\bG/\bP$ is a projective variety.  For example,
for $\bG = \GL_{n+1}$ the stabilizer $\bP$ of a line $L$ in $k^{n+1}$ is a parabolic subgroup with
$\bG / \bP \cong \bbP^n_k$; indeed, $\bG(k) = \GL_{n+1}(k)$ acts transitively on the set
$\bbP^n_k(k)$ of lines in $k^{n+1}$ and the stabilizer of the point $L \in \bbP^n_k(k)$ is $\bP(k)$.  Projective varieties are compact from the viewpoint
of algebraic geometry, so parabolic subgroups are ``large'' subgroups.} $k$-subgroups of $\bG$ and plays a fundamental role
in representation theory and algebraic topology.  For the classical split groups
\[\bG \in \{\GL_n, \SL_n, \Sp_{2n}, \SO_{n,n}, \SO_{n,n+1}\},\]
Ash--Putman--Sam \cite{AshPutmanSam} proved that the homology of
 $\bG(k)$ with
coefficients in the Steinberg representation vanishes in a range.  
Using a new approach inspired by Miller--Patzt--Wilson's work on the Rognes connectivity
conjecture \cite{MillerPatztWilsonRognes}, this paper extends \cite{AshPutmanSam} to all reductive groups.\footnote{For example, we
can handle $k$-forms of $\GL_n$, i.e., algebraic groups $\bG$ over $k$ that are not necessarily isomorphic to $\GL_n$, but whose
base-change $\bG_{\ok}$ to an algebraic closure $\ok$ is isomorphic to $\GL_n$.}
Even for the classical split groups, our approach 
is technically easier.\footnote{To handle general reductive groups we will have to use
a lot of their basic structural properties, so this might not be clear on a first reading.  One
concrete simplification is that unlike \cite{AshPutmanSam}, our proof is self-contained and
does not rely on understanding the topology of the complexes of partial bases and partial
isotropic bases.}

Church--Farb--Putman \cite{ChurchFarbPutmanConjecture} made a vanishing conjecture
about the high-dimensional cohomology of $\SL_n(\Z)$.  By Borel--Serre duality,
this is equivalent to the vanishing in a stable range of the
homology of $\SL_n(\Z)$ with coefficients in the Steinberg representation of $\SL_n(\Q)$.
This can be viewed as an integral refinement of the $\SL_n$ case Ash--Putman--Sam's vanishing theorem.

We show that our approach to the
vanishing of the homology of $\SL_n(\Q)$ with coefficients in its Steinberg representation
can be extended to $\SL_n(\Z)$ if a certain simplicial complex is highly connected, and
prove this high connectivity in some simple cases.  Using this, we refine the
known cases of this conjecture in degrees $1$ and $2$.

\subsection{Tits building}

Let the semisimple $k$-rank of $\bG$ be $n$.  The spherical Tits building for $\bG$, denoted
$\Tits(\bG)$, is an $(n-1)$-dimensional simplicial complex whose $r$-simplices
are the proper parabolic $k$-subgroups $\bP$ of $\bG$ whose semisimple $k$-rank is $n-1-r$.
The simplex corresponding to a proper parabolic $k$-subgroup $\bP'$ is a face of the one
corresponding to $\bP$ if\footnote{It is not obvious that this specifies a simplicial complex.  Also, we remark
that if $\bP$ is a parabolic $k$-subgroup of $\bG$ and $\bQ$ is another $k$-subgroup with $\bP \subset \bQ$, then $\bQ$
is a parabolic $k$-subgroup.} $\bP \subset \bP'$.  
The conjugation action of $\bG(k)$ on itself
permutes $k$-points $\bP(k)$ of the different parabolic $k$-subgroups $\bP$, giving an action of $\bG(k)$ on $\Tits(\bG)$.
See \cite{TitsBN, BrownBuildings} for more details.

\begin{example}
If $\bG$ is either $\GL_{n+1}$ or $\SL_{n+1}$, then the semisimple $k$-rank of $\bG$ is $n$ and
the parabolic $k$-subgroups of $\bG$ are the stabilizers of flags
\begin{equation}
\label{eqn:flag}
0 \subsetneq V_0 \subsetneq V_1 \subsetneq \cdots \subsetneq V_r \subsetneq k^{n+1}
\end{equation}
of subspaces of $k^{n+1}$.  The quotient of $\bG$ by the stabilizer of \eqref{eqn:flag} is a flag variety.
The semisimple $k$-rank of the stabilizer of \eqref{eqn:flag}
is $n-1-r$.  The Tits building $\Tits(\bG)$ is the simplicial complex
whose $r$-simplices are length-$r$ flags of subspaces of $k^{n+1}$.
\end{example}

\begin{remark}
If $\bG$ has no proper parabolic $k$-subgroups, then its semisimple $k$-rank is $0$ and
$\Tits(\bG) = \emptyset$.  Such $\bG$ are called {\em anisotropic}.  Easy examples
include $\GL_0$ and $\GL_1$.  A more interesting example is
the orthogonal group of a quadratic form $q$ on a $k$-vector space $V$ that is anisotropic, i.e., such that the only
$\vec{v} \in V$ with $q(\vec{v}) = 0$ is $\vec{v} = 0$.
\end{remark}

\subsection{Steinberg representation}

Let $\bbF$ be a commutative ring.  The Solomon--Tits theorem \cite{SolomonTits} says
that $\Tits(\bG)$ is homotopy equivalent to a wedge of $(n-1)$-dimensional spheres.  The
Steinberg representation of $\bG$ with coefficients in $\bbF$, denoted $\St(\bG;\bbF)$, 
is its reduced $(n-1)$-dimensional homology group with coefficients in $\bbF$, i.e., $\St(\bG;\bbF) = \RH_{n-1}(\Tits(\bG);\bbF)$.
For $\bbF = \Z$, we will omit $\bbF$ from our notation and just write $\St(\bG)$.
The action of $\bG(k)$ on $\Tits(\bG)$ induces an action of $\bG(k)$ on $\St(\bG;\bbF)$.
See \cite{HumphreysSurvey, SteinbergSurvey} for surveys.

These are finite-dimensional representations of $\bG(k)$
when $k$ is finite, and Steinberg and Curtis \cite{Steinberg1, Steinberg2, Steinberg3,
CurtisBN} showed that if $k$ is finite and $\bbF$ is a field, then $\St(\bG;\bbF)$ is
usually\footnote{For example, this holds if $\Char(\bbF) = 0$ or $\Char(\bbF) = \Char(k)$.} 
an irreducible representation of $\bG(k)$.  Except in degenerate cases, for infinite $k$ the representation $\St(\bG;\bbF)$
is infinite-dimensional.
Putman--Snowden \cite{PutmanSnowden} showed that in this case $\St(\bG;\bbF)$ is
irreducible for all fields $\bbF$. 

\begin{remark}
\label{remark:anisotropicst}
If $\bG$ is anisotropic, then $n=0$ and $\Tits(\bG) = \emptyset$.  Our convention
is 
$\St(\bG;\bbF) = \RH_{-1}(\Tits(\bG);\bbF) = \RH_{-1}(\emptyset;\bbF) = \bbF$.
The group $\bG(k)$ acts trivially on $\St(\bG;\bbF) = \bbF$.
\end{remark}

\subsection{Weak form of main theorem}

We can now state a weak form of our main theorem.

\begin{maintheorem}[Weak form]
\label{maintheorem:fieldseasy}
Let $\bG$ be a reductive group over a field $k$ with semisimple $k$-rank $n$ and let
$\bbF$ be a commutative ring.  Then $\HH_i(\bG(k);\St(\bG;\bbF))=0$ for
$i \leq b(n)$, where $b\colon \N \rightarrow \Z$ is a function with $\lim_{n \mapsto \infty} b(n) = \infty$.
\end{maintheorem}

\begin{remark}
For anisotropic groups $\bG$, we have $n=0$.  The function
$b$ satisfies $b(0) = -1$, so our theorem says nothing nontrivial about such $\bG$.
\end{remark}

\begin{remark}
If $k$ is a finite field and $\bG(k)$ is a finite group of Lie type, then $\St(\bG;k)$ is a projective
$\bG(k)$-module (see \cite{HumphreysSurvey}) and thus $\HH_i(\bG(k);\St(\bG;k)) = 0$ for $i \geq 0$.  However, 
$\St(\bG;\bbF)$ is not projective for a general commutative ring like $\bbF = \Z$, so even for finite groups
of Lie type Theorem \ref{maintheorem:fieldseasy} has nontrivial content.
\end{remark}

The stronger form of our main theorem gives a specific vanishing range for $\bG$ depending on
its relative root system, which we now discuss.

\subsection{Relative root system}

Let $\bT$ be a maximal\footnote{Here ``maximal'' means among $k$-split tori.} $k$-split torus\footnote{A {\em $k$-split torus} is an algebraic group $\bT$ defined over $k$ that is
isomorphic to $\left(\GL_1\right)^{d}$ for some $d \geq 0$, so $\bT(k) = (k^{\times})^{d}$.
If $\bG$ is $\GL_{n+1}$ or $\SL_{n+1}$, then its subgroup
of diagonal matrices is a maximal $k$-split torus.} in $\bG$ and let $\fg$ be the Lie algebra of $\bG$.  The
action of $\bT$ on $\fg$ decomposes into eigenspaces.  The
{\em relative root system} of $\bG$, denoted $\bPhik(\bG)$, is the set of characters of
$\bT$ that are nontrivial eigenvalues for this action.  As we will discuss
in more detail in \S \ref{section:reductive}, there is a natural inner product on $\bPhik(\bG)$
for which $\bPhik(\bG)$ is a root system.  All maximal $k$-split tori in $\bG$ are conjugate, so this
does not depend on the choice of $\bT$.

\begin{example}
Reduced (cf.\ Remark \ref{remark:reduced}) irreducible\footnote{A root system is irreducible if it does not
decompose as a nontrivial product of root systems.} root systems fall into $4$ infinite families $\{\dA_n, \dB_n, \dC_n, \dD_n\}$ and $5$ exceptional systems
$\{\dG_2, \dF_4, \dE_6, \dE_7, \dE_8\}$.  For the classical split groups, the relative root systems are as follows:
\begin{alignat*}{11}
&&\bPhik(\GL_{n+1}) &\cong &\dA_n &\quad &\bPhik(\SL_{n+1}) &\cong &\dA_{n} &\quad &\bPhik(\SO_{n,n+1}) &\cong &\dB_n \\
&&\bPhik(\Sp_{2n})  &\cong &\dC_n &\quad &\bPhik(\SO_{n,n}) &\cong &\dD_n   &      &                    &      &
\end{alignat*}
\par \vspace{-1.7\baselineskip}\qedhere
\end{example}

\begin{remark}
\label{remark:reduced}
When $k$ is not algebraically closed, it might be the case that maximal tori\footnote{An (algebraic) torus is an algebraic group
$\bT$ defined over $k$ whose base-change $\bT_{\overline{k}}$ to an algebraic closure is a $\overline{k}$-split torus.  For example,
let $k = \R$ and let $\bT = \SO(2)$.  For an $\R$-algebra $A$, we have $\bT(A) = \Set{$(x,y) \in A^2$}{$x^2+y^2=1$}$ with
the multiplication $(x_1,y_1) \Cdot (x_2,y_2) = (x_1 x_2-y_1 y_2, x_1 y_2 + x_2 y_1)$.  This is a non-split torus
with $\bT(\R) \cong \bbS^1$ and $\bT_{\C} \cong \GL_1$.  
The isomorphism
between $\bT_{\C}(\C) = \bT(\C)$ and $\GL_1(\C) = \C^{\times}$ takes $(x,y) \in \C^2$ with $x^2+y^2 = 1$ to $x + i y \in \C^{\times}$.} in $\bG$ are
not split over $k$, so $\bT$ is not a maximal torus.  In this case, the relative root system $\bPhik(\bG)$
might not be reduced, i.e., there might be roots $\alpha \in \bPhik(\bG)$ such that
$c \alpha \in \bPhik(\bG)$ for some $c \neq \pm 1$.  It turns out that the only additional possibility
here is $c = \pm 2$, and the only nonreduced irreducible root systems are those of
type $\dBC_n$.
\end{remark}

\begin{remark}
\label{remark:absoluteroot}
The {\em absolute root system} of $\bG$ is the root system $\bPhi(\bG)$ associated to
a maximal torus $\bS$.  All maximal tori in $\bG$ are conjugate, so this does not depend on the choice of $\bS$.  Choosing $\bS$ with
$\bT \subset \bS$, we can restrict characters from $\bS$ to $\bT$ and get a surjective map
\[\bPhi(\bG) \sqcup \{0\} \longrightarrow \bPhik(\bG) \sqcup \{0\}.\]
If $\bS$ splits\footnote{This always holds when $k$ is algebraically closed.} and thus $\bS = \bT$, then $\bPhi(\bG) = \bPhik(\bG)$ and we say that $\bG$ is {\em split}.
\end{remark}

\subsection{Stronger form of main theorem}

\begin{table}[t]
\begin{tblr}{r|l}
$\bPhi$ & $\bb(\bPhi)$ \\
\hline
\hline
$\dA_n$                                      & $\left\lfloor \frac{n-1}{2} \right\rfloor$ \\
\hline
$\dB_n$, $\dC_n$, $\dBC_n$                   & $\left\lfloor \frac{n-2}{2} \right\rfloor$ \\
\hline
$\dD_n$                                      & $\left\lfloor \frac{n-3}{2} \right\rfloor$ \\
\hline
$\dG_2$, $\dF_4$, $\dE_6$, $\dE_7$, $\dE_8$  & $0$ \\
\hline
\hline
\end{tblr}
\caption{The bounds $\bb(\bPhi)$ associated to irreducible relative root systems $\bPhi$.}
\label{table:bounds}
\end{table}

The irreducible root systems $\bPhi$ are listed in Table \ref{table:bounds}.  For each of them,
let $\bb(\bPhi)$ be the bound listed in that table.  A general root system $\bPhi$ can be written
in the form
$\bPhi = \bPhi_1 \times \cdots \times \bPhi_m$
with each $\bPhi_j$ irreducible, and we define
\begin{equation}
\label{eqn:reducibleb}
\bb(\bPhi) = (m-1) + \sum_{j=1}^m \bb(\bPhi_j).
\end{equation}
If $\bPhi = \emptyset$, then $m=0$ and $\bb(\bPhi) = -1$.  With this notation, our main theorem is as follows:

\begin{maintheorem}[Stronger form]
\label{maintheorem:fields}
Let $\bG$ be a reductive group over a field $k$ and let
$\bbF$ be a commutative ring.  Then $\HH_i(\bG(k);\St(\bG;\bbF))=0$ for
$i \leq \bb(\bPhik(\bG))$.
\end{maintheorem}

\begin{remark}
For the classical split groups, this was proved by Ash--Putman--Sam \cite{AshPutmanSam}.  Their
bound is the same as ours for $\bG \in \{\GL_n, \SL_n, \SO_{n,n+1}, \Sp_{2n}\}$.  However, for $\bG = \SO_{n,n}$
their bound is $\lfloor (n-2)/2 \rfloor$, while ours is $\lfloor (n-3)/2 \rfloor$.
\end{remark}

\begin{remark}
The bounds in Table \ref{table:bounds} for the exceptional root systems are not the best that can be achieved via our methods.\footnote{They 
are not entirely trivial since they say something about $\HH_0$.  Also, due to the $(m-1)$ in \eqref{eqn:reducibleb} these zeros do not affect
the fact that $\bb(\bPhik(\bG))$ goes to infinity as the semisimple $k$-rank of $\bG$ goes to infinity.}
With more work, one could instead achieve the bounds
\begin{align*}
\bb(\dG_2) &= 0, \quad 
\bb(\dF_4) =  0, \quad
\bb(\dE_6) =  1, \quad
\bb(\dE_7) =  1, \quad
\bb(\dE_8) =  2.
\end{align*}
However, doing this would make our proof less uniform, so we decided not to do it.
\end{remark}

\subsection{Arithmetic groups}

Our proof of Theorem \ref{maintheorem:fields} suggests an approach to also proving vanishing
theorems for certain arithmetic groups, which we now describe.
Consider $\bG = \SL_n$ over the field $k = \Q$.  This algebraic group has a natural $\Z$-form, so we
can talk about the arithmetic group $\bG(\Z) = \SL_n(\Z)$.  
Borel--Serre \cite{BorelSerreCorners} proved that the virtual cohomological dimension of $\SL_n(\Z)$ is $\binom{n}{2}$.  They also
proved that $\SL_n(\Z)$ is a virtual duality group with dualizing module the Steinberg representation.  This implies
that 
\begin{equation}
\label{eqn:duality}
\HH^{\binom{n}{2} - i}(\SL_n(\Z);\Q) \cong \HH_i(\SL_n(\Z);\St(\SL_n;\Q)) \quad \text{for $i \geq 0$}.
\end{equation}
Church--Farb--Putman \cite{ChurchFarbPutmanConjecture} conjectured that
\[\HH^{\binom{n}{2} - i}(\SL_n(\Z);\Q) = 0 \quad \text{for $n \geq i+2$}.\]
By \eqref{eqn:duality}, this is equivalent to the following:

\begin{conjecture}
\label{conjecture:cfp}
For $n \geq 2$, we have $\HH_i(\SL_n(\Z);\St(\SL_n;\Q)) = 0$ for $i \leq n-2$.  
\end{conjecture}

This can be viewed as an integral version of Theorem \ref{maintheorem:fields}, though the stable
range in it is stronger than Theorem \ref{maintheorem:fields}.
Conjecture \ref{conjecture:cfp} is known to hold for $n \leq 7$ (see \cite[Table 1]{ChurchFarbPutmanConjecture}).  It has
also been proven for $i=0$ by Lee--Szczarba \cite{LeeSzczarba}, for $i=1$ by Church--Putman \cite{ChurchPutmanCodimOne}, and
for $i=2$ by Br\"{u}ck--Miller--Patzt--Sroka--Wilson \cite{BruckMillerPatztSrokaWilson}.

\begin{remark}
More generally, it is natural to conjecture that something like Conjecture \ref{conjecture:cfp} holds for
other integral Chevalley groups $\bG$ like $\Sp_{2g}$ or $\SO_{n,n}$.  This is asked explicitly
in \cite[Question 1.2]{BruckSantosRegoSroka}, which in many cases proves this for $i=0$.  See
also \cite{BruckPatztSroka}, which proves this for $i=1$ in the special case of $\bG = \Sp_{2g}$.
Our conjectural approach for proving a version of Conjecture \ref{conjecture:cfp} can also
be generalized to this setting, though to keep this paper a reasonable length we chose not
to make this explicit.
\end{remark}

\subsection{Double Tits building}
\label{section:doubletits}

As part of their work on the Rognes connectivity conjecture \cite{RognesConjecture}, Miller--Patzt--Wilson \cite{MillerPatztWilsonRognes}
gave an approach to Conjecture \ref{conjecture:cfp}.  
Define $\Tits(\Z^n)$ to be the simplicial complex whose $r$-simplices are flags
\begin{equation}
\label{eqn:singleflag}
0 \subsetneq V_0 \subsetneq V_1 \subsetneq \cdots \subsetneq V_r \subsetneq \Z^n
\end{equation}
of direct summands of $\Z^n$.  There is a bijection between subspaces of $\Q^n$ and direct
summands of $\Z^n$ that takes a subspace $V \subset \Q^n$ to $V \cap \Z^n$.  Using this,
there is an isomorphism $\Tits(\Z^n) \cong \Tits(\SL_n)$.  Here just like above
$\SL_n$ is taken over the field $k = \Q$.

For every flag \eqref{eqn:singleflag}, there exists a basis $B = \{\vec{v}_1,\ldots,\vec{v}_n\}$
for $\Z^n$ such that each $V_i$ is the span of some subset of $B$.  Say that two flags
\begin{equation}
\label{eqn:compatibleflags}
0 \subsetneq V_0 \subsetneq V_1 \subsetneq \cdots \subsetneq V_r \subsetneq \Z^n \quad \text{and} \quad
0 \subsetneq W_0 \subsetneq W_1 \subsetneq \cdots \subsetneq W_s \subsetneq \Z^n
\end{equation}
of direct summands are {\em compatible} if there exists a basis $B = \{\vec{v}_1,\ldots,\vec{v}_n\}$ for $\Z^n$
such that each $V_i$ and $W_j$ is the span of some subset of $B$.  Define the {\em double building}
$\Tits^2(\Z^n)$ to be the simplicial complex whose $p$-simplices are pairs of of compatible flags
\eqref{eqn:compatibleflags} with $p=r+s+1$.  
Here if the $V_i$ (resp.\ the $W_j$) form the empty flag we have $r=-1$ (resp.\ $s = -1$).
We can identify $\Tits^2(\Z^n)$ with a subcomplex of the join $\Tits(\Z^n) \ast \Tits(\Z^n)$.  The
flags formed by the $V_i$ live in the first term, and the flags formed by the $W_j$ live in the second.

\subsection{Connectivity of double Tits building}

The complex $\Tits^2(\Z^n)$ is $(2n-3)$-dimensional.  Miller--Patzt--Wilson \cite{MillerPatztWilsonRognes}
proved that if it was $(2n-4)$-connected, then the Rognes connectivity conjecture \cite[Conjecture 12.3]{RognesConjecture}
would hold for the ring $\Z$.
Their results show that this would also imply the following slight weakening of Conjecture \ref{conjecture:cfp}:
\begin{itemize}
\item For $n \geq 2$ and $\bbF$ a field of characteristic $0$, we have 
$\HH_i(\GL_n(\Z);\St(\GL_n;\bbF)) = \HH_i(\SL_n(\Z);\St(\SL_n;\bbF)) = 0$
for $i \leq \lfloor (n-2)/2 \rfloor$.
\end{itemize}
This combines results from \cite{MillerPatztWilsonRognes} (cf.\ Theorem \ref{theorem:doubletitshomology} below)
with \cite[Theorems 7.2 \& 7.8]{MillerNagpalPatzt}.

\begin{remark}
By \cite{AshPutmanSam} (which is generalized in Theorem \ref{maintheorem:fields}), for all commutative rings $\bbF$ we have
$\HH_i(\GL_n(\Q);\St(\GL_n;\bbF)) = \HH_i(\SL_n(\Q);\St(\SL_n;\bbF)) = 0$
for $i \leq \lfloor (n-2)/2 \rfloor$.  This is the same range Miller--Patzt--Wilson's 
work conjecturally gives for $\GL_n(\Z)$ and $\SL_n(\Z)$.
\end{remark}

\begin{remark}
It is natural to hope that $\Tits^2(\Z^n)$ is $(2n-4)$-connected.  Indeed, if
we were working over a field $k$ rather than $\Z$, then it
is easy to see that any two flags are compatible.  Thus $\Tits^2(k^n) = \Tits(k^n) \ast \Tits(k^n)$,
which is $(2n-4)$-connected since $\Tits(k^n)$ is $(n-3)$-connected.
\end{remark}

\subsection{Integral theorems}

We refine this.  Roughly speaking, the following says
that to prove vanishing for $\GL_n(\Z)$ and $\SL_n(\Z)$ up to degree $i=b$, we must
prove that $\Tits^2(\Z^n)$ is $(n-2+b)$-connected for $n \gg 0$:

\begin{maintheorem}
\label{maintheorem:doublecomvanish}
Fix $b \geq 1$.  Assume that the following holds for all $n \geq 4$:
\begin{equation}
\tag{$\dagger$}\label{eqn:doublecomasm} \text{The space
$\Tits^2(\Z^n)$ is $n-2+\min(b,\lfloor (n-1)/2 \rfloor)$-connected.}
\end{equation}
Then $\HH_i(\GL_n(\Z);\St(\GL_n;\bbF)) = \HH_i(\SL_n(\Z);\St(\SL_n;\bbF)) = 0$
for all commutative rings $\bbF$ and all $n,i \geq 0$ such that:\noeqref{eqn:doublecomasm}
\begin{itemize}
\item $i \leq \min(b,\lfloor (n-2)/2 \rfloor)$ if $2$ and $3$ are invertible in $\bbF$; and
\item $i \leq \min(b,\lfloor (n-3)/3 \rfloor)$ in general.
\end{itemize}
\end{maintheorem}

\begin{remark}
As we will see, for $n \gg 0$ the hypothesis \eqref{eqn:doublecomasm} gives
a length-$b$ partial resolution of $\St(\GL_n;\bbF)$ from which 
$\HH_i(\GL_n(\Z);\St(\GL_n;\bbF))$ can be inductively calculated, and similarly
for $\SL_n(\Z)$.  
Similar resolutions appeared in the proofs of the cases $i=1$ and $i=2$ of Conjecture \ref{conjecture:cfp}; however,
it was unclear how to extend those resolutions further.
\end{remark}

\begin{remark}
More generally, if $\Tits^2(\Z^n)$ is $(n-2+b)$-connected for $n \gg 0$ then
our proof will show that
$\HH_i(\GL_n(\Z);\St(\GL_n;\bbF)) = \HH_i(\SL_n(\Z);\St(\SL_n;\bbF)) = 0$
for $i \leq b$ and $n \gg 0$.  We remark that $\Tits^2(\Z^n)$ will not be
$(n-2+b)$-connected for small $b$, which is why our connectivity assumption
in Theorem \ref{maintheorem:doublecomvanish} is phrased the way it is.

The specific connectivity ranges in Theorem \ref{maintheorem:doublecomvanish}
are the minimal ones needed to get the optimal vanishing ranges from our proof strategy.
Without further ideas, improvements to them would not yield better vanishing ranges.
\end{remark}

As evidence that $\Tits^2(\Z^n)$ should be highly-connected, we prove the following,
which verifies the hypotheses of Theorem \ref{maintheorem:doublecomvanish} for\footnote{In fact, for $n=4$ it is
slightly better since for $b=2$ Theorem \ref{maintheorem:doublecomvanish} only requires $\Tits^2(\Z^4)$ to be $3$-connected.} $b=2$.

\begin{maintheorem}
\label{maintheorem:doubletits}
The complex $\Tits^2(\Z^n)$ is $n$-connected for $n \geq 4$.
\end{maintheorem}

\begin{remark}
The proof of Theorem \ref{maintheorem:doubletits} uses the same technology as the proofs of the cases $i=1$ and $i=2$ of
Conjecture \ref{conjecture:cfp}.
\end{remark}

Combining Theorems \ref{maintheorem:doublecomvanish} and \ref{maintheorem:doubletits}, we deduce the following:

\begin{maincorollary}
\label{corollary:integralvanish}
We have $\HH_i(\GL_n(\Z);\St(\GL_n;\bbF)) = \HH_i(\SL_n(\Z);\St(\SL_n;\bbF)) = 0$
for:
\begin{itemize}
\item $i=1$ and $n \geq 4$ if $2$ and $3$ are invertible in the commutative ring $\bbF$; and
\item $i=1$ and $n \geq 5$ for general commutative rings $\bbF$; and
\item $i=2$ and $n \geq 6$ if $2$ and $3$ are invertible in the commutative ring $\bbF$; and
\item $i=2$ and $n \geq 8$ for general commutative rings $\bbF$.
\end{itemize}
\end{maincorollary}

\begin{remark}
Previously, $\HH_i(\GL_n(\Z);\St(\GL_n;\bbF))$
and $\HH_i(\SL_n(\Z);\St(\SL_n;\bbF))$ were known to 
vanish for the following $i$ and $n$ and $\bbF$:
\begin{center}
\begin{tblr}{c|c|c|c}
$\mathbf{i}$ & $\mathbf{n}$ & {\bf coefficients} & {\bf reference} \\
\hline
$i=0$          & $n \geq 2$   & $\bbF$ arbitrary                       & \cite{LeeSzczarba} \\
$i=1$          & $n \geq 3$   & $\bbF$ field of characteristic $0$     & \cite{ChurchPutmanCodimOne} \\
$i=1$          & $n \geq 6$   & $\bbF$ arbitrary                       & \cite{MillerNagpalPatzt} \\
$i=2$          & $n \geq 3$   & $\bbF$ field of characteristic $0$     & \cite{BruckMillerPatztSrokaWilson} 
\end{tblr}
\end{center}
In fact, with only a little more effort the proofs in \cite{BruckMillerPatztSrokaWilson,ChurchPutmanCodimOne}
work for $\bbF$ a commutative ring
in which all primes $p \leq n$ are invertible.  See the proof of Claim \ref{claim:strongintegralrank2.3} of Lemma \ref{lemma:strongintegralrank2} below.
\end{remark}

\subsection{Outline}
We start in Part \ref{part:general} with some general results about reductive groups.
This part closes by reducing Theorem \ref{maintheorem:fields} to 
the special case where $\bPhik(\bG)$ is irreducible and non-exceptional, i.e.,
$\bPhik(\bG) \in \{\dA_n,\dB_n,\dC_n,\dBC_n,\dD_n\}$.  
The heart of the paper consists of Parts \ref{part:an} and
\ref{part:bcn} and \ref{part:dn}, which prove Theorem \ref{maintheorem:fields} for those root systems.
The three cases are $\bPhik(\bG) = \dA_n$, and $\bPhik(\bG) \in \{\dB_n,\dC_n,\dBC_n\}$, and $\bPhik(\bG) = \dD_n$.
In Part \ref{part:integral} we generalize this to $\SL_n(\Z)$ and $\GL_n(\Z)$
and prove Theorem \ref{maintheorem:doublecomvanish}.  
Finally, in Part \ref{part:doubletits} we prove Theorem \ref{maintheorem:doubletits}.

\subsection{Conventions}
\label{section:conventions}

Unless otherwise specified, all algebraic groups $\bG$ are defined over a field $k$ that is fixed throughout the paper.
All subgroups, morphisms, quotients, etc.\ we discuss
are closed and defined over $k$.  In particular, we will just talk about parabolic subgroups rather than
parabolic $k$-subgroups and the semisimple rank of a group rather than the semisimple $k$-rank.  All reductive
groups are connected.

\part{General theory (Theorem \ref{maintheorem:fields})}
\label{part:general}

We survey some general properties of reductive groups in \S \ref{section:reductive}.
Next, in \S \ref{section:initialreductions} we perform some initial reductions.
The main tool for our proofs is a spectral sequence we introduce in \S \ref{section:resolution}.
To show how this spectral sequence can be used, in \S \ref{section:rank2} we prove
a surjectivity result in rank $2$.  We then develop some tools for understanding
reducible root systems in \S \ref{section:reducible}, and close in \S \ref{section:reductionirreducible}
by reducing Theorem \ref{maintheorem:fields} to the case where the relative root
system is irreducible and non-exceptional.  Those cases will be handled in subsequent parts.

\section{Reductive groups, Tits buildings, and the Steinberg representation}
\label{section:reductive}

Let $\bG$ be a reductive group (cf.\ the conventions in \S \ref{section:conventions}).
We thus have a group $\bG(k)$ of $k$-points of $\bG$.
In this section, we discuss some background about
$\bG$ needed for the rest of the paper.  To help make this understandable to readers less familiar with the general
picture, we will also explain what this background means for\footnote{We use $\GL_{n+1}$ since it has semisimple rank $n$, so
its numerology agrees with that of the general theory.} $\bG = \GL_{n+1}$.  Unless otherwise
specified, the results in this section are due to Borel--Tits (\cite{BorelTitsReductive, BorelTitsReductive2}, see \cite[Chapter V]{BorelBook}
and \cite{MilneBook} for textbook references).

\subsection{Levi factors and unipotent radicals}

Let $\bP$ be a parabolic subgroup of $\bG$.  We can write $\bP = \bV \rtimes \bL$, with
$\bV$ a normal unipotent subgroup of $\bP$ called its {\em unipotent radical} and $\bL$
a reductive subgroup of $\bP$ called a {\em Levi factor}.  All Levi factors of $\bP$
are conjugate.  The decomposition $\bP = \bV \rtimes \bL$ is the {\em Levi decomposition}
of $\bP$.

\begin{example}
\label{example:gllevidecomp}
Let $\bG = \GL_{n+1}$.  Write $n+1 = m+m'$ with $m,m' \geq 1$.  Let $\bP<\bG$ be the subgroup
fixing the subspace $k^m \oplus 0$ of $k^{n+1} = k^m \oplus k^{m'}$.
The Levi decomposition of $\bP$ is $\bV \rtimes \bL$, where:
\begin{itemize}
\item the Levi factor $\bL$ is the subgroup $\GL_m \times \GL_{m'}$ of $\bG$; and
\item the unipotent radical $\bV$ is the unipotent subgroup consisting of block matrices
$\left(\begin{matrix} \unit_m & V \\ 0 & \unit_{m'} \end{matrix}\right)$.  Here
$V$ is an arbitrary $m' \times m$ matrix.\qedhere
\end{itemize}
\end{example}

\subsection{Split tori, minimal parabolic subgroups, and relative Weyl groups}
\label{section:basicgroups}

Let $\bB$ be a minimal parabolic subgroup of $\bG$ and let $\bT$ be a maximal split torus in $\bB$.  If $k$ was
algebraically closed or more generally if $\bG$ was split,
then $\bB$ would be a Borel subgroup\footnote{A subgroup $\bB$ of $\bG$ is a Borel subgroup if
$\bB$ is connected and solvable, and the base-change $\bB_{\overline{k}}$ to an algebraic closure is maximal
among connected solvable subgroups of $\bG_{\overline{k}}$.  Borel subgroups are always parabolic, and if they
exist then every parabolic subgroup contains a Borel subgroup.} and $\bT$ would be a maximal torus, but in general this
need not hold.
The centralizer $Z_{\bG}(\bT)$ is a Levi factor of $\bB$.  Let $\bU$ be the unipotent
radical of $\bB$.  We then have $\bB = \bU \rtimes Z_{\bG}(\bT)$.
All pairs $(\bB,\bT)$ are conjugate in $\bG$.

\begin{example}
\label{example:basicgl}
For $\bG = \GL_{n+1}$, we can let $\bB$ be the Borel subgroup of upper triangular matrices,
$\bU$ be the unipotent subgroup of strictly upper triangular matrices, and $\bT$ be the group
of diagonal matrices.  In this case, $\bT = Z_{\bG}(\bT)$.  The group $\bT$ is a Levi factor
of $\bB$, the group $\bU$ is the unipotent radical of $\bB$, and $\bB = \bU \rtimes \bT$. 
\end{example}

Let $N_{\bG}(\bT)$ be the normalizer of $\bT$ in $\bG$.  The pair $(\bB(k),N_{\bG}(\bT)(k))$ forms a BN-pair in $\bG(k)$; see 
\cite[Theorem 5.2]{TitsBN} and the references therein for the proof, and
\cite[Chapter V.2]{BrownBuildings} for a textbook reference about BN-pairs.  We have
$N_{\bG}(\bT) \cap \bB = Z_{\bG}(\bT)$, and the relative Weyl group is
\[W = N_{\bG}(\bT) / Z_{\bG}(\bT).\]
This is a finite reflection group (see \S \ref{section:relativeroots} below for more on this).  
The Bruhat decomposition asserts that representatives of $W$ form a complete set
of $(\bB(k),\bB(k))$-double cosets of $\bG(k)$.  More precisely, for each $w \in W$ let $\tw \in N_{\bG}(\bT)(k)$ be a
representative.  We then have
\[\bG(k) = \bigsqcup_{w \in W} \bB(k) \cdot \tw \cdot \bB(k).\]

\begin{example}
For $\bG = \GL_{n+1}$, let $\bB$ and $\bU$ and $\bT$ be as in Example \ref{example:basicgl}.
The normalizer $N_{\bG}(\bT)$ is the group of monomial matrices, and
\[W = N_{\bG}(\bT) / Z_{\bG}(\bT) = N_{\bG}(\bT) / \bT \cong S_{n+1}.\]
In this case, we can identify $W$ with the group of permutation matrices in $N_{\bG}(\bT)(k) \subset \bG(k)$, but
in general the surjection $N_{\bG}(\bT)(k) \rightarrow W$ need not split.
\end{example}

\subsection{Structure of the building}
\label{section:structurebuilding}

The Tits building $\Tits(\bG)$ is the spherical building associated to the BN-pair $(\bB(k),N_{\bG}(\bT)(k))$.  See
\cite{BrownBuildings} for a textbook reference on buildings.  Let $n$ be the semisimple rank of $\bG$.  The space $\Tits(\bG)$
is an $(n-1)$-dimensional simplicial complex.  Its $r$-dimensional simplices are in bijection
with parabolic subgroups of semisimple rank $(n-1-r)$.  The group $\bG(k)$ acts on $\Tits(\bG)$ via its conjugation action on itself,
which permutes the parabolic subgroups.

The top-dimensional simplices are called the {\em chambers}, and are in bijection
with the minimal parabolic subgroups.  Letting $\fP_{\min}$ be the set of minimal parabolic subgroups of $\bG$, we thus
have $\CC_{n-1}(\Tits(\bG);\bbF) \cong \bbF[\fP_{\min}]$.  The group $\bG(k)$ acts transitively on the minimal parabolic subgroups, so every
element of $\fP_{\min}$ is of the form $g \cdot \bB$ for some $g \in \bG(k)$.

By definition, $\St(\bG;\bbF) = \RH_{n-1}(\Tits(\bG);\bbF)$.
Since $\CC_n(\Tits(\bG);\bbF) = 0$, we have
\begin{align*}
\St(\bG;\bbF) &= \ker(\CC_{n-1}(\Tits(\bG);\bbF) \stackrel{\partial}{\longrightarrow} \CC_{n-2}(\Tits(\bG);\bbF)) \\
              &= \ker(\bbF[\fP_{\min}] \stackrel{\partial}{\longrightarrow} \CC_{n-2}(\Tits(\bG);\bbF)).
\end{align*}
In particular, $\St(\bG;\bbF)$ is a subrepresentation of $\bbF[\fP_{\min}]$.

\begin{example}
For $\bG = \GL_{n+1}$, the semisimple rank of $\bG$ is $n$.  The
parabolic subgroups of $\bG$ are the stabilizers of flags
\begin{equation}
\label{eqn:flag2}
0 \subsetneq V_0 \subsetneq V_1 \subsetneq \cdots \subsetneq V_r \subsetneq k^{n+1}
\end{equation} 
of linear subspaces of $k^{n+1}$.  The semisimple rank of the stabilizer of \eqref{eqn:flag2}
is $n-1-r$.  The Tits building $\Tits(\bG)$ can thus be identified with the simplicial complex
whose $r$-simplices are length-$r$ flags of linear subspaces of $k^{n+1}$.  The minimal parabolic
subgroups are the stabilizers of complete flags, so the chambers in
$\Tits(\bG)$ can be identified with the complete flags.  A complete flag has $r = n-1$, so
$\Tits(\bG)$ is $(n-1)$-dimensional.
\end{example}

\subsection{Apartments}
\label{section:apartments}

The Steinberg representation $\St(\bG;\bbF)$ is spanned by apartment classes.  These are the homology classes of certain $(n-1)$-dimensional
subcomplexes of $\Tits(\bG)$ that are isomorphic to the Coxeter complexes of the relative Weyl group
$W = N_{\bG}(\bT) / Z_{\bG}(\bT)$.  They are defined as follows.
Consider some $g \in \bG(k)$.  For each $w \in W$, pick some representative $\tw \in N_{\bG}(\bT)(k)$.  Since $Z_{\bG}(\bT)$ is
contained in $\bB$, the element $\tw \cdot \bB$ does not depend on the choice of $\tw$.  The 
apartment class $\bbA_g$ is then
\[\bbA_g = \sum_{w \in W} (-1)^w g \tw \cdot \bB \in \St(\bG;\bbF) \subset \bbF[\fP_{\min}].\]
Here $(-1)^w$ is the sign of $w$, which is defined since $W$ is a finite reflection group.

\begin{example}
For $\bG = \GL_{n+1}$, 
the Coxeter complex of the relative Weyl group $W \cong S_{n+1}$ is the barycentric subdivision
of the boundary of an $n$-simplex, so in particular it has $(n+1)!$ top-dimensional simplices.  For $1 \leq j \leq n+1$, let $L_j \subset k^{n+1}$ be the line
spanned by the $j^{\text{th}}$ coordinate vector.  Recall that $\bB$ is the subgroup of upper triangular
matrices, i.e., the stabilizer of the flag
\[0 \subsetneq L_1 \subsetneq L_1 + L_2 \subsetneq \cdots \subsetneq L_1 + \cdots + L_{n} \subsetneq k^{n+1}.\]
Consider $g \in \bG(k) = \GL_{n+1}(k)$.  For $w \in W \cong S_{n+1}$, let $\tw$ be the associated permutation matrix.  We then have $\tw L_j = L_{w(j)}$
so $g \tw \cdot \bB$ is the stabilizer of the flag
\[0 \subsetneq g L_{w(1)} \subsetneq g L_{w(1)} + g L_{w(2)} \subsetneq \cdots \subsetneq g L_{w(1)} + \cdots + g L_{w(n)} \subsetneq k^{n+1}.\]
Here $g L_{w(j)}$ is the line spanned by the $w(j)^{\text{th}}$ column of $g$.  The stabilizers of these $(n+1)!$ flags form the chambers in $\bbA_g$.
\end{example}

\subsection{Basis for Steinberg}
\label{section:steinbergbasis}

The apartment classes $\bbA_g$ are not linearly independent.  Recall that $\bU$ is the unipotent radical of $\bB$.  One
version of the Solomon--Tits theorem \cite[Theorem 4.73]{BrownBuildings} says that
\[\cB = \Set{$\bbA_u$}{$u \in \bU(k)$}\]
forms a basis for $\St(\bG;\bbF)$.  In other words, as an $\bbF$-module $\St(\bG;\bbF)$ is isomorphic to
$\bbF[\bU(k)]$.  The action of $\bB(k)$ on $\St(\bG;\bbF) \cong \bbF[\bU(k)]$ comes from the conjugation action
of $\bB(k)$ on $\bU(k)$; however, it is not easy to describe the action of $\bG(k)$ on $\bbF[\bU(k)]$ coming from this
identification.  Another way of thinking about $\cB$ is that it is precisely the set of apartment classes that
when expressed as an element of $\bbF[\fP_{\min}]$ have $\bB$-coefficient $1$.

\begin{remark}
For $b \in \bB(k)$ the $\bB$-coefficient of $\bbA_b$ is also $1$.  However, we can write $b = u t$ with $u \in \bU(k)$ and $t \in \bT(k)$.
For $w \in W$, since $W = N_{\bG}(\bT) / Z_{\bG}(\bT)$ the element $t \tw$ is also an element of $N_{\bG}(\bT)(k)$ lifting $w$, so
in particular $t \tw \cdot \bB = \tw \cdot \bB$.  We conclude that
\[\bbA_b = \sum_{w \in W} (-1)^w b \tw \cdot \bB = \sum_{w \in W} (-1)^w u t \tw \cdot \bB = \sum_{w \in W} (-1)^w u \tw \cdot \bB = \bbA_u \in \cB. \qedhere\]
\end{remark}

\subsection{Reeder map}
\label{section:reedermap}

Recall that $Z_{\bG}(\bT)$ is a Levi factor of the minimal parabolic subgroup $\bB$ and
that $\bB = \bU \rtimes Z_{\bG}(\bT)$.  The semisimple rank of $Z_{\bG}(\bT)$ is $0$, i.e., it
is anisotropic.  This implies (cf.\ Remark \ref{remark:anisotropicst}) that
$\St(Z_{\bG}(\bT);\bbF) \cong \bbF$ is the trivial representation.  Another way of stating
the isomorphism $\St(\bG;\bbF) \cong \bbF[\bU(k)]$ is that as a representation
of $\bB(k)$, we have
\[\St(\bG;\bbF) \cong \bbF[\bU(k)] = \bbF[\bB(k) / Z_{\bG}(\bT)(k)] \cong \Ind_{Z_{\bG}(\bT)(k)}^{\bB(k)} \bbF \cong \Ind_{Z_{\bG}(\bT)(k)}^{\bB(k)} \St(Z_{\bG}(\bT);\bbF).\]
Reeder gave a beautiful generalization of this:

\begin{theorem}[{Reeder, \cite[Proposition 1.1]{Reeder}}]
\label{theorem:reedermap}
Let $\bG$ be a reductive group.  Let $\bP$ be a parabolic subgroup of
$\bG$ and let $\bL$ be a Levi factor of $\bP$.  As a representation of $\bP(k)$, we then have
\[\St(\bG;\bbF) \cong \Ind_{\bL(k)}^{\bP(k)} \St(\bL;\bbF).\]
\end{theorem}

Underlying the isomorphism in Theorem \ref{theorem:reedermap} is a map
$\St(\bL;\bbF) \rightarrow \St(\bG;\bbF)$ that we will call the {\em Reeder map}.

\subsection{Relative root system}
\label{section:relativeroots}

We now return to general properties of reductive groups.  
Let\footnote{If $\bG$ is a semisimple group like $\SL_{n+1}$, then $X(\bT) \cong \Z^n$, but in general it is only isomorphic to $\Z^m$ for some $m \geq n$.} $X(\bT) \cong \Z^m$ be the abelian group of characters of the maximal split torus $\bT$.  The conjugation action of $N_{\bG}(\bT)$ on $\bT$ induces
an action of $W = N_{\bG}(\bT) / Z_{\bG}(\bT)$ on $X(\bT)$.  Fix a $W$-invariant inner product on
\[X_{\R}(\bT) = X(\bT) \otimes \R \cong \R^m.\]
Let $\fg$ be the Lie algebra of $\bG$.
The action of $\bT$ on $\fg$ can be diagonalized, resulting in a decomposition
\[\fg = \fg_0 \oplus \bigoplus_{\alpha \in \bPhik(\bG)} \fg_{\alpha}.\]
Here $\bPhik(\bG) \subset X(\bT) \subset X_{\R}(\bT)$ is a finite set of nontrivial characters, and for $\alpha \in \{0\} \cup \bPhik(\bG)$
the subspace $\fg_{\alpha} \neq 0$ is the $\alpha$-eigenspace for the action of $\bT$ on $\fg$.  

The action of $W$ on $X_{\R}(\bT)$ preserves $\bPhik(\bG)$, and with respect to the $W$-invariant inner
product on $X_{\R}(\bT)$ the subset $\bPhik(\bG)$ is the relative root system\footnote{If $\bG$ is a non-semisimple group like $\GL_{n+1}$, then
$\bPhik(\bG)$ does not span $X_{\R}(\bT)$.  In this case, $\bPhik(\bG)$ is a root system in the subspace of $X_{\R}(\bT)$ spanned by $\bPhik(\bG)$.} of $\bG$.  
The orthogonal reflections of $X_{\R}(\bT)$ in the roots $\bPhik(\bG)$ preserve $\bPhik(\bG)$ and generate $W$, making $W$ into 
a finite reflection group.
Maximal split tori in $\bG$
are all conjugate, so this does not depend on the choice of $\bT$.  It is also independent of the choice of $W$-invariant inner product.

\begin{example}
For $\bG = \GL_{n+1}$, the group $\bT$ is the group of diagonal matrices.  For $t_1,\ldots,t_{n+1} \in k^{\times}$,
let $\diag(t_1,\ldots,t_{n+1})$ be the associated diagonal matrix.  We then have that $X(\bT) \cong \Z^{n+1}$, generated
by the elements $E_i \in X(\bT)$ defined via the formula
\[E_i(\diag(t_1,\ldots,t_{n+1})) = t_i \quad \text{for $t_1,\ldots,t_{n+1} \in k^{\times}$}.\]
The relative Weyl group $W \cong S_{n+1}$ permutes the $E_i$.  We have
\[\bPhik(\bG) = \Set{$E_i - E_j$}{$1 \leq i,j \leq n+1$ distinct} \cong \dA_{n},\]
and the subspace $\fg_{E_i - E_j}$ of $\fg \cong \fgl_{n+1}$ consists of matrices whose only nonzero entry is at position $(i,j)$.
The reflection in $E_i - E_j$ is the transposition $(i,j) \in W$.
\end{example}

\begin{remark}
As we discussed in Remark \ref{remark:absoluteroot}, the absolute root system $\bPhi(\bG)$ is the root
system associated to a maximal torus $\bS$.  If maximal tori are split,\footnote{This always holds if $k$ is algebraically closed.} then $\bG$ is said to be split
and $\bPhi(\bG) = \bPhik(\bG)$.  For example, $\bG = \GL_{n+1}$ is split.
While the absolute root system $\bPhi(\bG)$ is a reduced root system, 
the relative root system $\bPhik(\bG)$ need not be reduced.
See \cite{TitsSurvey} for a survey about the relationship between the absolute and relative root systems.
\end{remark}

\subsection{Root subgroups}

For $\alpha \in \bPhik(\bG)$, let $(\alpha)$ be the set of all elements of $\bPhik(\bG)$ of the form
$c \alpha$ for some $c \geq 1$ and let $\fg_{(\alpha)}$ be the direct sum of all $\fg_{\alpha'}$ with
$\alpha' \in (\alpha)$.  There then exists a unique connected unipotent subgroup $\bU_{(\alpha)}$ of $\bG$ called a {\em root subgroup} that is 
normalized by $\bT$ and has Lie algebra $\fg_{(\alpha)}$.  There are two cases:
\begin{itemize}
\item The root $\alpha$ is {\em non-multipliable}, i.e., the only $c \geq 1$ with $c \alpha \in \bPhik(\bG)$ is $c = 1$.
In this case, $\bU_{(\alpha)}$ is an abelian unipotent subgroup.  We will sometimes omit the parentheses and
just write $\bU_{\alpha}$.
\item The root $\alpha$ is {\em multipliable}, so $(\alpha) = \{\alpha, 2 \alpha\}$.  In this case, $\bU_{(\alpha)}$ is
$2$-step nilpotent and $\bU_{2 \alpha}$ is its center.
\end{itemize}
The root $\alpha$ is {\em non-divisible} if $\frac{1}{2} \alpha \notin \bPhik(\bG)$.

Recall that the relative Weyl group $W = N_{\bG}(\bT) / Z_{\bG}(\bT)$ acts on $\bPhik(\bG)$ via the conjugation action of
$N_{\bG}(\bT)$ on $\bT$.  For $w \in W$ represented by $\tw \in N_{\bG}(\bT)(k)$, we have
\[\tw \bU_{(\alpha)}(k) \tw^{-1} = \bU_{(w \cdot \alpha)}(k) \quad \text{for $\alpha \in \bPhik(\bG)$}.\]

\begin{example}
For $\bG = \GL_{n+1}$ and $\alpha = E_i - E_j \in \bPhi(\bG)$, the root subgroup $\bU_{\alpha}$ is the group of matrices that equal
the identity matrix except possibly at position $(i,j)$.
\end{example}

\begin{remark}
If $\bG$ is split, then $\bPhik(\bG) = \bPhi(\bG)$ is reduced and hence all the roots are non-multipliable.  Moreover,
each root subgroup $\bU_{\alpha}$ is one-dimensional and isomorphic to the additive group\footnote{The additive group $\bbG_a$ is the algebraic group
over $k$ with $\bbG_a(A)$ equal to the additive group of $A$ for any commutative $k$-algebra $A$.} $\bbG_a$.  
\end{remark}

\begin{remark}
In the nonsplit case, for $\alpha \in \bPhik(\bG)$ let $\{\alpha_1,\ldots,\alpha_r\}$ be the preimage of $(\alpha)$
under the surjection $\bPhi(\bG) \sqcup \{0\} \rightarrow \bPhik(\bG) \sqcup \{0\}$ that restricts characters from
a maximal torus to the maximal split torus $\bT$.  The absolute
root subgroups $\bU_{\alpha_j}$ are defined over an algebraic closure $\ok$, and $\bU_{(\alpha)}(\ok)$ is the subgroup of $\bG(\ok)$ generated
by the $\bU_{\alpha_j}(\ok)$.
\end{remark}

\subsection{Positive roots and Dynkin diagrams}
\label{section:positive}

\begin{table}[t]
\begin{tblr}{r|c}
$\bPhi$ & {\bf Dynkin Diagram} \\
\hline
\hline
\raisebox{-.5\height}{$\dA_n$}  & \raisebox{-.5\height}{\psfig{file=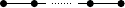,scale=2}}  \\
\hline
\raisebox{-.5\height}{$\dB_n$}  & \raisebox{-.5\height}{\psfig{file=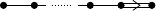,scale=2}}  \\
\hline
\raisebox{-.5\height}{$\dC_n$}  & \raisebox{-.5\height}{\psfig{file=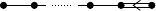,scale=2}}  \\
\hline
\raisebox{-.5\height}{$\dBC_n$} & \raisebox{-.5\height}{\psfig{file=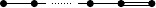,scale=2}} \\
\hline
\raisebox{-.5\height}{$\dD_n$}  & \raisebox{-.5\height}{\psfig{file=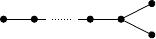,scale=2}}  \\
\hline
\hline
\end{tblr}
\caption{The Dynkin diagrams of irreducible non-exceptional root systems.}
\label{table:dynkin}
\end{table}

Recall that $\bU$ is the unipotent radical of our fixed minimal parabolic subgroup $\bB$.
Let $\bPhik^+(\bG)$ be the set of all $\lambda \in \bPhik(\bG)$ such that $\bU_{(\alpha)} \subset \bU$.  These
are the {\em positive roots}, and for every $\lambda \in \bPhik(\bG)$ there exists a unique sign $\epsilon = \pm 1$ such that
$\epsilon \lambda \in \bPhik^+(\bG)$.  Letting $\bPhik^+_{\text{nd}}(\bG)$ be the set of
non-divisible elements of $\bPhik^+(\bG)$, the product map
\begin{equation}
\label{eqn:uprod}
\prod_{\alpha \in \bPhik^+_{\text{nd}}(\bG)} \bU_{(\alpha)}(k) \rightarrow \bU(k)
\end{equation}
is a set-theoretic bijection.
Inside $\bPhik^+(\bG)$ is a unique set $\bDeltak(\bG)$ of simple roots, that is,
elements of $\bPhik^+(\bG)$ that cannot be written as the sum of two elements of
$\bPhik^+(\bG)$.  The set
$\bDeltak(\bG)$ is a linearly independent set of characters such that
\[\bPhik^+(\bG) = \Set{$\lambda \in \bPhik(\bG)$}{$\lambda = \sum_{\alpha \in \bDeltak(\bG)} c_{\alpha} \alpha$ with $c_{\alpha} \in \Z$ nonnegative}.\]
Recall from \S \ref{section:relativeroots} that the relative Weyl group $W$ is generated by reflections in the
roots from $\bPhik(\bG)$.  In fact, $W$ is generated by the reflections in the simple roots $\bDeltak(\bG)$.

The set $\bDeltak$ forms the vertices of the Dynkin diagram for $\bPhik(\bG)$.  A root system is irreducible (that is, not
isomorphic to a product) if its Dynkin diagram is connected.  The Dynkin diagrams of the irreducible non-exceptional\footnote{The exceptional
root systems are those of types $\{\dG_2,\dF_4, \dE_6, \dE_7, \dE_8\}$.  We will not need their Dynkin diagrams in this paper.}
root systems are listed in Table \ref{table:dynkin}.  A general root system is a product of irreducible root systems.  

\begin{example}
For $\bG = \GL_{n+1}$, the positive roots are
\[\bPhik^+(\bG) = \Set{$E_i-E_j$}{$1 \leq i < j \leq n+1$}\]
and the simple roots are
\[\bDeltak(\bG) = \{E_1-E_2,E_2-E_3,\ldots,E_{n}-E_{n+1}\}.\]
The root system $\bPhik(\bG)$ is of type $\dA_{n}$, and in the Dynkin diagram for $\dA_{n}$ depicted
in Table \ref{table:dynkin} the simple root $E_i-E_{i+1}$ is the $i^{\text{th}}$ node from the left.
\end{example}

\subsection{Other minimal parabolics containing torus}

The above choice of positive roots was dictated by our choice of minimal parabolic subgroup $\bB$ containing
the maximal split torus $\bT$.  The group $\bT$ is the unique maximal split torus in $\bB$.
Since $\bG(k)$ acts transitively by conjugation on the set of minimal parabolic
subgroups, the group $N_{\bG}(\bT)(k)$ acts transitively on the set of minimal parabolic subgroups containing
$\bT$.  Just like for all parabolic subgroups, we have $N_{\bG}(\bB) = \bB$.  It follows that
the $N_{\bG}(\bT)(k)$-stabilizer of $\bB$ is
\[\bB(k) \cap N_{\bG}(\bT)(k) = Z_{\bG}(\bT)(k).\]
In other words, the group $W = N_{\bG}(\bT) / Z_{\bG}(\bT)$ acts simply transitively
on the set of minimal parabolic subgroups containing $\bT$.  
For $w \in W$, let\footnote{With this notation, the apartment $\bbA_g$ corresponding
to $g \in \bG(k)$ (cf.\ \S \ref{section:apartments}) is $\bbA_g = \sum_{w \in W} (-1)^w g \cdot \bB_w$.} $\bB_w$ be the image of $\bB$ under the action of $w$ and let $\bU_w$ be the unipotent
radical of $\bB_w$.  We thus have $\bB_w = \bU_w \rtimes Z_{\bG}(\bT)$.  Just like in
\eqref{eqn:uprod}, the product map
\[\prod_{\alpha \in \bPhik^+_{\text{nd}}(\bG)} \bU_{(w(\alpha))}(k) \rightarrow \bU_w(k)\]
is a set-theoretic bijection.  If we used $\bB_w$ as our base minimal parabolic
instead of $\bB$, then the positive roots would be $w(\bPhik^+(\bG))$.  

\subsection{Opposite minimal parabolic}
\label{section:oppositeparabolic}

Let $\bPhik^-(\bG) = \Set{$-\lambda$}{$\lambda \in \bPhik^+(\bG)$}$, so
$\bPhik(\bG) = \bPhik^+(\bG) \sqcup \bPhik^-(\bG)$.  Let 
$S$ be the generating set for $W$ consisting of reflections in the simple roots $\bDeltak(\bG)$.  For $w \in W$, the
length of the minimal word in $S$ representing $w$ equals the number of $\lambda \in \bPhik^+(\bG)$
with $w(\lambda) \in \bPhik^-(\bG)$.  The group $W$ contains a unique element $w_0$ whose
$S$-word length is maximal.  This element satisfies $w_0(\bPhik^+(\bG)) = \bPhik^-(\bG)$.  Define
\[\bB^- = \bB_{w_0} \quad \text{and} \quad \bU^- = \bU_{w_0}.\]
The group $\bB^-$ is the {\em opposite} minimal parabolic subgroup to $\bB$.  It satisfies
$\bB \cap \bB^- = Z_{\bG}(\bT)$.  In particular, $\bB \cap \bU^- = 1$.

\begin{example}
For $\bG = \GL_{n+1}$, the group $\bB$ is the group of upper triangular matrices and $\bB^-$ is the group
of lower triangular matrices.
\end{example}

\subsection{Intersecting minimal parabolics}

Consider $\Lambda \subset \bPhik(\bG)$.  The set $\Lambda$ is {\em closed} if $\lambda_1+\lambda_2 \in \Lambda$
for all $\lambda_1,\lambda_2 \in \Lambda$ with $\lambda_1+\lambda_2 \in \bPhik(\bG)$.
It is {\em one-sided} if 
$\Lambda \subset w(\bPhik^+(\bG))$ for some $w \in W$, i.e., if $\Lambda$ lies in the set of positive roots
associated to some minimal parabolic containing $\bT$.  

Assume that $\Lambda$ is
closed and one-sided.  Let $\bU_{\Lambda}$ be the subgroup of $\bG$ generated by the $\bU_{(\lambda)}$ for 
$\lambda \in \Lambda$.  Our assumptions imply that $\bU_{\Lambda}$ is
a unipotent subgroup of $\bG$.  Let $\Lambda_{\text{nd}}$ be the set of all $\lambda \in \Lambda$
with $\frac{1}{2} \lambda \notin \Lambda$.  Our assumptions also
imply that the product map
\[\prod_{\lambda \in \Lambda_{\text{nd}}} \bU_{(\lambda)}(k) \rightarrow \bU_{\Lambda}(k)\]
is a set-theoretic bijection.  We then have the following:

\begin{lemma}
\label{lemma:intersectunipotent}
Let $\bG$ be a reductive group.  Fix a minimal parabolic subgroup $\bB$
with unipotent radical $\bU$.  Let $\Lambda \subset \bPhik(\bG)$ be closed and one-sided.  Set 
$\Lambda' = \Lambda \cap \bPhik^+(\bG)$.  We then have $\bB(k) \cap \bU_{\Lambda}(k) = \bU_{\Lambda'}(k)$.
\end{lemma}
\begin{proof}
We clearly have have $\bU_{\Lambda'}(k) \subset \bB(k) \cap \bU_{\Lambda}(k)$, so we must
prove the other inclusion.  Consider $u \in \bB(k) \cap \bU_{\Lambda}(k)$.
Let $\Lambda'' = \Lambda \setminus \Lambda'$.  The product map
\[\bU_{\Lambda''}(k) \times \bU_{\Lambda'}(k) \rightarrow \bU_{\Lambda}(k)\]
is a set-theoretic bijection, so we can write $u = u'' u'$ with $u'' \in \bU_{\Lambda''}(k)$
and $u' \in \bU_{\Lambda'}(k)$.  To prove that $u \in \bU_{\Lambda'}(k)$, it is
enough to prove that $u'' = 1$.
Since $\bU_{\Lambda'}(k) \subset \bU(k) \subset \bB(k)$, we
have $u' \in \bB(k)$.  Since $u$ is assumed to be in $\bB(k)$, it follows
that $u'' \in \bB(k)$ and thus that $u'' \in \bB(k) \cap \bU_{\Lambda''}(k)$.  
Since $\Lambda'' \subset \bPhik^-(\bG)$ we have $\bU_{\Lambda''}(k) \subset \bU^-(k)$.
Since $\bU^-(k) \cap \bB(k) = 1$, this 
implies that $\bU_{\Lambda''}(k) \cap \bB(k) = 1$ and thus that $u'' = 1$, as desired.
\end{proof}

The following corollary will play a key role in one of the base cases of our proof:

\begin{corollary}
\label{corollary:findu}
Let $\bG$ be a reductive group.  Fix a minimal parabolic subgroup $\bB$
with maximal split torus $\bT$ and unipotent radical $\bU$.  Letting $W = N_{\bG}(\bT) / Z_{\bG}(\bT)$
be the relative Weyl group, for each $w \in W$ fix some representative $\tw \in N_{\bG}(\bT)(k)$.
There then exists some $u \in \bU(k)$ such that the only $w_1,w_2 \in W$ with
$\tw_1 u \tw_2 \in \bB(k)$ are $w_1 = w_2 = 1$.
\end{corollary}

\begin{example}
For $\bG = \GL_{n+1}$, the group $\bU$ is the subgroup of strictly upper triangular matrices.  Pick $u \in \bU(k)$ 
such that all of its upper triangular entries are nonzero.
For $w \in W \cong S_{n+1}$, let $\tw$ be the associated permutation
matrix.  For $w_1, w_2 \in W$, the matrix $\tw_1 u \tw_2$ is obtained from $u$ by permuting the rows via $w_1$
and the columns via $w_2$.  It is an easy exercise to see that this will have a nonzero entry below
the diagonal unless $w_1 = w_2 = 1$.
\end{example}

\begin{proof}[{Proof of Corollary \ref{corollary:findu}}]
Enumerate $\bPhik^+_{\text{nd}}$ as $\bPhik^+_{\text{nd}} = \{\alpha_1,\ldots,\alpha_r\}$.  For $1 \leq j \leq r$, pick $u_j \in \bU_{(\alpha_j)}(k)$
with $u_j \neq 1$.  Set $u = u_1 \cdots u_r \in \bU(k)$, and consider $w_1,w_2 \in W$ with $\tw_1 u \tw_2 \in \bB(k)$.
We must prove that $w_1=w_2 = 1$.
Set $v = \tw_1 u \tw_2 \in \bB(k)$.  We have $\tw_1^{-1} = u \tw_2 v^{-1}$, so $\tw_2$ and $\tw_1^{-1}$ lie in the same $\bB(k)$-double
coset.  By the Bruhat decomposition, we have $\tw_2 = \tw_1^{-1}$.  Set $w = w_1$, so $v = \tw u \tw^{-1}$.  

Let $\Lambda = w(\bPhik^+(\bG))$ and $\Lambda' = \Lambda \cap \bPhik^+(\bG)$.  Lemma \ref{lemma:intersectunipotent} implies
that 
\[\bB(k) \cap \bU_w(k) = \bB(k) \cap \bU_{\Lambda}(k) = \bU_{\Lambda'}(k).\]
For $1 \leq j \leq r$, we have 
\[\tw u_j \tw^{-1} \in \bU_{(w(\alpha_j))}(k) \subset \bU_{w}(k).\]
We therefore have
\begin{equation}
\label{eqn:wuw}
\tw u \tw^{-1} = \left(\tw u_1 \tw^{-1}\right) \cdots \left(\tw u_r \tw^{-1}\right) \in \bB(k) \cap \bU_{w}(k) = \bU_{\Lambda'}.
\end{equation}
By construction, we have
$\Lambda'_{\text{nd}} \subset \bPhik^+_{\text{nd}}$.
Reordering the $\alpha_j$ if necessary, there thus exists some $1 \leq s \leq r$ with $\Lambda'_{\text{nd}} = \{\alpha_1,\ldots,\alpha_s\}$.
The product maps
\begin{alignat*}{4}
\bU_{(w(\alpha_1))}(k) &\times \cdots \times \bU_{(w(\alpha_s))}(k) \times \bU_{(w(\alpha_{s+1}))}(k) \times \cdots \times \bU_{(w(\alpha_r))}(k) &\ \longrightarrow &\ \bU_{w}(k) \\
\bU_{(\alpha_1)}(k)    &\times \cdots \times \bU_{(\alpha_s)}(k)    &\ \longrightarrow &\ \bU_{\Lambda'}(k)
\end{alignat*}
are set-theoretic bijections, and for each $1 \leq j \leq s$ there exists a unique $1 \leq \ell_j \leq r$ with
$\alpha_j = w(\alpha_{\ell_j})$.  Since each $u_j$ is nontrivial and \eqref{eqn:wuw} holds, we conclude
that in fact we have must have $r = s$, and thus $w(\bPhik^+(\bG)) = \bPhik^+(\bG)$.  In other words, the element
$w$ of the relative Weyl group $W$ takes each positive root to another positive root.  As we discussed
in \S \ref{section:oppositeparabolic}, this implies that $w = 1$, as desired.
\end{proof}

\subsection{Standard parabolic and Levi subgroups}
\label{section:standardlevi}

Recall that $\bDeltak(\bG) \subset \bPhik^+(\bG)$ is the set of simple roots.  Consider some $\Delta \subset \bDeltak(\bG)$.  Let $[\Delta] \subset \bPhik^+(\bG)$ be the closed
subset generated by $\Delta$ and let $\Psi(\Delta) = \bPhik^+(\bG) \setminus [\Delta]$.
Both $[\Delta]$ and $\Psi(\Delta)$ are closed and one-sided. 
Each $\alpha \in \Delta$ is a character of $\bT$, i.e., a homomorphism $\alpha\colon \bT \rightarrow \GL_1$.  Define\footnote{Here the $0$ indicates
that we are taking the connected component of the identity.}
\[\bT_{\Delta} = \left(\bigcap_{\alpha \in \Delta} \ker(\alpha)\right)^0 \subset \bT.\]
Let $\bL_{\Delta}$ be the centralizer of $\bT_{\Delta}$ in $\bG$ and let $\bP_{\Delta}$ be the subgroup of $\bG$
generated by $\bL_{\Delta}$ and $\bU_{\Psi(\Delta)}$.  The group $\bP_{\Delta}$ is a parabolic subgroup of $\bG$
with Levi factor $\bL_{\Delta}$ and unipotent radical $\bU_{\Psi(\Delta)}$, so
\[\bP_{\Delta} = \bU_{\Psi(\Delta)} \rtimes \bL_{\Delta}.\]
The parabolic subgroups of the form $\bP_{\Delta}$ for some $\Delta \subset \bDeltak(\bG)$ are the {\em standard
parabolic subgroups} and their Levi factors $\bL_{\Delta}$ are the {\em standard Levi subgroups}.  Every parabolic
subgroup is conjugate to a unique standard parabolic subgroup.  The standard Levi subgroup $\bL_{\Delta}$ is a
reductive group of semisimple rank $|\Delta|$ whose relative Dynkin diagram is the full subgraph
of the relative Dynkin diagram of $\bG$ with vertex set $\Delta$.

\begin{example}
\label{example:standardlevi}
Consider $\bG = \GL_8$, and let $\Delta = \{E_1-E_2, E_3-E_4,E_4-E_5,E_5-E_6\} \subset \bDeltak(\bG)$.  This corresponds to
the following subgraph of the Dynkin diagram for $\bPhi(\bG) \cong \dA_7$:\\
\centerline{\psfig{file=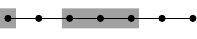,scale=2}}
We should have $\bPhi(\bL_{\Delta}) \cong \dA_1 \times \dA_3$, and indeed keeping in mind that $\bPhi(\GL_1) = \emptyset$ we have
\begin{align*}
\bT_{\Delta}(k) &= \Set{$\diag(a,a,b,b,b,b,c,d)$}{$a,b,c,d \in k^{\times}$} \subset \bT(k), \\
\bL_{\Delta}(k) &= Z_{\GL_8(k)}(\bT_{\Delta}(k)) = \GL_2(k) \times \GL_4(k) \times \GL_1(k) \times \GL_1(k) \subset \GL_8(k).
\end{align*}
More generally, if $\bG = \GL_{n+1}$ and $\Delta \subset \bDeltak(\bG)$ is such that
$\bDeltak_{\bG} \setminus \Delta$ contains $r$ roots, then
\[\bL_{\Delta} = \GL_{n_1} \times \cdots \times \GL_{n_{r+1}} \subset \GL_{n+1} \quad \text{with $n_1 + \cdots + n_{r+1} = n+1$}.\qedhere\]
\end{example}

\subsection{Conjugacy of Levi subgroups}

The following observation will play an important role in our proof:

\begin{lemma}
\label{lemma:leviconjugate}
Let $\bG$ be a reductive group with $\bPhik(\bG) = \dA_n$.
Let $\Delta, \Delta' \subset \bDeltak(\bG)$ be such that $\bPhik(\bL_{\Delta}) = \bPhik(\bL_{\Delta'}) = \dA_m$ for some
$m \geq 1$.  Then $\bL_{\Delta}(k)$ and $\bL_{\Delta'}(k)$ are conjugate\footnote{Note that
our assumptions do not imply that $\bP_{\Delta}(k)$ and $\bP_{\Delta'}(k)$ are conjugate subgroups of $\bG(k)$.}
subgroups of $\bG(k)$.
\end{lemma}

\begin{example}
Let $\bG = \GL_{n+1}$, so $\bPhi(\bG) \cong \dA_{n}$.  Let $\Delta,\Delta' \subset \bDeltak(\bG)$ be
as in Lemma \ref{lemma:leviconjugate}.  Since $\bPhi(\bL_{\Delta}) \cong \dA_m$, the
set $\Delta$ consists of $m$ consecutive vertices in the Dynkin diagram $\dA_{n}$ for $\bG = \GL_{n+1}$, for instance:\\
\centerline{\psfig{file=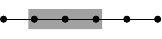,scale=2}}
Keeping in mind that $\bPhi(\GL_{m+1}) \cong \dA_m$, just like in Example \ref{example:standardlevi} we can write
\[\bL_{\Delta} = \left(\GL_1\right)^a \times \GL_{m+1} \times \left(\GL_1\right)^b \subset \GL_{n+1} \quad \text{with $a+(m+1)+b=n+1$}.\]
Similarly,
\[\bL_{\Delta'} = \left(\GL_1\right)^{a'} \times \GL_{m+1} \times \left(\GL_1\right)^{b'} \subset \GL_{n+1} \quad \text{with $a'+(m+1)+b'=n+1$}.\]
Let $w \in \bG(k)$ be an $(n+1) \times (n+1)$ permutation matrix representing a permutation
taking the subset $\{a+1,\ldots,a+m+1\}$ of $\{1,\ldots,n\}$ to $\{a'+1,\ldots,a'+m+1\}$.
We then have $w \bL_{\Delta}(k) w^{-1} = \bL_{\Delta'}(k)$, as desired.
\end{example}

\begin{proof}[{Proof of Lemma \ref{lemma:leviconjugate}}]
The full subgraphs of the Dynkin diagram of $\bPhik(\bG)$ with vertex sets $\Delta$ and $\Delta'$ are isomorphic the Dynkin diagram of $\dA_m$.
Our assumptions imply that there exists
some $w$ in the relative Weyl group $W = N_{\bG}(\bT) / Z_{\bG}(\bT) \cong \fS_n$ with $w \cdot \Delta = \Delta'$.  
The sets of subgroups
\[\Set{$\ker(w \cdot \alpha)$}{$\alpha \in \Delta$} \quad \text{and} \quad \Set{$\ker(\alpha)$}{$\alpha \in \Delta'$}\]
of $\bT$ are the same.  Letting $\tw \in N_{\bG}(\bT)(k)$ be a representative of $w$, it follows that
$\tw \bL_{\Delta_1}(k) \tw^{-1}$ equals
\[\tw Z_{\bG(k)}\left(\bigcap_{\alpha \in \Delta} \ker(\alpha)(k) \right)^0 \tw^{-1} = Z_{\bG(k)}\left(\bigcap_{\alpha \in \Delta} \ker(w \cdot \alpha)(k) \right)^0 = Z_{\bG(k)}\left(\bigcap_{\alpha \in \Delta'} \ker(\alpha)(k) \right)^0.\]
This is $\bL_{\Delta_2}(k)$,
as desired.
\end{proof}

\section{Two initial reductions}
\label{section:initialreductions}

We now begin developing the tools needed to prove Theorem \ref{maintheorem:fields}, which says that for a
reductive group $\bG$ we have $\HH_i(\bG(k);\St(\bG;\bbF)) = 0$ for $i$ in a range depending on the relative root
system $\bPhik(\bG)$. 

\subsection{Reduction to the integers}
We first show that it is enough to prove this for $\bbF = \Z$.  For a reductive group $\bG$, we will
write $\St(\bG)$ for $\St(\bG;\Z)$.

\begin{lemma}
\label{lemma:reductionlemma}
Let $\bG$ be a reductive group.  Let $b \geq 0$ be such that
$\HH_i(\bG(k);\St(\bG)) = 0$ for $i \leq b$.  Then for all commutative rings $\bbF$ we have
$\HH_i(\bG(k);\St(\bG;\bbF)) = 0$ for $i \leq b$.
\end{lemma}
\begin{proof}
Letting $F_{\bullet} \rightarrow \Z$ be a free resolution of $\bG(k)$-modules, the homology groups
of $F_{\bullet} \otimes_{\bG(k)} \St(\bG)$ are $\HH_i(\bG(k);\St(\bG))$.  Since
$\St(\bG)$ is a free abelian group equipped with an action of $\bG(k)$, each $F_i \otimes_{\bG(k)} \St(\bG)$
is also a free abelian group.  We can therefore apply the universal coefficients theorem
to
\[F_{\bullet} \otimes_{\bG(k)} \St(\bG) \otimes_{\Z} \bbF = F_{\bullet} \otimes_{\bG(k)} \St(\bG;\bbF),\]
which computes $\HH_i(\bG(k);\St(\bG;\bbF))$.  The result is a short exact sequence
\[0 \rightarrow \HH_i(\bG(k);\St(\bG)) \otimes \bbF \rightarrow \HH_i(\bG(k);\St(\bG;\bbF)) \rightarrow \Tor(\HH_{i-1}(\bG(k);\St(\bG)),\bbF) \rightarrow 0.\]
The lemma follows.
\end{proof}
 
We can therefore omit the $\bbF$ from our calculations, though at one point when handling reducible root
systems we will need more general coefficients.

\subsection{Zeroth homology}

We next prove our vanishing theorem for $\HH_0$:

\begin{lemma}
\label{lemma:h0}
Let $\bG$ be a reductive group of semisimple rank $n \geq 1$.
Then $\HH_0(\bG(k);\St(\bG)) = 0$.
\end{lemma}
\begin{proof}
The homology group $\HH_0(\bG(k);\St(\bG))$ equals the $\bG(k)$-coinvariants $\St(\bG)_{\bG(k)}$, so we must
show that these coinvariants vanish.  We divide this into two cases.

\begin{case}{1}
\label{case:h0.1}
The semisimple rank $n$ of $\bG$ is $1$, i.e., $\bPhik(\bG) \cong \dA_1$.
\end{case}

The Tits building $\Tits(\bG)$ is $(n-1)$-dimensional, and the Steinberg representation $\St(\bG)$ equals $\RH_{n-1}(\Tits(\bG))$.
Since $n=1$, we deduce that $\Tits(\bG)$ is $0$-dimensional, i.e., a discrete set.\footnote{In the case $\bG = \GL_2$, the
Tits building $\Tits(\bG)$ is the discrete set of lines in $k^2$.}  The points of $\Tits(\bG)$ are the proper parabolic subgroups of $\bG$.
The group $\bG(k)$ acts double-transitively\footnote{The vertices of $\Tits(\bG)$ are its chambers and the apartments
are Coxeter complexes of type $\dA_1$, i.e., $0$-spheres.  This double-transitivity therefore follows
from the following standard properties of the building associated to a reductive group: any two chambers lie in an apartment, the group $\bG(k)$ acts
transitively on apartments, and the $\bG(k)$-stabilizer of an apartment acts transitively on its chambers.}
on the vertices of $\Tits(\bG)$.

For $\theta \in \St(\bG)$, let $[\theta]$ be its image in the $\bG(k)$-coinvariants.
We can identify the abelian group
$\St(\bG) = \RH_0(\Tits(\bG))$ with the collection of formal sums
\[c_1 P_1 + \cdots + c_m P_m \quad \text{with $P_1,\ldots,P_m \in \Tits(\bG)^{(0)}$ and $c_1,\ldots,c_m \in \Z$ such that $\sum c_i = 0$}.\]
This is generated by elements of the form $P - P'$ with $P,P' \in \Tits(\bG)^{(0)}$ distinct, so it is enough to show that $[P - P'] =0$.
Letting $Q \in \Tits(\bG)^{(0)}$ be a third vertex,\footnote{Whether or not this exists for all buildings of type $\dA_1$ depends
on your conventions, but this always exists for such buildings if they come from reductive groups.  In this case
the building is what is called a ``thick building''; see, e.g., \cite[Remark IV.1.1]{BrownBuildings}.}
we have
\begin{equation}
\label{eqn:kill1}
[P-P'] = [P - Q] + [Q - P'].
\end{equation}
By the double transitivity of the action of $\bG(k)$ on the vertices, we can find $g_1,g_2 \in \bG(k)$ such that
\[g_1 \cdot P - g_1 \cdot P' = P - Q \quad \text{and} \quad g_2 \cdot P - g_2 \cdot P' = Q - P'.\]
It follows that in the $\bG(k)$-coinvariants we have
\begin{equation}
\label{eqn:kill2}
[P-P'] = [P - Q] = [Q - P'].
\end{equation}
Combining \eqref{eqn:kill1} and \eqref{eqn:kill2}, we conclude that $[P - P'] = 2 [P - P']$, so $[P-P'] = 0$, as desired.

\begin{case}{2}
\label{case:h0.2}
The semisimple rank of $\bG$ is greater than $1$.
\end{case}

Let $\Delta \subset \bDeltak(\bG)$ be a set consisting of a single simple root.  The standard Levi subgroup $\bL_{\Delta}$ of
the standard parabolic subgroup $\bP_{\Delta}$ thus has semisimple rank $|\Delta|=1$, i.e., $\bPhik(\bL_{\Delta}) \cong \dA_1$.  Theorem \ref{theorem:reedermap}
says that as a $\bP_{\Delta}(k)$-representation, we have
\[\St(\bG) = \Ind_{\bL_{\Delta}(k)}^{\bP_{\Delta}(k)} \St(\bL_{\Delta}).\]
It follows that
\[\St(\bG)_{\bP_{\Delta}(k)} = \St(\bL_{\Delta})_{\bL_{\Delta}},\]
which vanishes by Case \ref{case:h0.1}.  This implies that $\St(\bG)_{\bG(k)} = 0$, as desired.
\end{proof}

\section{Resolution of Steinberg and its associated spectral sequence}
\label{section:resolution}

Let $\bG$ be a reductive group of semisimple rank $n$.  In this section, we explain
how to relate $\HH_i(\bG(k);\St(\bG))$ to the
homology groups of Levi factors of parabolic subgroups of $\bG$ with coefficients in their Steinberg representations.
This relationship is expressed as a spectral sequence that will form the foundation for our
proof of Theorem \ref{maintheorem:fields}.

\subsection{Resolution}
As in \S \ref{section:standardlevi}, for each $\Delta \subset \bDeltak(\bG)$ there is a standard Levi subgroup
$\bL_{\Delta}$.  For $i \geq 0$, let $\cL_i(\bG)$ be the set of all
subsets $\Delta \subset \bDeltak(\bG)$ with $|\Delta| = n-1-i$.  For these $\Delta$, the
semisimple rank $n-1-i$ of $\bL_{\Delta}$ is smaller than that of $\bG$, so
in some sense $\bL_{\Delta}$ is a simpler group than $\bG$.  
The following (non-projective) resolution of
$\St(\bG)$ relates it to the Steinberg representations of the $\bL_{\Delta}$:

\begin{proposition}
\label{proposition:resolution}
Let $\bG$ be a reductive group of semisimple rank $n$.
We then have an exact sequence
\[0 \rightarrow \bR_{n} \rightarrow \bR_{n-1} \rightarrow \cdots \rightarrow \bR_0 \rightarrow \St(\bG) \rightarrow 0\]
of $\Z[\bG(k)]$-modules with
\[\bR_{i} \cong \bigoplus_{\Delta \in \cL_{i}(\bG)} \Ind_{\bL_{\Delta}(k)}^{\bG(k)} \St(\bL_{\Delta}) \quad \text{for $0 \leq i \leq n-1$}\]
and with $\bR_n \cong \St(\bG)^{\otimes 2}$.
\end{proposition}

\begin{example}
\label{example:resolutiongl}
The resolution of $\St(\GL_{n+1})$ from Proposition \ref{proposition:resolution} has the following concrete form.
For a finite-dimensional $k$-vector space $V$, let $\cT(V)$ be the complex
of flags of nonzero proper subspaces of $V$, so $\cT(k^{n+1})$ is the Tits building for $\GL_{n+1}$.  For $d = \dim(V)$,
let $\St(V) = \RH_{d-2}(\cT(V))$, so $\St(\GL_{n+1}) = \St(k^{n+1})$.
For $\bG = \GL_{n+1}$, as we discussed in Example \ref{example:standardlevi}
the groups $\bL_{\Delta}$ associated to $\Delta \in \cL_i(\bG)$ are those of the
form
\[\bL_{\Delta} = \GL_{n_1} \times \cdots \times \GL_{n_{i+2}} \quad \text{with $n_1,\ldots,n_{i+2} \geq 1$ and $n_1 + \cdots + n_{i+2} = n+1$}.\]
For reductive groups $\bH_1$ and $\bH_2$, we have $\St(\bH_1 \times \bH_2) \cong \St(\bH_1) \otimes \St(\bH_2)$
(cf. Lemma \ref{lemma:steinbergproduct} below).
It follows that
\[\bR_i \cong \bigoplus_{n_1 + \cdots + n_{i+2} = n+1} \Ind_{\GL_{n_1}(k) \times \cdots \times \GL_{n_{i+2}}(k)}^{\GL_n(k)} \St(k^{n_1}) \otimes \cdots \otimes \St(k^{n_{i+2}}).\]
Here and throughout this discussion the sum is over expressions $n_1 + \cdots + n_{i+2} = n+1$
with $n_1,\ldots,n_{i+2} \geq 1$.  
Fixing such an expression, the group $\GL_{n+1}(k)$ acts transitively on
decompositions $k^{n+1} = V_1 \oplus \cdots \oplus V_{i+2}$ with $\dim(V_j) = n_j$ for $1 \leq j \leq i+2$.  Moreover, the stabilizer
of the standard decomposition $k^{n+1} = k^{n_1} \oplus \cdots \oplus k^{n_{i+2}}$ is
$\GL_{n_1}(k) \times \cdots \times \GL_{n_{i+2}}(k)$.  
It follows that
\begin{align*}
\bR_i &\cong \bigoplus_{n_1 + \cdots + n_{i+2} = n+1} \Ind_{\GL_{n_1}(k) \times \cdots \times \GL_{n_{i+2}}(k)}^{\GL_n(k)} \St(k^{n_1}) \otimes \cdots \otimes \St(k^{n_{i+2}}) \\
      &\cong \bigoplus_{n_1 + \cdots + n_{i+2} = n+1} \left(\bigoplus_{\substack{k^{n+1} = V_1 \oplus \cdots \oplus V_{i+2} \\ \dim(V_j) = n_j}} \St(V_1) \otimes \cdots \otimes \St(V_{i+2})\right) \\
      &\cong \bigoplus_{V_1 \oplus \cdots \oplus V_{i+2} = k^{n+1}} \St(V_1) \otimes \cdots \otimes \St(V_{i+2}).
\end{align*}
Here the final direct sum is over such expressions with $\dim(V_j) \geq 1$ for all $1 \leq j \leq i+2$.
It follows that aside from the final term $\bR_n = \St(k^{n+1})^{\otimes 2}$, our resolution for $\St(\GL_{n+1})$
is a version of the bar resolution for the Steinberg representation introduced in \cite{MillerNagpalPatzt}.

For the special case $\bG = \GL_{n+1}$ we are considering, \cite{MillerNagpalPatzt} proves the exactness of the truncated sequence
\[\bR_{n-1} \rightarrow \cdots \rightarrow \bR_0 \rightarrow \St(k^{n+1}) \rightarrow 0\]
from Proposition \ref{proposition:resolution}.  The reference \cite{MillerNagpalPatzt} has a technical error
that was corrected in \cite{MillerNagpalPatztCorrection}.  Another proof of Proposition \ref{proposition:resolution} for $\bG = \GL_{n+1}$ 
that also includes the initial $\St(k^{n+1})^{\otimes 2}$ term can be found in \cite[Theorem 6.3]{MillerPatztWilsonRognes}.
Our proof of Proposition \ref{proposition:resolution}
is different from and more direct than either of these approaches.  In the special case $\bG = \GL_{n+1}$, our proof
was discovered independently by Charlton--Radchenko--Rudenko \cite{SteinbergPolylogarithms}.
\end{example}

\begin{remark}
For other $\bG$, it is hard to interpret the resolution from Proposition \ref{proposition:resolution} as a bar resolution.  The
issue is that the standard Levi subgroups are often different from what you expect.  For instance, consider $\bG = \Sp_{2g}$.  For
a finite-dimensional $k$-vector space $W$ equipped with a symplectic form, let $\cT_{\Sp}(W)$ be the complex of flags
of nonzero proper isotropic subspaces of $W$, so $\cT_{\Sp}(k^{2g})$ is the Tits building for $\Sp_{2g}$.  For $d = \dim(W)$, let
$\St_{\Sp}(W) = \RH_{d/2-1}(\cT_{\Sp}(W))$, so $\St_{\Sp}(k^{2g}) = \St(\Sp_{2g})$.  

An analysis similar to the one we did for $\GL_{n+1}$ in Example \ref{example:resolutiongl} shows that the resolution $\bR_{\bullet}$
of $\St(\Sp_{2g})$ from Proposition \ref{proposition:resolution} has the following description.  For $0 \leq i \leq g-1$, the term $\bR_i$
is the direct sum of terms of the following two forms:
\begin{itemize}
\item $\St(V_1) \otimes \cdots \St(V_{i+1})$, where $U = V_1 \oplus \cdots \oplus V_{i+1}$ is a Lagrangian, i.e., a maximal isotropic
subspace of $k^{2g}$.  Necessarily $\dim(U) = g$.
\item $\St(V_1) \otimes \cdots \St(V_{i+1}) \otimes \St_{Sp}(W)$, where:
\begin{itemize}
\item $W$ is a nonzero proper symplectic subspace of $k^{2g}$; and
\item $U = V_1 \oplus \cdots \oplus V_{i+1}$ is a Lagrangian in $W^{\perp}$.
\end{itemize}
\end{itemize}
Here we emphasize that the $\St(V_j)$ terms are the Steinberg representations of $\GL(V_j)$, not of symplectic groups.
\end{remark}

\begin{proof}[{Proof of Proposition \ref{proposition:resolution}}]
Since $\cT(\bG)$ is an $(n-1)$-dimensional $(n-2)$-connected complex and $\St(\bG) = \RH_{n-2}(\Tits(\bG))$, we have 
an exact sequence
\begin{equation}
\label{eqn:initialresolution}
0 \rightarrow \St(\bG) \rightarrow \CC_{n-1}(\cT(\bG)) \rightarrow \cdots \rightarrow \CC_0(\cT(\bG)) \rightarrow \Z \rightarrow 0.
\end{equation}
Here the final map $\CC_0(\cT(\bG)) \rightarrow \Z$ is the augmentation.  The term $\CC_i(\cT(\bG))$ is a free $\Z$-module with basis 
the parabolic subgroups
of $\bG$ of semisimple rank $n-1-i$, which are are permuted by the conjugation action of $\bG(k)$.
Every such parabolic subgroup of $\bG$ is conjugate to a unique
standard parabolic subgroup of the form $\bP_{\Delta}$ for some $\Delta \in \cL_i(\bG)$.  Moreover,
a parabolic subgroup $\bP$ of $\bG$ satisfies $N_{\bG}(\bP) = \bP$.  It follows that 
\[\CC_{i}(\cT(\bG)) \cong \bigoplus_{\Delta \in \cL_i(\bG)} \Z[\bG(k)/\bP_{\Delta}(k)] \quad \text{for $0 \leq i \leq n-1$}.\]
Now tensor \eqref{eqn:initialresolution} with $\St(\bG)$:
\[0 \rightarrow \St(\bG)^{\otimes 2} \rightarrow \CC_{n-1}(\cT(\bG)) \otimes \St(\bG) \rightarrow 
\cdots \rightarrow \CC_0(\cT(\bG)) \otimes \St(\bG) \rightarrow \St(\bG) \rightarrow 0.\]
Since $\St(\bG)$ is a free $\Z$-module, this chain complex is also exact.  By the above, we have
\[\CC_{i}(\cT(\bG)) \otimes \St(\bG) \cong \bigoplus_{\Delta \in \cL_i(\bG)} \Z[\bG(k)/\bP_{\Delta}(k)] \otimes \St(\bG) \quad \text{for $0 \leq i \leq n-1$}.\]
For $\Delta \in \cL_i(\bG)$, we have
\[\Z[\bG(k)/\bP_{\Delta}(k)] \otimes \St(\bG) \cong \Ind_{\bP_{\Delta}(k)}^{\bG(k)} \Res^{\bG(k)}_{\bP_{\Delta}(k)} \St(\bG).\]
By Theorem \ref{theorem:reedermap},
\[\Res^{\bG(k)}_{\bP_{\Delta}(k)} \St(\bG) \cong \Ind_{\bL_{\Delta}(k)}^{\bP_{\Delta}(k)} \St(\bL_{\Delta}).\]
Combining all of this, we see that the above exact sequence has all the properties claimed in the proposition.
\end{proof}

\subsection{Spectral sequence}

The following spectral sequence will play a key role in our proof:

\begin{corollary}
\label{corollary:spectralsequence}
Let $\bG$ be a reductive group of semisimple rank $n$.
There is then a spectral sequence $\ssE^r_{pq}$ converging to $\HH_{p+q}(\bG(k);\St(\bG))$ with
\[\ssE^1_{pq} \cong \begin{cases}
\bigoplus_{\Delta \in \cL_p(\bG)} \HH_q(\bL_{\Delta}(k);\St(\bL_{\Delta})) & \text{if $0 \leq p \leq n-1$} \\
\HH_q(\bG(k);\St(\bG)^{\otimes 2})                                                                & \text{if $p = n$},\\
0                                                                                                      & \text{otherwise}.
\end{cases}\]
\end{corollary}
\begin{proof}
Let $\bR_{\bullet} \rightarrow \St(\bG)$ be the resolution of $\Z[\bG(k)]$-modules given by
Proposition \ref{proposition:resolution}.  As is standard,\footnote{Note that if the $\bR_q$ were all projective, this would
reduce to the usual recipe for computing $\HH_{\bullet}(\bG(k);\St(\bG))$ from a projective resolution.} there is a spectral sequence of the form
\[\ssE^1_{pq} \cong \HH_q(\bG(k);\bR_p) \Rightarrow \HH_{p+q}(\bG(k);\St(\bG)).\]
We can apply the description of $\bR_p$ from Proposition \ref{proposition:resolution} to see that
\[\ssE^1_{pq} \cong \begin{cases}
\bigoplus_{\Delta \in \cL_p(\bG)} \HH_q(\bG(k);\Ind_{\bL_{\Delta}(k)}^{\bG(k)} \St(\bL_{\Delta})) & \text{if $0 \leq p \leq n-1$} \\
\HH_q(\bG(k);\St(\bG)^{\otimes 2})                                                                & \text{if $p = n$},\\
0                                                                                                      & \text{otherwise}.
\end{cases}\]
For $\Delta \in \cL_p(\bG)$, Shapiro's lemma implies that
\[\HH_q(\bG(k);\Ind_{\bL_{\Delta}(k)}^{\bG(k)} \St(\bL_{\Delta})) \cong \HH_q(\bL_{\Delta}(k);\St(\bL_{\Delta})).\]
The corollary follows.
\end{proof}

\section{Semisimple rank two}
\label{section:rank2}

As a first application of the spectral sequence from Corollary \ref{corollary:spectralsequence}, we prove
the following lemma.  Later we will only use it when $\bPhik(\bG) = \dA_2$, but making that assumption
would not simplify the proof.

\begin{lemma}
\label{lemma:rank2}
Let $\bG$ be a reductive group of semisimple rank $n=2$.
Then the map
\begin{equation}
\label{eqn:toprovecaseii}
\bigoplus_{\Delta \in \cL_0(\bG)} \HH_1(\bL_{\Delta}(k);\St(\bL_{\Delta})) \rightarrow \HH_1(\bG(k);\St(\bG))
\end{equation}
is surjective.
\end{lemma}
\begin{proof}
Recall that for $p \geq 0$ the set $\cL_p(\bG)$ is the collection of subsets $\Delta \subset \bDeltak(\bG)$ of
simple roots with $|\Delta| = n-1-p = 1-p$.
Corollary \ref{corollary:spectralsequence} gives a spectral sequence converging to $\HH_{p+q}(\bG(k);\St(\bG))$ with
\[\ssE^1_{pq} \cong \begin{cases}
\bigoplus_{\Delta \in \cL_p(\bG)} \HH_q(\bL_{\Delta}(k);\St(\bL_{\Delta})) & \text{if $0 \leq p \leq 1$} \\
\HH_q(\bG(k);\St(\bG)^{\otimes 2})                                         & \text{if $p = 2$},\\
0                                                                          & \text{otherwise}.
\end{cases}\]
In particular,
\[\ssE^1_{01} = \bigoplus_{\Delta \in \cL_0(\bG)} \HH_1(\bL_{\Delta}(k);\St(\bL_{\Delta})).\]
Our goal is to prove that this surjects onto $\HH_1(\bG(k);\St(\bG))$.
The part of the $\ssE^1$-page of the spectral sequence needed to compute $\HH_1(\bG(k);\St(\bG))$ is
\begin{center}
\begin{tblr}{|ccccc}
$\ssE^1_{01}$ & $\leftarrow$  & $\ssE^1_{11}$ &                                   &            \\
$\ssE^1_{00}$        & $\leftarrow$  & $\ssE^1_{10}$ & $\stackrel{\partial}{\leftarrow}$ & $\ssE^1_{20}$ \\
\hline
\end{tblr}
\end{center}
To prove that $\ssE^1_{01}$ surjects onto $\HH_1(\bG(k);\St(\bG))$, it is enough to prove that
the labeled differential $\partial$
is surjective.\footnote{Note that Lemma \ref{lemma:h0} implies that $\ssE^1_{00} = 0$, so this surjectivity
is the only way to kill the $\ssE^1_{10}$-term.}  We divide this into four steps.

\begin{step}{1}
\label{step:basecase2.1}
We give a concrete description of $\partial\colon \ssE^1_{20} \rightarrow \ssE^1_{10}$.
\end{step}

The spectral sequence in Corollary \ref{corollary:spectralsequence} is constructed using the resolution of $\St(\bG)$ from
Proposition \ref{proposition:resolution}.  Ignoring the identifications made in that proposition, that
resolution is obtained in two steps:
\begin{itemize}
\item Start with the augmented chain complex for the Tits building $\Tits(\bG)$.  In our case, since
the semisimple rank of $\bG$ is $2$ the complex $\Tits(\bG)$ is $1$-dimensional, so this chain complex
is
\[\CC_1(\Tits(\bG)) \rightarrow \CC_0(\Tits(\bG)) \rightarrow \Z \rightarrow 0.\]
Just like in \S \ref{section:structurebuilding}, we identify $\CC_1(\Tits(\bG))$ with $\Z[\fP_{\min}]$, where
$\fP_{\min}$ is the set of minimal parabolic subgroups of $\bG$.  The kernel of $\CC_1(\Tits(\bG)) \rightarrow \CC_0(\Tits(\bG))$
is $\St(\bG)$, so we can extend the above to an exact sequence
\[0 \rightarrow \St(\bG) \stackrel{\iota}{\rightarrow} \Z[\fP_{\min}] \rightarrow \CC_0(\Tits(\bG)) \rightarrow \Z \rightarrow 0.\]
\item Tensor this with $\St(\bG)$ to get a resolution
\[0 \longrightarrow \St(\bG)^{\otimes 2} \stackrel{\iota \otimes \text{id}}{\longrightarrow} \Z[\fP_{\min}] \otimes \St(\bG) \longrightarrow \CC_0(\Tits(\bG)) \otimes \St(\bG) \longrightarrow \St(\bG) \rightarrow 0.\]
\end{itemize}
We then have
\begin{align*}
\ssE^1_{20} &= \HH_0(\bG(k);\St(\bG)^{\otimes 2}) = \left(\St(\bG)^{\otimes 2}\right)_{\bG(k)},\\
\ssE^1_{10} &= \HH_0(\bG(k);\Z[\fP_{\min}] \otimes \St(\bG)) = \left(\Z[\fP_{\min}] \otimes \St(\bG)\right)_{\bG(k)},
\end{align*}
and the differential $\partial\colon \ssE^1_{20} \rightarrow \ssE^1_{10}$ is exactly the map on these coinvariants induced by the
map $\iota \otimes \text{id} \colon \St(\bG)^{\otimes 2} \rightarrow \Z[\fP_{\min}] \otimes \St(\bG)$.

\begin{step}{2}
\label{step:basecase2.2}
We construct an element $\theta \in \ssE^1_{20}$ and give a concrete description of $\partial(\theta) \in \ssE^1_{10}$.
\end{step}

As in \S \ref{section:basicgroups}, let $\bB$ be a minimal parabolic subgroup of $\bG$, let $\bT$ be a maximal split torus
in $\bB$, and let $\bU$ be the unipotent radical of $\bB$.  We then have the relative Weyl group $W = N_{\bG}(\bT) / Z_{\bG}(\bT)$.  For each $w \in W$, pick a
representative $\tw \in N_{\bG}(\bT)(k)$.  Apply Corollary \ref{corollary:findu} to find some $u \in \bU(k)$ with the following property:
\begin{equation}
\label{eqn:uprop}
\text{the only $w_1,w_2 \in W$ with $\tw_1 u \tw_2 \in \bB(k)$ is $w_1 = w_2 = 1$}.
\end{equation}
Recall from \S \ref{section:apartments} that for $g \in \bG(k)$ we have the apartment class $\bbA_g \in \St(\bG)$, which
is defined via the formula
\[\bbA_{g} = \sum_{w \in W} (-1)^w g \tw \cdot \bB \in \St(\bG) \subset \Z[\fP_{\min}].\]
The element
\[\theta \in \ssE^1_{20} = \left(\St(\bG)^{\otimes 2}\right)_{\bG(k)}\]
we will use is the image in the coinvariants of
\[\bbA_{\text{id}} \otimes \bbA_u \in \St(\bG)^{\otimes 2}.\]
For $x \in \Z[\fP_{\min}] \otimes \St(\bG)$, denote by
\[[x] \in \ssE^1_{10} = \left(\Z[\fP_{\min}] \otimes \St(\bG)\right)_{\bG(k)}\]
its image in the coinvariants.  The image $\partial(\theta) \in \ssE^1_{10}$ is
$[\iota(\bbA_{\text{id}}) \otimes \bbA_u]$.  We expand this out as follows:
\begin{align*}
[\iota(\bbA_{\text{id}}) \otimes \bbA_u] &= [\left(\sum_{w \in W} (-1)^w \tw \cdot \bB\right) \otimes \bbA_u]
                                          = [\sum_{w \in W} (-1)^w \tw \left(\bB \otimes \tw^{-1} \cdot \bbA_u\right)] \\
                                         &= [\sum_{w \in W} (-1)^w \left(\bB \otimes \tw^{-1} \cdot \bbA_u\right)]
                                          = [\bB \otimes \sum_{w \in W} (-1)^w \tw \cdot \bbA_u].
\end{align*}
Here the third $=$ comes from the fact that we are working in the $\bG(k)$-coinvariants and the fourth from the fact that $(-1)^w = (-1)^{w^{-1}}$.

\begin{step}{3}
\label{step:basecase2.3}
We prove that $\ssE^1_{10} \cong \Z$ and then give a way to calculate the image of certain elements of
$\ssE^1_{10}$ under the isomorphism $\ssE^1_{10} \cong \Z$.
\end{step}

By definition, the minimal parabolic $\bB$ equals the standard parabolic
$\bP_{\emptyset}$ associated to $\emptyset \subset \bDeltak(\bG)$.  Every
parabolic subgroup is conjugate to a unique standard parabolic subgroup,
so every minimal parabolic subgroup is conjugate to $\bB$.  Moreover,
as is the case for every parabolic subgroup we have $N_{\bG}(\bB) = \bB$.  It follows that
$\Z[\fP_{\min}] \cong \Z[\bG(k)/\bB(k)]$.  This implies that
\begin{equation}
\label{eqn:e1id.1}
\ssE^1_{10} = \left(\Z[\fP_{\min}] \otimes \St(\bG)\right)_{\bG(k)} \cong \left(\Z[\bG(k)/\bB(k)] \otimes \St(\bG)\right)_{\bG(k)} \cong \St(\bG)_{\bB(k)}.
\end{equation}
Recall from \S \ref{section:steinbergbasis} that $\St(\bG) \cong \Z[\bU(k)]$.  Since $\bB(k)$ acts transitively
on $\bU(k)$, we conclude that
\begin{equation}
\label{eqn:e1id.2}
\ssE^1_{10} \cong \Z[\bU(k)]_{\bB(k)} \cong \Z.
\end{equation}
This isomorphism is induced by the augmentation $\epsilon\colon \Z[\bU(k)] \rightarrow \Z$.
For $\theta \in \Z[\fP_{\min}] \otimes \St(\bG)$, write $[\theta]$ for the image of $\theta$ in $\ssE^1_{10} \cong \Z$ under the
identifications \eqref{eqn:e1id.1} and \eqref{eqn:e1id.2}.

Consider $x \in \St(\bG)$.  We claim that $[\bB \otimes x] \in \Z$ is the
$\bB$-coordinate of $x$ when we express $x$ as an element of $\Z[\fP_{\min}]$.  
As we discussed in \S \ref{section:steinbergbasis}, the abelian group $\St(\bG)$ is
generated by apartment classes $\bbA_u$ with $u \in \bU(k)$.  By linearity, it
is therefore enough to prove the claim for $x = \bbA_u$ with $u \in \bU(k)$.
Examining \eqref{eqn:e1id.1} and \eqref{eqn:e1id.2}, what we must prove is that the
image $\theta \in \Z[\bU(k)]$ of $x = \bbA_u$ has $\epsilon(\theta)$ equal
to the $\bB$-coordinate of $x = \bbA_u$.  To see this, note that:
\begin{itemize}
\item $x = \bbA_u$ goes to $\theta = u$ in $\Z[\bU(k)]$, so $\epsilon(\theta) = \epsilon(u) = 1$; and
\item as was noted in \S \ref{section:steinbergbasis}, the $\bB$-coordinate
of $x = \bbA_u$ is also $1$.
\end{itemize}

\begin{step}{4}
\label{step:basecase2.4}
For the element $\theta$ from Step \ref{step:basecase2.2},
we prove that $\partial(\theta)$ generates $\ssE^1_{10} \cong \Z$.  This will imply that $\partial$ is surjective, completing the proof.
\end{step}

Combining the isomorphism $\ssE^1_{10} \cong \Z$ from Step \ref{step:basecase2.3} with what we did in Step \ref{step:basecase2.2},
we see that what we must prove
is that the $\bB$-coordinate of
\[x = \sum_{w \in W} (-1)^w \tw \cdot \bbA_u \in \St(\bG) \subset \Z[\fP_{\min}]\]
is $1$.  Expanding out $\bbA_u$, we see that
\[x = \sum_{w_1 \in W} (-1)^{w_1} \tw_1 \cdot \sum_{w_2 \in W} (-1)^{w_2} u \tw_2 \cdot \bB = \sum_{w_1,w_2 \in W} (-1)^{w_1 w_2} \tw_1 u \tw_2 \cdot \bB.\]
The only elements $g \in \bG(k)$ with $g \cdot \bB = \bB$ are $g \in \bB(k)$, and
the defining feature of
$u$ from \eqref{eqn:uprop} in Step \ref{step:basecase2.2} is that the only $w_1,w_2 \in W$ with $\tw_1 u \tw_2 \in \bB(k)$ are $w_1=w_2 = 1$.  The
fact that the $\bB$-coordinate of $x$ is $1$ follows.
\end{proof}

\section{Reducible root systems}
\label{section:reducible}

Let $\bG$ be a reductive group.  Recall that our goal is to prove Theorem \ref{maintheorem:fields},
which says that $\HH_i(\bG(k);\St(\bG))=0$ for $i$ in a range depending on the relative root
system $\bPhik(\bG)$.  A root system is irreducible when it cannot be decomposed as a direct product,
or equivalently if its Dynkin diagram is connected.  Even if you only care about $\bG$ with $\bPhik(\bG)$ irreducible,
reducible root systems will arise when you study Levi factors of parabolic subgroups.  This section
gives tools for handling reducible root systems.

\subsection{Semisimplicity}

We start with a preliminary reduction.  For a reductive group $\bG$, its
derived subgroup $\bGder$ is\footnote{Just like the definition of a reductive group, the
definition of a semisimple group is not particularly enlightening and we refer
the reader to standard sources for it, e.g., \cite{BorelBook, MilneBook}.  The key example
to keep in mind is that $\GL_{n+1}$ is reductive and its derived subgroup $\SL_{n+1}$
is semisimple.} semisimple \cite[Proposition 19.21]{MilneBook}.  The inclusion $\bGder \hookrightarrow \bG$ induces
a bijection between parabolic subgroups, and thus isomorphisms $\Tits(\bG) \cong \Tits(\bGder)$ and
$\Res^{\bG(k)}_{\bGder(k)} \St(\bG) \cong \St(\bGder)$.
The following lemma will allow us to reduce many questions
to the semisimple case:

\begin{lemma}
\label{lemma:reducetosemisimple}
Let $\bG$ be a reductive group and $b \geq -1$.  Assume that
$\HH_i(\bGder(k);\St(\bGder)) = 0$ for $i \leq b$.  Then $\HH_i(\bG(k);\St(\bG)) = 0$ for $i \leq b$ and the map
\[\HH_{b+1}(\bGder(k);\St(\bGder)) \longrightarrow \HH_{b+1}(\bG(k);\St(\bG))\]
is a surjection.  
\end{lemma}
\begin{proof}
We have a short exact sequence
\[1 \longrightarrow \bGder(k) \longrightarrow \bG(k) \longrightarrow A \longrightarrow 1\]
with $A$ abelian.  The associated Hochschild--Serre
spectral sequence with coefficients in $\St(\bG)$ is of the form
\[\ssE^2_{pq} = \HH_p(A;\HH_q(\bGder(k);\St(\bGder))) \Rightarrow \HH_{p+q}(\bG(k);\St(\bG)).\]
Our vanishing assumption implies that $\ssE^2_{pq} = 0$ for $q \leq b$, so
$\HH_i(\bG(k);\St(\bG)) = 0$ for $i \leq b$ and
$\HH_{b+1}(\bG(k);\St(\bG)) = \ssE^2_{0,b+1}$.  We deduce that
\[\HH_{b+1}(\bG(k);\St(\bG)) = \HH_0(A;\HH_{b+1}(\bGder(k);\St(\bGder))) = \HH_{b+1}(\bGder(k);\St(\bGder))_A,\]
where the subscript indicates that we are taking coinvariants.  For any group $\Gamma$ and any
$\Gamma$-module $M$, the map $M \rightarrow M_{\Gamma}$ is a surjection.  In light of the above
identity, we deduce that the map $\HH_{b+1}(\bGder(k);\St(\bGder)) \rightarrow \HH_{b+1}(\bG(k);\St(\bG))$
is a surjection, as desired.  
\end{proof}

\subsection{Products}

For a semisimple group $\bG$ with $\bPhik(\bG)$ reducible, the following lemma will allow us to reduce questions about
$\HH_i(\bG;\St(\bG))$ to questions about a product of semisimple groups with irreducible relative root systems:

\begin{lemma}
\label{lemma:reducetoproduct}
Let $\bG$ be a semisimple group.  Write
\[\bPhik(\bG) = \bPhi_1 \times \cdots \times \bPhi_m\]
with each $\bPhi_i$ irreducible.  There then exists a semisimple group $\bH$ equipped with
a surjection $\bH \rightarrow \bG$ such that:
\begin{itemize}
\item[(i)] the surjection $\bH \rightarrow \bG$ induces a bijection between parabolic subgroups, and thus
isomorphisms $\Tits(\bH) \cong \Tits(\bG)$ and $\St(\bH) \cong \Res^{\bG(k)}_{\bH(k)} \St(\bG)$; and
\item[(ii)] $\bH = \bH_1 \times \cdots \times \bH_m$ with $\bPhik(\bH_j) = \bPhi_j$ for $1 \leq j \leq m$.
\end{itemize}
If in addition for some $b \geq -1$ we have $\HH_i(\bG(k);\St(\bG))=0$ for $i \leq b$, then:
\begin{itemize}
\item[(iii)] the map $\HH_{b+1}(\bH(k);\St(\bH)) \rightarrow \HH_{b+1}(\bG(k);\St(\bG))$ is an isomorphism.
\end{itemize}
\end{lemma}
\begin{proof}
Let $\bG_1,\ldots,\bG_r$ be the minimal elements among the connected normal subgroups of $\bG$ that have strictly
positive dimension.  By \cite[Theorem 22.10]{BorelBook}, the following hold:
\begin{itemize}
\item Letting $\bH = \bG_1 \times \cdots \times \bG_r$, the map $\bH \rightarrow \bG$ is surjective with finite central kernel.
\item Exactly $m \leq r$ of the $\bG_j$ have positive semisimple rank.  Moreover, we can order
the $\bG_j$ such that $\bPhik(\bG_j) = \bPhi_j$ for $1 \leq j \leq m$.
\end{itemize}
That the map $\bH \rightarrow \bG$ satisfies (i) is immediate, and for the decomposition $\bH = \bH_1 \times \cdots \times \bH_m$ from (ii)
we can take
\[\bH_j = \begin{cases}
\bG_j & \text{if $1 \leq j \leq m-1$},\\
\bG_m \times \bG_{m+1} \times \cdots \times\bG_r & \text{if $j = m$}.
\end{cases}\]
Now assume that for some $b \geq -1$ we have $\HH_i(\bG(k);\St(\bG))=0$ for $i \leq b$.  We must prove (iii).
By construction, we have a central extension
\[1 \longrightarrow C \longrightarrow \bH(k) \longrightarrow \bG(k) \longrightarrow 1.\]
Since the map $\bH \rightarrow \bG$ induces an isomorphism $\St(\bH) \cong \Res^{\bG(k)}_{\bH(k)} \St(\bG)$, the group
$C$ acts trivially on $\St(\bH) \cong \Res^{\bG(k)}_{\bH(k)} \St(\bG)$.  Consider the
associated Hochschild--Serre spectral sequence with coefficients in $\St(\bG)$.
Since $\St(\bG)$ is a free abelian group on which $C$ acts trivially, this spectral sequence
has
\begin{align*}
\ssE^2_{pq} &= \HH_p(\bG(k);\HH_q(C;\St(\bG))) \\
            &= \HH_p(\bG(k);\HH_q(C) \otimes \St(\bG)) \\
            &= \HH_p(\bG(k);\St(\bG;\HH_q(C))).
\end{align*}
By our vanishing assumption and Lemma \ref{lemma:reductionlemma}, we have $\ssE^2_{pq} = 0$ for $p \leq b$.  We conclude
that
\[\HH_{b+1}(\bH(k);\St(\bH)) = E^2_{b+1,0} = \HH_{b+1}(\bG(k);\St(\bG;\HH_0(C))) = \HH_{b+1}(\bG(k);\St(\bG)),\]
as desired.
\end{proof}

To use Lemma \ref{lemma:reducetoproduct}, we need to know the Steinberg representation of a product:

\begin{lemma}
\label{lemma:steinbergproduct}
Let $\bG_1$ and $\bG_2$ be reductive groups and let $\bbF$ be a commutative ring.  Then
$\St(\bG_1 \times \bG_2;\bbF) \cong \St(\bG_1;\bbF) \otimes \St(\bG_2;\bbF)$.
\end{lemma}
\begin{proof}
We have $\Tits(\bG_1 \times \bG_2) = \Tits(\bG_1) \ast \Tits(\bG_2)$, where $\ast$ means join.
If $n_i$ is the semisimple rank of $\bG_i$, then the Solomon--Tits theorem says that $\Tits(\bG_i)$
is $n_i-2$ connected and $\St(\bG_i;\bbF) = \RH_{n_i-1}(\Tits(\bG_i);\bbF)$.  The
semisimple rank of $\bG_1 \times \bG_2$ is $n_1+n_2$, so 
\begin{align*}
\St(\bG_1 \times \bG_2;\bbF) &= \RH_{(n_1+n_2)-1}(\Tits(\bG_1) \ast \Tits(\bG_2);\bbF) \\
                         &= \RH_{n_1-1}(\Tits(\bG_1);\bbF) \otimes \RH_{n_2-1}(\Tits(\bG_2);\bbF) = \St(\bG_1;\bbF) \otimes \St(\bG_2;\bbF).\qedhere
\end{align*}
\end{proof}

\subsection{Reducible root systems and vanishing}

Recall that our goal is to prove that for a reductive group $\bG$ we have
$\HH_i(\bG(k);\St(\bG))=0$ for $i$ in some range.  The following result
will allow us to reduce this to $\bG$ with $\bPhik(\bG)$ irreducible:

\begin{lemma}[Reducible vanishing]
\label{lemma:reduciblevanishing}
Let $\bPhi_1,\ldots,\bPhi_m$ be nontrivial irreducible root systems.  For each
$1 \leq j \leq m$, assume that there is some $b_j \geq -1$ such that the following holds:\noeqref{vanishing1}
\begin{equation}
\tag{$\varheartsuit$}\label{vanishing1}
\parbox{35em}{If $\bH_j$ is a reductive group with $\bPhik(\bH_j) = \bPhi_j$, then $\HH_i(\bH_j(k);\St(\bH_j)) = 0$ for $i \leq b_j$.}
\end{equation}
Then for all reductive groups $\bG$ with
$\bPhik(\bG) = \bPhi_1 \times \cdots \times \bPhi_m$,
we have $\HH_i(\bG(k);\St(\bG)) = 0$ for $i \leq (m-1) + b_1 + \cdots + b_m$.
\end{lemma}
\begin{proof}
By Lemma \ref{lemma:reducetosemisimple}, it is enough to prove this for $\bG$ semisimple.
The proof will be by induction on $m$.  The base case $m=1$ follows immediately from
\eqref{vanishing1}, so assume that $m \geq 2$ and that the lemma is true whenever $m$ is smaller.
For all $b \leq (m-1) + b_1 + \cdots + b_m$, we will prove that:\noeqref{spadeb}
\begin{equation*}
\tag{$\mathbf{\spadesuit}_{\Bind}$}\label{spadeb}
\parbox{35em}{\raggedright For all semisimple groups $\bG$ with $\bPhik(\bG) = \bPhi_1 \times \cdots \times \bPhi_m$,
we have $\HH_i(\bG(k);\St(\bG))=0$ for $i \leq b$.}
\end{equation*}
The proof will be by induction on $b$.  The base case $b=-1$ is trivial, so
assume that $-1 \leq b < (m-1) + b_1 + \cdots + b_m$ and that {\renewcommand\Bind{b}\eqref{spadeb}}
holds.  We will prove that {\renewcommand\Bind{b+1}\eqref{spadeb}} holds.  Let $\bG$ be a semisimple
group with $\bPhik(\bG) = \bPhi_1 \times \cdots \times \bPhi_m$.  Since {\renewcommand\Bind{b}\eqref{spadeb}} holds,
we have $\HH_i(\bG(k);\St(\bG)) = 0$ for $i \leq b$.  To prove {\renewcommand\Bind{b+1}\eqref{spadeb}},
we must prove that $\HH_{b+1}(\bG(k);\St(\bG)) = 0$.
Applying Lemma \ref{lemma:reducetoproduct},
we see that there exists a semisimple group $\bH$ equipped with
a surjection $\bH \rightarrow \bG$ such that:\footnote{We get (iii) because of {\renewcommand\Bind{b}\eqref{spadeb}}.}
\begin{itemize}
\item[(i)] the surjection $\bH \rightarrow \bG$ induces a bijection between parabolic subgroups, and thus
isomorphisms $\Tits(\bH) \cong \Tits(\bG)$ and $\St(\bH) \cong \Res^{\bG(k)}_{\bH(k)} \St(\bG)$; and
\item[(ii)] $\bH = \bH_1 \times \cdots \times \bH_m$ with $\bPhik(\bH_j) = \bPhi_j$ for $1 \leq j \leq m$; and
\item[(iii)] the map $\HH_{b+1}(\bH(k);\St(\bH)) \rightarrow \HH_{b+1}(\bG(k);\St(\bG))$ is an isomorphism.
\end{itemize}
By (iii), to prove that $\HH_{b+1}(\bG(k);\St(\bG)) = 0$ it is enough to prove
that $\HH_{b+1}(\bH(k);\St(\bH)) = 0$.

As notation, set $\bA = \bH_1 \times \cdots \times \bH_{m-1}$ and
$\bB = \bH_m$, so $\bH = \bA \times \bB$.  By Lemma \ref{lemma:steinbergproduct}, we have
$\St(\bA \times \bB) = \St(\bA) \otimes \St(\bB)$.
Since the abelian groups underlying $\St(\bA)$ and $\St(\bB)$ are free abelian, the K\"{u}nneth formula applies and
shows that $\HH_i(\bA(k) \times \bB(k);\St(\bA \times \bB))$ fits
into a short exact sequence with the following kernel and cokernel:
\begin{itemize}
\item $\bigoplus_{i_1 + i_2 = i} \HH_{i_1}(\bA(k);\St(\bA)) \otimes \HH_{i_2}(\bB(k);\St(\bB))$.
\item $\bigoplus_{i_1 + i_2 = i-1} \Tor(\HH_{i_1}(\bA(k);\St(\bA)),\HH_{i_2}(\bB(k);\St(\bB)))$.
\end{itemize}
Recall that we are inducting on $m$.  By this induction hypothesis and \eqref{vanishing1}, we have 
\begin{itemize}
\item $\HH_{i_1}(\bA(k);\St(\bA)) = 0$ for $i_1 \leq (m-2) + b_1 + \cdots + b_{m-1}$; and
\item $\HH_{i_2}(\bB(k);\St(\bB)) = 0$ for $i_2 \leq b_m$.
\end{itemize}
Note that if
\[i_1 + i_2 \leq ((m-2) + b_1 + \cdots + b_{m-1}) + b_m + 1 = (m-1) + b_1 + \cdots + b_m,\]
then either
\[i_1 \leq (m-2) + b_1 + \cdots + b_{m-1} \quad \text{or} \quad i_2 \leq b_m.\]
It follows that $\HH_i(\bA(k) \times \bB(k);\St(\bA \times \bB)) = 0$
for $i \leq (m-1) + b_1 + \cdots + b_m$.  In particular,\footnote{We are inducting on $b$
so we can reduce to a group that splits as a product of groups with
relative root systems $\bPhi_j$.  If our group $\bG$ was already a product, then there would be no need to induct on $b$.}
this holds for $i = b+1$, as desired.
\end{proof}

\subsection{Reducible root systems and surjectivity}

Our proof that for some $b$ we have $\HH_i(\bG(k);\St(\bG)) = 0$ for $i \leq b$ will be by induction, and
to make the induction work we will have to also prove a surjectivity statement
for $\HH_{b+1}$.  The following will allow us to reduce this surjectivity
statement to $\bG$ with $\bPhik(\bG)$ irreducible:

\begin{lemma}[Reducible surjectivity]
\label{lemma:reduciblesurjectivity}
Let $\bPhi_1,\ldots,\bPhi_m$ be nontrivial irreducible root systems.  For each
$1 \leq j \leq m$, assume that there is some $b_j \geq -1$ such that the following holds:\noeqref{surjectivity1}
\begin{equation} 
\tag{$\varheartsuit$}\label{surjectivity1}
\parbox{35em}{If $\bH_j$ is a reductive group with $\bPhik(\bH_j) = \bPhi_j$, then $\HH_i(\bH_j(k);\St(\bH_j)) = 0$ for $i \leq b_j$.}
\end{equation}
Additionally, for some $1 \leq j_0 \leq m$ assume there exists a subset $\Delta'$ of the set of simple roots
of $\bPhi_{j_0}$ such that the following holds:\noeqref{surjectivity2}
\begin{equation} 
\tag{$\varheartsuit\varheartsuit$}\label{surjectivity2}
\parbox{35em}{If $\bH_{j_0}$ is a reductive group with $\bPhik(\bH_{j_0}) = \bPhi_{j_0}$ and $\bL^{\bH_{j_0}}_{\Delta'}$ is the corresponding
standard Levi subgroup of $\bH$, then the map 
$\HH_{b_{j_0}+1}(\bL^{\bH_{j_0}}_{\Delta'}(k);\St(\bL^{\bH_{j_0}}_{\Delta'})) \rightarrow \HH_{b_{j_0}+1}(\bH_{j_0}(k);\St(\bH_{j_0}))$
is surjective.}
\end{equation}
Let $\bG$ be a reductive group with
$\bPhik(\bG) = \bPhi_1 \times \cdots \times \bPhi_m$.
Let 
\[\oDelta' = \Delta' \sqcup \bigsqcup_{\substack{1 \leq j \leq m \\ j \neq j_0}} \bDeltak(\bPhi_j),\]
and let $\bL^{\bG}_{\oDelta'}$ be the corresponding
standard Levi subgroup of $\bG$.  Then for
\[b = (m-1) + b_1 + \cdots + b_m\]
the map $\HH_{b+1}(\bL^{\bG}_{\oDelta'}(k);\St(\bL^{\bG}_{\Delta'})) \rightarrow \HH_{b+1}(\bG(k);\St(\bG))$
is surjective.
\end{lemma}
\begin{proof}
Lemma \ref{lemma:reduciblevanishing} (reducible vanishing) implies that $\HH_i(\bGder(k);\St(\bGder)) = 0$ for $i \leq b$.  We can therefore
apply Lemma \ref{lemma:reducetosemisimple} and see that the map
\begin{equation}
\label{eqn:derivedsurjection}
\HH_{b+1}(\bGder;\St(\bGder)) \rightarrow \HH_{b+1}(\bG;\St(\bG))
\end{equation}
is surjective.  We can also apply Lemma \ref{lemma:reducetoproduct} to find a semisimple group
$\bH$ equipped with a surjection $\bH \rightarrow \bGder$ such that:
\begin{itemize}
\item[(i)] the surjection $\bH \rightarrow \bGder$ induces a bijection between parabolic subgroups, and thus
isomorphisms $\Tits(\bH) \cong \Tits(\bGder)$ and $\St(\bH) \cong \Res^{\bGder(k)}_{\bH(k)} \St(\bGder)$; and
\item[(ii)] $\bH = \bH_1 \times \cdots \times \bH_m$ with $\bPhik(\bH_j) = \bPhi_j$ for all $1 \leq j \leq m$; and
\item[(iii)] the map $\HH_{b+1}(\bH(k);\St(\bH)) \rightarrow \HH_{b+1}(\bGder(k);\St(\bGder))$ is an isomorphism.
\end{itemize}
In particular, combining (iii) and the fact that \eqref{eqn:derivedsurjection} is surjective, we see that
the map
\[\HH_{b+1}(\bH_1(k) \times \cdots \times \bH_m(k);\St(\bH_1 \times \cdots \times \bH_m)) \rightarrow \HH_{b+1}(\bG;\St(\bG))\]
is surjective.

Let $\bL_{\oDelta'}^{\bH}$ and $\bL_{\Delta'}^{\bH_{j_0}}$ be the standard Levi subgroups of $\bH$ and $\bH_{j_0}$ corresponding to
the simple roots $\oDelta'$ and $\Delta'$.  We therefore have
\[\bL_{\oDelta'}^{\bH} = \bH_1 \times \cdots \times \bL_{\Delta'}^{\bH_{j_0}} \times \cdots \times \bH_m.\]
Examining the proof of Lemma \ref{lemma:reducetoproduct}, we see that $\bL_{\oDelta'}^{\bH}$ is precisely
is precisely the group obtained by applying Lemma \ref{lemma:reducetoproduct} to the
derived subgroup of $\bL^{\bG}_{\oDelta'}$, though in this case we do not have conclusion (iii) of Lemma \ref{lemma:reducetoproduct}.  
We therefore have a commutative diagram
\[\begin{tikzcd}[column sep=small]
\HH_{b+1}(\bL_{\oDelta'}^{\bH}(k);\St(\bL_{\oDelta'}^{\bH})) \arrow{r} \arrow{d}{\tf} & \HH_{b+1}(\bL^{\bG}_{\oDelta'}(k);\St(\bL^{\bG}_{\oDelta'})) \arrow{d}{f} \\
\HH_{b+1}(\bH(k);\St(\bH))                                 \arrow[two heads]{r}     & \HH_{b+1}(\bG(k);\St(\bG))
\end{tikzcd}\]
From this, we see that to prove that $f$ is surjective, it is enough to prove that $\tf$ is surjective.

In light of \eqref{surjectivity1}, we can use the K\"{u}nneth formula just like we did in the
proof of Lemma \ref{lemma:reduciblevanishing} (reducible vanishing) to see that
\[\HH_{b+1}(\bH(k);\St(\bH)) = \bigotimes_{j=1}^m \HH_{b_j+1}(\bH_j(k);\St(\bH_j)).\]
We thus see that it is enough to prove that the map
\begin{align*}
&\left(\bigotimes_{j=1}^{j_0-1} \HH_{b_j+1}(\bH_j(k);\St(\bH_j))\right) \otimes \HH_{b_{j_0}+1}(\bL_{\Delta'}^{\bH_{j_0}}(k);\St(\bL_{\Delta'}^{\bH_{j_0}})) \\
&\quad\quad\otimes \left(\bigotimes_{j=j_0+1}^{m} \HH_{b_j+1}(\bH_j(k);\St(\bH_j))\right)
\rightarrow \bigotimes_{j=1}^m \HH_{b_j+1}(\bH_j(k);\St(\bH_j))
\end{align*}
is surjective.  Since the maps on all but one tensor factor are the identity, this is equivalent
to the surjectivity of the map
\[\HH_{b_{j_0}+1}(\bL_{\Delta'}^{\bH_{j_0}}(k);\St(\bL_{\Delta'}^{\bH_{j_0}})) \rightarrow \HH_{b_{j_0}+1}(\bH_{j_0}(k);\St(\bH_{j_0}))\]
on the remaining tensor factor, which is exactly \eqref{surjectivity2}. 
\end{proof}

\section{Reduction to irreducible case}
\label{section:reductionirreducible}

In this section, we reduce Theorem \ref{maintheorem:fields} to the case of irreducible root systems.
Recall that for a reductive group $\bG$ 
we defined the bound $\bb(\bPhik(\bG))$ for our vanishing range in Table \ref{table:bounds} and equation
\eqref{eqn:reducibleb}.  We will recall the formula for $\bb(\bPhik(\bG))$ for
various $\bPhik(\bG)$ as it is needed.

\subsection{Irreducible cases}

The heart of our proof is verifying Theorem \ref{maintheorem:fields}
for groups $\bG$ whose relative root system $\bPhik(\bG)$ is one
of the non-exceptional types $\{\dA_n,\dB_n,\dC_n,\dBC_n,\dD_n\}$:

\begin{numberedtheorem}{maintheorem:fields}{1}[{Type $\dA$}]
\label{theorem:typean}
Let $\bG$ be a reductive group with $\bPhik(\bG) = \dA_n$ for some $n \geq 1$.  Then
$\HH_i(\bG(k);\St(\bG)) = 0$ for $i \leq \bb(\bPhik(\bG)) = \lfloor (n-1)/2 \rfloor$.
\end{numberedtheorem}

\begin{numberedtheorem}{maintheorem:fields}{2}[{Type $\dB/\dC/\dBC$}]
\label{theorem:typebcn}
Let $\bG$ be a reductive group with $\bPhik(\bG) \in \{\dB_n,\dC_n,\dBC_n\}$
for some\footnote{The root systems $\dB_n$ and $\dC_n$ and $\dBC_n$ are only defined for $n \geq 2$.} $n \geq 2$.  Then
$\HH_i(\bG(k);\St(\bG)) = 0$ for $i \leq \bb(\bPhik(\bG)) = \lfloor (n-2)/2 \rfloor$.
\end{numberedtheorem}

\begin{numberedtheorem}{maintheorem:fields}{3}[{Type $\dD$}]
\label{theorem:typedn}
Let $\bG$ be a reductive group with $\bPhik(\bG) = \dD_n$
for some\footnote{The root system $\dD_n$ is only defined for $n \geq 4$.} $n \geq 4$.  Then
$\HH_i(\bG(k);\St(\bG)) = 0$ for $i \leq \bb(\bPhik(\bG)) = \lfloor (n-3)/2 \rfloor$
\end{numberedtheorem}

We will prove these three theorems in Parts \ref{part:an}, \ref{part:bcn}, and \ref{part:dn}.

\subsection{General case}

Before starting our work on Theorems \ref{theorem:typean}, \ref{theorem:typebcn}, and \ref{theorem:typedn}, we show that
they imply the general case:

\theoremstyle{plain}
\newtheorem*{maintheorem:fields}{Theorem \ref{maintheorem:fields}}
\begin{maintheorem:fields}
Let $\bG$ be a reductive group and
$\bbF$ be a commutative ring.  Then we have $\HH_i(\bG(k);\St(\bG;\bbF))=0$ for
$i \leq \bb(\bPhik(\bG))$.
\end{maintheorem:fields}
\begin{proof}[Proof of Theorem \ref{maintheorem:fields}, assuming Theorems \ref{theorem:typean}, \ref{theorem:typebcn}, and \ref{theorem:typedn}]
By Lemma \ref{lemma:reductionlemma}, we can assume that $\bbF = \Z$, and thus omit $\bbF$ from our notation.
If $\bPhik(\bG)$ is irreducible, then it either lies in $\{\dA_n,\dB_n,\dC_n,\dBC_n,\dD_n\}$ or is one
of the exceptional root systems $\{\dG_2, \dF_4, \dE_6, \dE_7, \dE_8\}$.  If $\bPhik(\bG)$ is non-exceptional,
then the result follows from Theorems \ref{theorem:typean}, \ref{theorem:typebcn}, and \ref{theorem:typedn}.  If it
is exceptional, then by definition we have $\bb(\bPhik(\bG)) = 0$, so the theorem asserts that $\HH_0(\bG(k);\St(\bG)) = 0$.
This is exactly the content of Lemma \ref{lemma:h0}.

It remains to handle the case where $\bPhik(\bG)$ is reducible.  For some $m \geq 2$, we can write
\[\bPhik(\bG) = \bPhi_1 \times \cdots \times \bPhi_m\] 
with each $\bPhi_i$ irreducible.  By definition, we have
\[\bb(\bPhik(\bG)) = (m-1) + \sum_{j=1}^m \bb(\bPhi_j).\]
Since we have already proved our theorem for groups with irreducible relative root systems,
Lemma \ref{lemma:reduciblevanishing} applies and shows that $\HH_i(\bG(k);\St(\bG;\bbF))=0$ for
$i \leq \bb(\bPhik(\bG))$.
\end{proof}

\part{Vanishing in type \texorpdfstring{$\dA$}{A} (Theorem \ref{theorem:typean})}
\label{part:an}

This part of the paper proves our vanishing result in type $\dA$ (Theorem \ref{theorem:typean}).
For our induction, we prove a stronger result (Theorem \ref{theorem:strongan})
that incorporates a surjectivity statement.  We state
this in \S \ref{section:strongan}.  After the preliminary \S \ref{section:levivanishan} about
Levi subgroups, in \S \ref{section:ssvanishan} we show that our 
inductive hypotheses gives a vanishing region in the spectral sequence
from Corollary \ref{corollary:spectralsequence}.  As we show in \S \ref{section:proofan}, this
almost implies Theorem \ref{theorem:strongan}.  The only missing ingredients are three computations of
differentials in our spectral sequence, which are in \S \ref{section:differentialan}.

\section{Vanishing and surjectivity (type \texorpdfstring{$\dA$}{A})}
\label{section:strongan}

In this section, we first introduce some notation for the standard Levi factors
of groups of type $\dA_n$.  We then state a stronger version of Theorem \ref{theorem:typean} and prove
it for ranks at most $2$.

\subsection{Levi factor notation}
\label{section:levian}

Let $\bG$ be a reductive group with $\bPhik(\bG) = \dA_n$.
Let $\bDelta = \bDeltak(\bG)$ be the set of simple roots of $\bPhik(\bG) = \dA_n$.  Number the
elements of $\bDelta$ from left to right as in the usual Dynkin diagram:\\
\centerline{\psfig{file=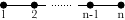,scale=2}}
For $1 \leq j_1,\ldots,j_{\ell} \leq n$, let $\bDelta[j_1,\ldots,j_{\ell}]$ be the result of removing
the simple roots labeled $j_1,\ldots,j_{\ell}$ from $\bDelta$.  We thus have a standard Levi
subgroup $\bL_{\bDelta[j_1,\ldots,j_{\ell}]}$ of $\bG$.

\begin{example}
We have $\bPhik(\bL_{\bDelta[1]}) = \bPhik(\bL_{\bDelta[n]}) = \dA_{n-1}$, while for
$2 \leq j \leq n-1$ we have $\bPhik(\bL_{\bDelta[j]}) = \dA_{j-1} \times \dA_{n-j}$.  In general,
for distinct $1 \leq j_1,\ldots,j_{\ell} \leq n$ we have
\[\bPhik(\bL_{\bDelta[j_1, \ldots, j_{\ell}]}) = \dA_{n_1} \times \cdots \times \dA_{n_m} \quad \text{with $n_1 + \cdots + n_m + \ell = n$}.\qedhere\]
\end{example}

We have a Reeder map (cf.\ \S \ref{section:reedermap}) of the form $\St(\bL_{\bDelta[j_1, \ldots, j_{\ell}]}) \rightarrow \St(\bG)$, and
thus maps $\HH_i(\bL_{\bDelta[j_1, \ldots, j_{\ell}]}(k);\St(\bL_{\bDelta[j_1, \ldots, j_{\ell}]})) \rightarrow \HH_i(\bG(k);\St(\bG))$.

\subsection{Strong vanishing}

The main result we will prove in this part of the paper is:

\begin{primedtheorem}{theorem:typean}
\label{theorem:strongan}
Let $\bG$ be a reductive group with $\bPhik(\bG) = \dA_n$.  Then:
\begin{itemize}
\item $\HH_i(\bG(k);\St(\bG)) = 0$ for $i \leq \lfloor (n-1)/2 \rfloor$; and
\item letting $\bDelta = \bDeltak(\bG)$, the maps
\begin{align*}
&\HH_i(\bL_{\bDelta[1]}(k);\St(\bL_{\bDelta[1]})) \rightarrow \HH_i(\bG(k);\St(\bG)) \quad \text{and} \\
&\HH_i(\bL_{\bDelta[n]}(k);\St(\bL_{\bDelta[n]})) \rightarrow \HH_i(\bG(k);\St(\bG))
\end{align*}
are both surjective for $i \leq \lfloor n/2 \rfloor$.
\end{itemize}
\end{primedtheorem}

This strengthens Theorem \ref{theorem:typean} by adding the indicated surjectivity statement.  This surjectivity
statement will be used to understand differentials in the spectral sequence from Corollary \ref{corollary:spectralsequence}.
To connect Theorem \ref{theorem:strongan}
to what we have already proven, we show that the spectral sequence argument from \S \ref{section:rank2} implies
Theorem \ref{theorem:strongan} for $n \leq 2$.  This requires a lemma.

\begin{lemma}
\label{lemma:anconjugate}
Let $\bG$ be a reductive group with $\bPhik(\bG) = \dA_n$ for some $n \geq 2$.  Then the images
of the two maps
\begin{align*}
&\HH_i(\bL_{\bDelta[1]}(k);\St(\bL_{\bDelta[1]})) \rightarrow \HH_i(\bG(k);\St(\bG)) \quad \text{and} \\
&\HH_i(\bL_{\bDelta[n]}(k);\St(\bL_{\bDelta[n]})) \rightarrow \HH_i(\bG(k);\St(\bG))
\end{align*}
are the same.
\end{lemma}
\begin{proof}
Lemma \ref{lemma:leviconjugate} says that $\bL_{\bDelta[1]}(k)$ and $\bL_{\bDelta[n]}(k)$ are conjugate
subgroups of $\bG(k)$.  This conjugation matches up the parabolic subgroups of $\bL_{\bDelta[1]}$ and $\bL_{\bDelta[n]}$, and
thus induces an isomorphism from $\St(\bL_{\bDelta[1]})$ to $\St(\bL_{\bDelta[n]})$.  
The lemma now follows from the fact that inner automorphisms act trivially on
group homology.
\end{proof}

\begin{lemma}
\label{lemma:typeanstrongrank2}
Theorem \textnormal{\ref{theorem:strongan}} holds for $n \leq 2$.
\end{lemma}
\begin{proof}
Let $\bG$ be a reductive group with $\bPhik(\bG) = \dA_n$ for some $n \leq 2$.  
For $n=0$, Theorem \ref{theorem:strongan} asserts nothing.
For $n = 1$, since a map to $0$ is surjective Theorem \ref{theorem:strongan} only asserts that
$\HH_0(\bG(k);\St(\bG)) = 0$ when $\bPhik(\bG) = \dA_1$, which follows from Lemma \ref{lemma:h0}.

For $n=2$, Theorem \ref{theorem:strongan} asserts 
that $\HH_0(\bG(k);\St(\bG)) = 0$, which follows
from Lemma \ref{lemma:h0}.  Theorem \ref{theorem:strongan} also asserts a surjectivity statement for $\HH_0$ and $\HH_1$.
The surjectivity statement for $\HH_0$ is vacuous since $\HH_0(\bG(k);\St(\bG)) = 0$, 
so the only nontrivial thing to prove is surjectivity for $\HH_1$.  Recall that
$\cL_p(\bG)$ consists of all subsets\footnote{Before we used the notation $\Delta$ instead of $R$,
but we use $R$ here to avoid confusion between the boldface $\bDelta$ and the non-boldface $\Delta$.} $R$ of $\bDelta$ 
with $|R| = n-p-1$.  Lemma \ref{lemma:rank2} says that
\begin{equation}
\label{eqn:typeanstrongrank2key}
\bigoplus_{R \in \cL_0(\bG)} \HH_1(\bL_{R}(k);\St(\bL_{R})) = \HH_1(\bL_{\bDelta[1]}(k);\St(\bL_{\bDelta[1]})) \oplus \HH_1(\bL_{\bDelta[2]}(k);\St(\bL_{\bDelta[2]}))
\end{equation}
surjects onto $\HH_1(\bG(k);\St(\bG))$.  Lemma \ref{lemma:anconjugate} says that both terms of \eqref{eqn:typeanstrongrank2key} have
the same image in $\HH_1(\bG(k);\St(\bG))$.  This implies that both surject onto $\HH_1(\bG(k);\St(\bG))$, as desired.
\end{proof}

Because of Lemma \ref{lemma:typeanstrongrank2}, for the rest of this part of the paper we can assume that $n \geq 3$.  We will also
assume as an inductive hypothesis that we have already proved Theorem \ref{theorem:strongan} in smaller ranks.  For this, we make the following
definition:

\begin{definition}
\label{definition:hypothesisan}
For $r \geq 0$, the {\em $r$-surjectivity and vanishing hypothesis in type $\dA$} is as follows.
Let $\bG$ be a reductive group with $\bPhik(\bG) = \dA_n$ for some $n \leq r$.  Then:
\begin{itemize}
\item $\HH_i(\bG(k);\St(\bG)) = 0$ for $i \leq \lfloor (n-1)/2 \rfloor$; and
\item letting $\bDelta = \bDeltak(\bG)$, the maps
\begin{align*}
&\HH_i(\bL_{\bDelta[1]}(k);\St(\bL_{\bDelta[1]})) \rightarrow \HH_i(\bG(k);\St(\bG)) \quad \text{and} \\
&\HH_i(\bL_{\bDelta[n]}(k);\St(\bL_{\bDelta[n]})) \rightarrow \HH_i(\bG(k);\St(\bG))
\end{align*}
are both surjective for $i \leq \lfloor n/2 \rfloor$.\qedhere
\end{itemize}
\end{definition} 

\section{Vanishing and surjectivity for Levi subgroups (type \texorpdfstring{$\dA$}{A})}
\label{section:levivanishan}

In this section, we show how to use the $r$-surjectivity and vanishing hypothesis
in type $\dA$ to analyze the homology of standard Levi subgroups.  Our main result
is as follows.  Its statement uses the ordering on the simple roots of 
of $\dA_{n_{j_0}}$ discussed in \S \ref{section:levian}.

\begin{lemma}[Levi vanishing and surjectivity]
\label{lemma:levivanishan}
Assume the $(n-1)$-surjectivity and vanishing hypothesis in type $\dA$ (Definition \ref{definition:hypothesisan}).
Let $\bG$ be a reductive group with $\bPhik(\bG) = \dA_n$.  Let $\Delta \subset \bDeltak(\bG)$ be a set
of simple roots with $\Delta \neq \bDeltak(\bG)$.  Write
\[\bPhik(\bL_{\Delta}) = \dA_{n_1} \times \cdots \times \dA_{n_m}.\]
Set
\[b = \bb(\bPhik(\bL_{\Delta})) = (m-1) + \lfloor (n_1-1)/2 \rfloor + \cdots + \lfloor(n_m-1)/2 \rfloor.\]
Then the following hold:
\begin{itemize}
\item[(i)] We have $\HH_i(\bL_{\Delta}(k);\St(\bL_{\Delta})) = 0$ for $i \leq b$.
\item[(ii)] For some $1 \leq j_0 \leq m$, assume that $n_{j_0}$ is even and nonzero.  Let $\Delta' \subset \Delta$
be the set of simple roots obtained by removing either the first or last simple root from $\dA_{n_{j_0}}$, so
\[\bPhik(\bL_{\Delta'}) = \dA_{n_1} \times \cdots \times \dA_{n_{j_0}-1} \times \cdots \times \dA_{n_m}.\]
Then the map $\HH_{b+1}(\bL_{\Delta'}(k);\St(\bL_{\Delta'})) \rightarrow \HH_{b+1}(\bL_{\Delta}(k);\St(\bG))$
is surjective.
\end{itemize}
\end{lemma}
\begin{proof}
Since $\Delta \neq \bDeltak(\bG)$, we have $n_j \leq n-1$ for $1 \leq j \leq m$.  The
$(n-1)$-surjectivity and vanishing hypothesis in type $\dA$ thus applies to all reductive
groups $\bH_j$ with $\bPhik(\bH_j) = \dA_{n_j}$.  This gives the hypothesis \eqref{vanishing1} in
Lemma \ref{lemma:reduciblevanishing} (reducible vanishing).  Applying Lemma \ref{lemma:reduciblevanishing},
we deduce (i).  Similarly, for $\Delta'$ as in (ii) it gives the hypotheses \eqref{surjectivity1} and \eqref{surjectivity2}
in Lemma \ref{lemma:reduciblesurjectivity} (reducible surjectivity).  Applying Lemma \ref{lemma:reduciblesurjectivity},
we deduce (ii).
\end{proof}

\section{Vanishing region (type \texorpdfstring{$\dA$}{A})}
\label{section:ssvanishan}

Let $\bG$ be a reductive group with $\bPhik(\bG) = \bA_n$ for some $n \geq 3$.
Corollary \ref{corollary:spectralsequence} gives a spectral sequence $\ssE^r_{pq}$ converging
to $\HH_{p+q}(\bG(k);\St(\bG))$ with
\[\ssE^1_{pq} \cong \begin{cases}
\bigoplus_{R \in \cL_p(\bG)} \HH_q(\bL_{R}(k);\St(\bL_{R})) & \text{if $0 \leq p \leq n-1$} \\
\HH_q(\bG(k);\St(\bG)^{\otimes 2})                          & \text{if $p = n$},\\
0                                                           & \text{otherwise}.
\end{cases}\]
The following lemma shows that our inductive hypothesis
implies that many terms of this spectral sequence vanish.

\begin{lemma}
\label{lemma:vanishan}
Let $\bG$ be a reductive group with $\bPhik(\bG) = \bA_n$ for some $n \geq 3$.
Assume the $(n-1)$-surjectivity and vanishing hypothesis in type $\dA$ (Definition \ref{definition:hypothesisan}).
Let $\ssE^1_{pq}$ be the spectral
sequence from Corollary \ref{corollary:spectralsequence}.  Then the following hold:
\begin{itemize}
\item For $n = 2d+1$ with $d \geq 1$, we have $\ssE^1_{pq} = 0$ for $p+q \leq d$ except for possibly
$\ssE^1_{0d}$.
\item For $n = 2d$ with $d \geq 2$, we have $\ssE^1_{pq} = 0$ for $p+q \leq d$ except for possibly
$\ssE^1_{0d}$ and $\ssE^1_{1,d-1}$.
\end{itemize}
\end{lemma}
\begin{proof}
Our goal is to prove a vanishing result for $\ssE^1_{pq}$.  The terms in question all have $p \leq \lfloor n/2 \rfloor$,
so they all satisfy\footnote{In fact, since $n \geq 3$ they even satisfy $p \leq n-2$.} $p \leq n-1$ and are therefore of the form
\[\ssE^1_{pq} = \bigoplus_{R \in \cL_p(\bG)} \HH_q(\bL_{R}(k);\St(\bL_{R})).\]
Consider $R \in \cL_p(\bG)$.  We will prove that our assumptions imply
that $\HH_q(\bL_R(k);\St(\bL_R)) = 0$ for the $p$ and $q$ where the lemma claims that
$\ssE^1_{pq} = 0$.

Since $R$ is obtained by deleting $p+1$ simple roots from $\bDelta$, we have
\[\bPhik(\bL_R) = \dA_{n_1} \times \cdots \times \dA_{n_m} \quad \text{with $n_1+\cdots+n_m+p+1 = n$}.\]
Lemma \ref{lemma:levivanishan} (Levi vanishing and surjectivity) implies 
that $\HH_q(\bL_R(k);\St(\bL_R)) = 0$ for
\[q \leq (m-1) + \lfloor (n_1-1)/2 \rfloor + \cdots + \lfloor (n_m-1)/2 \rfloor.\]
For $a,b \in \Z$, Lemma \ref{lemma:floorinequality} below implies that
$1+ \lfloor a/2 \rfloor + \lfloor b/2 \rfloor \geq \lfloor (a+b+1)/2 \rfloor$.
Applying this repeatedly, we deduce that
\[\bb(\dA_{n_1} \times \cdots \times \dA_{n_m}) \geq \left\lfloor (n_1+\cdots+n_m-1)/2 \right\rfloor = \left\lfloor (n-p-2)/2 \right\rfloor.\]
It follows that
\[\ssE^1_{pq} = 0 \quad \text{for $p \leq n-1$ and $q \leq \lfloor (n-p-2)/2 \rfloor$}.\]
For $p \leq n-1$, we deduce that $\ssE^1_{pq} = 0$ for
\[p+q \leq p+ \lfloor (n-p-2)/2 \rfloor = \lfloor (n+p-2)/2 \rfloor.\]
We now separate the cases where $n$ is odd and even:
\begin{itemize}
\item If $n = 2d+1$ is odd, then $\ssE^1_{pq} = 0$ for $p \leq n-1$ such that
\[p+q \leq \lfloor (n+p-2)/2 \rfloor = \lfloor (2d+p-1)/2 \rfloor = d + \lfloor (p-1)/2 \rfloor.\]
In particular, $\ssE^1_{pq} = 0$ for $p+q \leq d$ except for possibly
$\ssE^1_{0d}$.
\item If $n = 2d$ is even, then $\ssE^1_{pq} = 0$ for $p \leq n-1$ such that
\[p+q \leq \lfloor (n+p-2)/2 \rfloor = \lfloor (2d+p-2)/2 \rfloor = d - 1 + \lfloor p/2 \rfloor.\]
In particular, $\ssE^1_{pq} = 0$ for $p+q \leq d$ except for possibly
$\ssE^1_{0d}$ and $\ssE^1_{1,d-1}$.\qedhere
\end{itemize}
\end{proof}

The above proof used:

\begin{lemma}
\label{lemma:floorinequality}
For $a,b,d \in \Z$ with $d \geq 2$, we have $1+\lfloor a/d \rfloor + \lfloor b/d \rfloor \geq \lfloor (a+b+1)/d \rfloor$.
\end{lemma}
\begin{proof}
Write $a = q_1 d + r_1$ and $b = q_2 d + r_2$ with $0 \leq r_i \leq d-1$.  We then have
\[\lfloor (a+b+1)/d \rfloor = q_1+q_2+\lfloor (r_1+r_2+1)/d \rfloor \leq q_1+q_2  + \lfloor (2d-1)/d \rfloor = \lfloor a/d \rfloor + \lfloor b/d \rfloor + 1.\qedhere\]
\end{proof}

\section{Remaining tasks (type \texorpdfstring{$\dA$}{A})}
\label{section:proofan}

Lemma \ref{lemma:vanishan} implies many cases of Theorem \ref{theorem:strongan}.  To prove the
remaining cases, we need to compute some differentials in our spectral sequence.  We now explain
the structure of the argument, postponing three calculations to the next 
section.  Recall that Theorem \ref{theorem:strongan} is:

\theoremstyle{plain}
\newtheorem*{theorem:strongan}{Theorem \ref{theorem:strongan}}
\begin{theorem:strongan}
Let $\bG$ be a reductive group with $\bPhik(\bG) = \dA_n$.  Then:
\begin{itemize}
\item $\HH_i(\bG(k);\St(\bG)) = 0$ for $i \leq \lfloor (n-1)/2 \rfloor$; and
\item letting $\bDelta = \bDeltak(\bG)$, the maps
\begin{align*}
&\HH_i(\bL_{\bDelta[1]}(k);\St(\bL_{\bDelta[1]})) \rightarrow \HH_i(\bG(k);\St(\bG)) \quad \text{and} \\
&\HH_i(\bL_{\bDelta[n]}(k);\St(\bL_{\bDelta[n]})) \rightarrow \HH_i(\bG(k);\St(\bG))
\end{align*}
are both surjective for $i \leq \lfloor n/2 \rfloor$.
\end{itemize}
\end{theorem:strongan}
\begin{proof}
The proof is by induction
on $n$.  We proved the base cases $n \leq 2$ in Lemma \ref{lemma:typeanstrongrank2},
so we can assume that $n \geq 3$ and that the result is true for smaller ranks, i.e., that
the $(n-1)$-surjectivity and vanishing hypothesis in type $\dA$ holds.   

Corollary \ref{corollary:spectralsequence} gives a spectral sequence $\ssE^r_{pq}$ converging
to $\HH_{p+q}(\bG(k);\St(\bG))$, and Lemma \ref{lemma:vanishan} implies that
$\ssE^1_{pq} = 0$ for $p+q \leq \lfloor n/2 \rfloor - 1$.  This implies that
$\HH_i(\bG(k);\St(\bG)) = 0$ for $i \leq \lfloor n/2 \rfloor - 1$.
Since our surjectivity claim is trivial when the target is $0$, all that remains to prove
are the following two claims:

\begin{claim}{1}
Assume that $n = 2d+1$ with $d \geq 1$.  Then $\HH_d(\bG(k);\St(\bG)) = 0$.
\end{claim}

In this case, Lemma \ref{lemma:vanishan} says that the only potentially nonzero
term $\ssE^1_{pq}$ in our spectral sequence with $p+q = d$ is $\ssE^1_{0d}$.  We will prove in
Lemma \ref{lemma:differential2an} below that the differential $\ssE^1_{1d} \rightarrow \ssE^1_{0d}$
is surjective, so $\ssE^2_{0d} = 0$.  This implies that $\HH_d(\bG(k);\St(\bG)) = 0$, as desired.

\begin{claim}{2}
Assume that $n = 2d$ with $d \geq 2$.  Then the maps
\begin{align*}
&\HH_d(\bL_{\bDelta[1]}(k);\St(\bL_{\bDelta[1]})) \rightarrow \HH_d(\bG(k);\St(\bG)) \quad \text{and} \\
&\HH_d(\bL_{\bDelta[n]}(k);\St(\bL_{\bDelta[n]})) \rightarrow \HH_d(\bG(k);\St(\bG))
\end{align*}
are both surjective.
\end{claim}

Lemma \ref{lemma:anconjugate} says that these maps have the same image, so it is enough to prove that
\begin{equation}
\label{eqn:sumoftwoan}
\HH_d(\bL_{\bDelta[1]}(k);\St(\bL_{\bDelta[1]})) \oplus \HH_d(\bL_{\bDelta[n]}(k);\St(\bL_{\bDelta[n]}))
\end{equation}
surjects onto $\HH_d(\bG(k);\St(\bG))$.  
Lemma \ref{lemma:vanishan} says that the only potentially nonzero
terms $\ssE^1_{pq}$ in our spectral sequence with $p+q = d$ are $\ssE^1_{0d}$ and $\ssE^1_{1,d-1}$.  We
will prove in Lemma \ref{lemma:differential3an} below that the differential $\ssE^1_{2,d-1} \rightarrow \ssE^1_{1,d-1}$
is surjective, so $\ssE^2_{1,d-1} = 0$.  We will also prove in Lemma \ref{lemma:differential1an}
below that the summand \eqref{eqn:sumoftwoan} of
\[\ssE^1_{0d} = \bigoplus_{R \in \cL_0(\bG)} \HH_d(\bL_{R}(k);\St(\bL_{R})) = \bigoplus_{j=1}^{n} \HH_d(\bL_{\bDelta[j]}(k);\St(\bL_{\bDelta[j]}))\]
surjects onto the cokernel of the differential $\ssE^1_{1d} \rightarrow \ssE^1_{0d}$.
It follows that $\ssE^2_{0d}$ is a quotient of \eqref{eqn:sumoftwoan}.  Since $\ssE^2_{0d}$ is the
only potentially nonzero term of the form $\ssE^2_{pq}$ with $p+q = d$, it follows that \eqref{eqn:sumoftwoan}
surjects onto $\HH_d(\bG(k);\St(\bG))$, as desired.
\end{proof}

\section{Differentials (type \texorpdfstring{$\dA$}{A})}
\label{section:differentialan}

This final section of this part of the paper
determines the images of three differentials whose calculations were needed in the previous section.

\subsection{Differentials, I (type \texorpdfstring{$\dA$}{A})}
\label{section:differential1an}

Our first differential calculation is:

\begin{lemma}
\label{lemma:differential1an}
Let $\bG$ be a reductive group with $\bPhik(\bG) = \bA_{2d}$ for some $d \geq 2$.  
Assume the $(2d-1)$-surjectivity and vanishing hypothesis in type $\dA$ (Definition \ref{definition:hypothesisan}).
Let $\ssE^1_{pq}$ be the spectral
sequence from Corollary \ref{corollary:spectralsequence}.  Then the summand
\[\HH_d(\bL_{\bDelta[1]}(k);\St(\bL_{\bDelta[1]})) \oplus \HH_d(\bL_{\bDelta[2d]}(k);\St(\bL_{\bDelta[2d]}))\]
of $\ssE^1_{0d}$ surjects onto the cokernel of the 
differential $\ssE^1_{1d} \rightarrow \ssE^1_{0d}$.
\end{lemma}
\begin{proof}
As notation, for $1 \leq j_1,\ldots,j_{\ell} \leq 2d$ let 
\[M[j_1, \ldots, j_{\ell}] = \HH_d(\bL_{\bDelta[j_1, \ldots, j_{\ell}]}(k);\St(\bL_{\bDelta[j_1, \ldots, j_{\ell}]})).\]
We have
\[\ssE^1_{0d} = \bigoplus_{1 \leq j \leq 2d} M[j] \quad \text{and} \quad 
\ssE^1_{1d} = \bigoplus_{1 \leq j_1 < j_2 \leq 2d} M[j_1, j_2].\]
Consider some $1 < j < 2d$.  We must prove that when we quotient $\ssE^1_{0d}$ by the image of the differential
$\ssE^1_{1d} \rightarrow \ssE^1_{0d}$, the summand
$M[j]$ of $\ssE^1_{0d}$ is identified with a subspace of
$M[1] \oplus M[2d]$.  In fact, we will show that it is identified with a subspace of $M[1]$.  Our argument
will use the fact that $j<2d$, so it does not show that $M[2d]$ is identified with a subspace of $M[1]$.  
Note that
\begin{align}
\bPhik(\bL_{\bDelta[j]}) &= \dA_{j-1} \times \dA_{2d-j} \label{eqn:differential1an.1} \\
\bb(\bPhik(\bL_{\bDelta[j]})) &= 1 + \lfloor (j-2)/2 \rfloor + \lfloor (2d-j-1)/2 \rfloor = d-1. \label{eqn:differential1an.2}
\end{align}
There are two cases.

The first is that $j$ is odd.  Recall that by assumption $j>1$.
Let $f\colon M[1,j] \rightarrow M[j]$ and $g\colon M[1,j] \rightarrow M[1]$ be the obvious maps.
It is then immediate from the definitions that the differential 
$\ssE^1_{1d} \rightarrow \ssE^1_{0d}$
takes the summand $M[1,j]$ of $\ssE^1_{1d}$ to $\ssE^1_{0d}$ via the map
\[\begin{tikzcd}[column sep=large]
M[1,j] \arrow{r}{f \oplus (-g)} & M[j] \oplus M[1] \arrow[hook]{r} & \ssE^1_{0d}.
\end{tikzcd}\]
Since the $\dA_{j-1}$-factor in \eqref{eqn:differential1an.1} has $j-1$ 
even and positive, we can use Lemma \ref{lemma:levivanishan} (Levi vanishing and surjectivity)
to see that $f$ is surjective.  Here we are using the fact that $\bb(\bPhik(\bL_{\bDelta[j]}))+1 = d$; cf.\ \eqref{eqn:differential1an.2}.  
Thus when we quotient
$\ssE^1_{0d}$ by the image of the differential, $M[j]$ is identified with a subspace
of $M[1]$, as desired.

Now assume that $j$ is even.  Recall that by assumption $j<2d$.
Let $f'\colon M[j,j+1] \rightarrow M[j+1]$ and $g'\colon M[j,j+1] \rightarrow M[j]$ be the obvious maps.
Just like above, the differential $\ssE^1_{1d} \rightarrow \ssE^1_{0d}$
takes the summand $M[j,j+1]$ of $\ssE^1_{1d}$ to $\ssE^1_{0d}$ via the map
\[\begin{tikzcd}[column sep=large]
M[j,j+1] \arrow{r}{f' \oplus (-g')} & M[j+1] \oplus M[j] \arrow[hook]{r} &  \ssE^1_{0d}.
\end{tikzcd}\]
Since the $\dA_{2d-j}$-factor in \eqref{eqn:differential1an.1} has $2d-j$ even and positive,
we can use Lemma \ref{lemma:levivanishan} (Levi vanishing and surjectivity) 
to see that $g'$ is surjective.  Again, we are using the fact that 
$\bb(\bPhik(\bL_{\bDelta[j]}))+1 = d$; cf.\ \eqref{eqn:differential1an.2}.  Thus when we
quotient $\ssE^1_{0d}$ by the image of the differential, $M[j]$ is identified with
a subspace of $M[j+1]$.  Since $j+1$ is odd, the previous paragraph shows that
this quotient identifies $M[j+1]$ with a subspace of $M[1]$, completing the proof.
\end{proof}

\subsection{Differentials, II (type \texorpdfstring{$\dA$}{A})}
\label{section:differential2an}

Our second differential calculation is:

\begin{lemma}
\label{lemma:differential2an}
Let $\bG$ be a reductive group with $\bPhik(\bG) = \bA_{2d+1}$ for some $d \geq 1$.  
Assume the $2d$-surjectivity and vanishing hypothesis in type $\dA$ (Definition \ref{definition:hypothesisan}).
Let $\ssE^1_{pq}$ be the spectral
sequence from Corollary \ref{corollary:spectralsequence}.  Then the
differential $\ssE^1_{1d} \rightarrow \ssE^1_{0d}$ is surjective.
\end{lemma}
\begin{proof}
As notation, for $1 \leq j_1,\ldots,j_{\ell} \leq 2d+1$ let
\[M[j_1, \ldots, j_{\ell}] = \HH_d(\bL_{\bDelta[j_1, \ldots, j_{\ell}]}(k);\St(\bL_{\bDelta[j_1, \ldots, j_{\ell}]})).\]
We have
\[\ssE^1_{0d} = \bigoplus_{1 \leq j \leq 2d+1} M[j] \quad \text{and} \quad 
\ssE^1_{1d} = \bigoplus_{1 \leq j_1 < j_2 \leq 2d+1} M[j_1, j_2].\]
Consider some $1 \leq j \leq 2d+1$.  We must prove that when we quotient $\ssE^1_{0d}$ by the image of the differential
$\ssE^1_{1d} \rightarrow \ssE^1_{0d}$, the summand $M[j]$ is killed.
The first step is to show that many $M[j]$ already vanish:

\begin{claim}{1}
\label{claim:differential2an.1}
For $1 \leq j \leq 2d+1$ with $j$ even, we have $M[j] = 0$.
\end{claim}
\begin{proof}[Proof of claim]
Write $j = 2e$, so $\bPhik(\bL_{\bDelta[2e]}) = \dA_{2e-1} \times \dA_{2d+1-2e}$.  
Lemma \ref{lemma:levivanishan} (Levi vanishing and surjectivity) implies that
$\HH_i(\bL_{\bDelta[2e]}(k);\St(\bL_{\bDelta[2e]})) = 0$ for
\[i \leq 1 + \lfloor (2e-2)/2 \rfloor + \lfloor (2d-2e)/2 \rfloor = 1 + (e-1) + (d-e) = d.\]
In particular, $M[2e] = \HH_d(\bL_{\bDelta[2e]}(k);\St(\bL_{\bDelta[2e]})) = 0$.
\end{proof}

Now consider $1 \leq j \leq 2d+1$ with $j$ odd.  In light of Claim \ref{claim:differential2an.1}, it is enough
to prove that $M[j]$ is killed when we quotient $\ssE^1_{0d}$ by the image of the differential
$\ssE^1_{1d} \rightarrow \ssE^1_{0d}$.  Assume first that $j \neq 1$, so
\begin{equation}
\label{eqn:differential2an.1}
\bPhik(\bL_{\bDelta[j]})      = \begin{cases}
\dA_{j-1} \times \dA_{2d-j+1} & \text{if $1 < j < 2d+1$},\\
\dA_{2d}                      & \text{if $j = 2d+1$}.
\end{cases}
\end{equation}
Since $j$ is odd, in both cases we have
\begin{equation}
\label{eqn:differential2an.11}
\bb(\bPhik(\bL_{\bDelta[j]})) = d-1.
\end{equation}
We will show that
$M[j]$ is killed by the image of the summand $M[j-1,j]$ of $\ssE^1_{1d}$.  
On the summand $M[j-1,j]$, the differential is the map
\[\begin{tikzcd}
M[j-1,j] \arrow{r} & M[j] \oplus M[j-1] \arrow[hook]{r} & \ssE^1_{0d}.
\end{tikzcd}\]
Claim \ref{claim:differential2an.1} says that $M[j-1] = 0$, so to show that this differential kills
$M[j]$ it is enough to prove that $M[j-1,j] \rightarrow M[j]$ is surjective.  Since
the $\dA_{j-1}$ in \eqref{eqn:differential2an.1} has $j-1$ even and positive, this follows
from Lemma \ref{lemma:levivanishan} (Levi vanishing and surjectivity).  Here we are using
the fact that $\bb(\bPhik(\bL_{\bDelta[j]}))+1 = d$; cf.\ \eqref{eqn:differential2an.11}.

It remains to deal with the case $j = 1$.  Note that
\begin{align}
\bPhik(\bL_{\bDelta[1]}) &= \dA_{2d}, \label{eqn:differential2an.2} \\
\bb(\bPhik(\bL_{\bDelta[1]})) &= \lfloor (2d-1)/2 \rfloor = d-1. \label{eqn:differential2an.21}
\end{align}
In this case, we will use the summand $M[1,2]$ of $\ssE^1_{1d}$.  Just like
above, on this summand this differential takes the form
\[\begin{tikzcd}
M[1,2] \arrow{r} & M[2] \oplus M[1] \arrow[hook]{r} & \ssE^1_{0d}.
\end{tikzcd}\]
Claim \ref{claim:differential2an.1} says that $M[2] = 0$, and since the $\dA_{2d}$ in
\eqref{eqn:differential2an.2} has $2d$ even and positive Lemma \ref{lemma:levivanishan} (Levi vanishing and surjectivity) shows that
the map $M[1,2] \rightarrow M[1]$ is surjective.  
Here we are using
the fact that $\bb(\bPhik(\bL_{\bDelta[1]}))+1 = d$; cf.\ \eqref{eqn:differential2an.21}.
The lemma follows.
\end{proof}

\subsection{Differentials, III (type \texorpdfstring{$\dA$}{A})}
\label{section:differential3an}

Our final differential calculation is:

\begin{lemma}
\label{lemma:differential3an}
Let $\bG$ be a reductive group with $\bPhik(\bG) = \bA_{2d}$ for some $d \geq 2$.
Assume the $(2d-1)$-surjectivity and vanishing hypothesis in type $\dA$ (Definition \ref{definition:hypothesisan}).
Let $\ssE^1_{pq}$ be the spectral
sequence from Corollary \ref{corollary:spectralsequence}.  Then the differential
$\ssE^1_{2,d-1} \rightarrow \ssE^1_{1,d-1}$ is surjective.
\end{lemma}
\begin{proof}
As notation, for $1 \leq j_1,\ldots,j_{\ell} \leq 2d$ let
\[M[j_1, \ldots, j_{\ell}] = \HH_{d-1}(\bL_{\bDelta[j_1, \ldots, j_{\ell}]}(k);\St(\bL_{\bDelta[j_1, \ldots, j_{\ell}]})).\]
We have
\[\ssE^1_{1,d-1} = \bigoplus_{1 \leq j_1 < j_2 \leq 2d} M[j_1, j_2] \quad \text{and} \quad 
\ssE^1_{2,d-1} = \bigoplus_{1 \leq j_1 < j_2 < j_3 \leq 2d} M[j_1, j_2, j_3].\]
Consider some $1 \leq j_1 < j_2 \leq 2d$.  We must prove that when we quotient $\ssE^1_{1,d-1}$ by the image of the differential
$\ssE^1_{2,d-1} \rightarrow \ssE^1_{1,d-1}$, the summand $M[j_1, j_2]$ is killed.
The first step is to show that many $M[j_1, j_2]$ already vanish:

\begin{claim}{1}
\label{claim:differential3an.1}
For $1 \leq j_1 < j_2 \leq 2d$ with either $j_1$ even or $j_2$ odd, we have $M[j_1, j_2] = 0$.
\end{claim}
\begin{proof}[Proof of claim]
To simplify our notation, we will let $\dA_0 = \{0\} \subset \R^0$, regarded as a trivial root system of
rank $0$.  For a product
$\dA_{n_1} \times \cdots \times \dA_{n_m}$ with $n_j \geq 0$, the bound
\[(m-1) + \lfloor (n_1-1)/2 \rfloor + \cdots + \lfloor (n_m-1)/2 \rfloor\]
for homology vanishing from Lemma \ref{lemma:levivanishan} (Levi vanishing and surjectivity)
is still correct since any $n_j$ with $n_j = 0$ contribute $+1$ to $(m-1)$ and $-1$ to the rest of the sum.

Now observe that
$\bPhik(\bL_{\bDelta[j_1,j_2]}) = \dA_{j_1-1} \times \dA_{j_2-j_1-1} \times \dA_{2d-j_2}$.
By Lemma \ref{lemma:levivanishan} (Levi vanishing and surjectivity),
we have
$\HH_i(\bL_{\bDelta[j_1,j_2]}(k);\St(\bL_{\bDelta[j_1,j_2]})) = 0$ for
\begin{align*}
i \leq &2 + \lfloor (j_1-2)/2 \rfloor + \lfloor (j_2-j_1-2)/2 \rfloor + \lfloor (2d-j_2-1)/2 \rfloor \\
  =    &d + \lfloor j_1/2 \rfloor + \lfloor (j_2-j_1)/2 \rfloor + \lfloor (-j_2-1)/2 \rfloor.
\end{align*}
It follows that $M[j_1,j_2] = 0$ if the right hand side is at least $d-1$, i.e., if
\begin{equation}
\label{eqn:differential3an.1.toprove}
\lfloor j_1/2 \rfloor + \lfloor (j_2-j_1)/2 \rfloor + \lfloor (-j_2-1)/2 \rfloor \geq -1.
\end{equation}
The left hand side of \eqref{eqn:differential3an.1.toprove} is $-1$ if either $j_1$ is even or $j_2$ is odd, and is $-2$ otherwise.
The claim follows.
\end{proof}

Now consider $1 \leq j_1 < j_2 \leq 2d$ with $j_1$ odd and $j_2$ even.  In light of Claim \ref{claim:differential3an.1}, it is enough
to prove that $M[j_1, j_2]$ is killed when we quotient $\ssE^1_{1,d-1}$ by the image of the differential
$\ssE^1_{2,d-1} \rightarrow \ssE^1_{1,d-1}$.  Assume first that $j_1 > 1$, so
\begin{equation} 
\label{eqn:differential3an.1}
\bPhik(\bL_{\bDelta[j_1, j_2]}) =
\begin{cases}
\dA_{j_1-1} \times \dA_{j_2-j_1-1} \times \dA_{2d-j_2} & \text{if $j_2>j_1+1$, $j_2 < 2d$},\\
\dA_{j_1-1} \times \dA_{2d-j_1-1}                     & \text{if $j_2>j_1+1$, $j_2 = 2d$},\\
\dA_{j_1-1} \times \dA_{2d-j_1-1}                        & \text{if $j_2=j_1+1$, $j_2 < 2d$},\\
\dA_{2d-2}                                            & \text{if $j_2=j_1+1$, $j_2=2d$}.
\end{cases}
\end{equation}
In all four cases, since $j_1$ is odd and $j_2$ is even we have
\begin{equation}
\label{eqn:differential3an.11}
\bb(\bPhik(\bL_{\bDelta[j_1,j_2]})) = d-2.
\end{equation}
We will show that
$M[j_1, j_2]$ is killed by the image of the summand $M[j_1-1,j_1,j_2]$ of $\ssE^1_{2,d-1}$. 
On the summand $M[j_1-1,j_1,j_2]$, the differential is the map
\[\begin{tikzcd}
M[j_1-1,j_1,j_2] \arrow{r} & M[j_1,j_2] \oplus M[j_1-1,j_2] \oplus M[j_1-1,j_1] \arrow[hook]{r} & \ssE^1_{1,d-1}.
\end{tikzcd}\]
Since $j_1-1$ is even, 
Claim \ref{claim:differential3an.1} says that $M[j_1-1,j_2] = M[j_1-1,j_1] = 0$, so to show that this differential kills
$M[j_1, j_2]$ it is enough to prove that $M[j_1-1,j_1,j_2] \rightarrow M[j_1,j_2]$ is surjective.  Since the first $\dA$-term
in all four cases of \eqref{eqn:differential3an.1} has a subscript that is even and positive, this follows
from Lemma \ref{lemma:levivanishan} (Levi vanishing and surjectivity).  Here we are using the fact
that $\bb(\bPhik(\bL_{\bDelta[j_1,j_2]})) + 1 = d-1$; cf.\ \eqref{eqn:differential3an.11}.

Assume next that $j_2 < 2d$.  We remind the reader that $j_1$ is odd and $j_2$ is even.  We have
\begin{equation} 
\label{eqn:differential3an.2}
\bPhik(\bL_{\bDelta[j_1, j_2]}) =
\begin{cases}
\dA_{j_1-1} \times \dA_{j_2-j_1-1} \times \dA_{2d-j_2} & \text{if $j_1>1$, $j_2 > j_1+1$},\\
\dA_{j_2-2} \times \dA_{2d-j_2}                    & \text{if $j_1=1$, $j_2 > j_1+1$},\\
\dA_{j_2-2} \times \dA_{2d-j_2}                        & \text{if $j_1>1$, $j_2 = j_1+1$},\\
\dA_{2d-2}                                           & \text{if $j_1=1$, $j_2 = j_1+1$}.
\end{cases}
\end{equation}
In all four cases, since $j_1$ is odd and $j_2$ is even we have
\begin{equation}
\label{eqn:differential3an.21}
\bb(\bPhik(\bL_{\bDelta[j_1,j_2]})) = d-2.
\end{equation}
We will show that
$M[j_1, j_2]$ is killed by the image of the summand $M[j_1,j_2,j_2+1]$ of $\ssE^1_{2,d-1}$. 
On the summand $M[j_1,j_2,j_2+1]$, the differential is the map
\[\begin{tikzcd}
M[j_1,j_2,j_2+1] \arrow{r} & M[j_2,j_2+1] \oplus M[j_1,j_2+1] \oplus M[j_1,j_2] \arrow[hook]{r} & \ssE^1_{1,d-1}.
\end{tikzcd}\]
Since $j_2+1$ is odd, 
Claim \ref{claim:differential3an.1} says that $M[j_2,j_2+1] = M[j_1,j_2+1] = 0$, so to show that this differential kills
$M[j_1, j_2]$ it is enough to prove that $M[j_1,j_2,j_2+1] \rightarrow M[j_1,j_2]$ is surjective.  Since the last $\dA$-term 
in all our cases of \eqref{eqn:differential3an.2} has a subscript that is even and positive, this follows
from Lemma \ref{lemma:levivanishan} (Levi vanishing and surjectivity).
Here we are using the fact
that $\bb(\bPhik(\bL_{\bDelta[j_1,j_2]})) + 1 = d-1$; cf.\ \eqref{eqn:differential3an.21}.

The only case not covered by the previous two paragraphs is $j_1 = 1$ and $j_2 = 2d$.  We have
\begin{align} 
\bPhik(\bL_{\bDelta[1, 2d]}) &= \dA_{2d-2}, \label{eqn:differential3an.3}\\
\bb(\bPhik(\bL_{\bDelta[1, 2d]})) &= \lfloor (2d-3)/2 \rfloor = d-2. \label{eqn:differential3an.31}
\end{align}
In this case, we will use the summand $M[1,2,2d]$ of $\ssE^1_{2,d-1}$, which makes
sense since $d \geq 2$ by assumption.  Just like
above, on this summand the differential takes the form
\[\begin{tikzcd}
M[1,2,2d] \arrow{r} & M[2,2d] \oplus M[1,2d] \oplus M[1,2] \arrow[hook]{r} & \ssE^1_{1,d-1}.
\end{tikzcd}\]
Claim \ref{claim:differential3an.1} says that $M[2,2d] = 0$, and the previous paragraph
implies that $M[1,2]$ dies in the cokernel of the differential.  Finally, since the $\dA_{2d-2}$ in
\eqref{eqn:differential3an.3} has $2d-2$ even and positive Lemma \ref{lemma:levivanishan} (Levi vanishing and surjectivity) shows that
the map $M[1,2,2d] \rightarrow M[1,2d]$ is surjective.  
Here we are using the fact
that $\bb(\bPhik(\bL_{\bDelta[j_1,j_2]})) + 1 = d-1$; cf.\ \eqref{eqn:differential3an.31}.
It follows that $M[1,2d]$ dies in the cokernel
of the differential, as desired.
\end{proof}

\part{Vanishing in types \texorpdfstring{$\dB$}{B} and \texorpdfstring{$\dC$}{C} and \texorpdfstring{$\dBC$}{BC} (Theorem \ref{theorem:typebcn})}
\label{part:bcn}

This part of the paper is devoted to Theorem \ref{theorem:typebcn}, 
which is our vanishing result in types $\dB$ and $\dC$ and $\dBC$.  The proofs follow
the same outline as those for type $\dA$ in Part \ref{part:an}. 

\section{Vanishing and surjectivity (types \texorpdfstring{$\dB$}{B} and \texorpdfstring{$\dC$}{C} and \texorpdfstring{$\dBC$}{BC})}
\label{section:strongbcn}

In this section, we first introduce some notation for the standard Levi factors
of groups of type $\dB_n$ and $\dC_n$ and $\dBC_n$.  We then state a stronger version of Theorem \ref{theorem:typebcn} and prove
it in rank $2$.

\subsection{Levi factor notation}
\label{section:levibcn}

Let $\bG$ be a reductive group with $\bPhik(\bG) \in \{\dB_n, \dC_n, \dBC_n\}$ for some $n \geq 2$.  
We introduce the following convention:

\begin{convention}
Throughout this part of the paper, we will write $\dX_n$ with $\dX$ meaning either $\dB$ or $\dC$
or $\dBC$.  These root systems are only defined for $n \geq 2$, but to allow
uniform statements we will define $\dX_1 = \dA_1$.
\end{convention}

Let $\bDelta = \bDeltak(\bG)$ be the set of simple roots of $\bPhik(\bG)$.  Number the
elements of $\bDelta$ from left to right as in the usual Dynkin diagram; for instance,
if $\bPhik(\bG) = \dB_n$ we have:\\
\centerline{\psfig{file=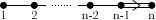,scale=2}}
For $1 \leq j_1,\ldots,j_{\ell} \leq n$, let $\bDelta[j_1, \ldots, j_{\ell}]$ be the result of removing
the simple roots labeled $j_1,\ldots,j_{\ell}$ from $\bDelta$.  We thus have a standard Levi
subgroup $\bL_{\bDelta[j_1, \ldots, j_{\ell}]}$ of $\bG$.

\begin{example}
We have $\bPhik(\bL_{\bDelta[1]}) = \dX_{n-1}$ and
$\bPhik(\bL_{\bDelta[n]}) = \dA_{n-1}$, while for
$2 \leq j \leq n-1$ we have $\bPhik(\bL_{\bDelta[j]}) = \dA_{j-1} \times \dX_{n-j}$.  In general,
for distinct $1 \leq j_1,\ldots,j_{\ell} \leq n$ we have either
\begin{align*}
&\bPhik(\bL_{\bDelta[j_1, \ldots, j_{\ell}]}) = \dA_{n_1} \times \cdots \times \dA_{n_{m-1}} \times \dX_{n_m} \quad \text{or} \\
&\bPhik(\bL_{\bDelta[j_1, \ldots, j_{\ell}]}) = \dA_{n_1} \times \cdots \times \dA_{n_{m}}
\end{align*}
with $n_1 + \cdots + n_m + \ell = n$.
\end{example}

We have a Reeder map (cf.\ \S \ref{section:reedermap}) of the form $\St(\bL_{\bDelta[j_1, \ldots, j_{\ell}]}) \rightarrow \St(\bG)$, and
thus maps $\HH_i(\bL_{\bDelta[j_1, \ldots, j_{\ell}]}(k);\St(\bL_{\bDelta[j_1, \ldots, j_{\ell}]})) \rightarrow \HH_i(\bG(k);\St(\bG))$.

\subsection{Strong vanishing}

The main result we will prove in this part of the paper is:

\begin{primedtheorem}{theorem:typebcn}
\label{theorem:strongbcn}
Let $\bG$ be a reductive group with $\bPhik(\bG) = \dX_n$ for some $\dX \in \{\dB, \dC, \dBC\}$ and $n \geq 2$.  Then:
\begin{itemize}
\item $\HH_i(\bG(k);\St(\bG)) = 0$ for $i \leq \lfloor (n-2)/2 \rfloor$; and
\item letting $\bDelta = \bDeltak(\bG)$, the map
$\HH_i(\bL_{\bDelta[1]}(k);\St(\bL_{\bDelta[1]})) \rightarrow \HH_i(\bG(k);\St(\bG))$
is surjective for $i \leq \lfloor (n-1)/2 \rfloor$.
\end{itemize}
\end{primedtheorem}

This strengthens Theorem \ref{theorem:typebcn} 
by adding the indicated surjectivity statement.
To avoid degenerate cases, we prove:

\begin{lemma}
\label{lemma:typebcnstrongrank2}
Theorem \textnormal{\ref{theorem:strongbcn}} holds for $n = 2$.
\end{lemma}
\begin{proof}
Let $\bG$ be a reductive group with $\bPhik(\bG) = \dX_2$ for some $\dX \in \{\dB, \dC, \dBC\}$.
Since a map to $0$ is surjective, for $\bG$ Theorem \ref{theorem:strongbcn} only asserts that
$\HH_0(\bG(k);\St(\bG)) = 0$, which follows from Lemma \ref{lemma:h0}.
\end{proof}

Because of Lemma \ref{lemma:typebcnstrongrank2}, for the rest of this part of the paper we can assume that $n \geq 3$.  We will also
assume as an inductive hypothesis that we have already proved Theorem \ref{theorem:strongbcn} in smaller ranks.  For this, we make the following
definition:

\begin{definition}
\label{definition:hypothesisbcn}
For $r \geq 2$, the {\em $r$-surjectivity and vanishing hypothesis in type $\dBC$} is as follows.
Let $\bG$ be a reductive group with $\bPhik(\bG) = \dX_n$ for some $\dX \in \{\dB, \dC, \dBC\}$ and $2 \leq n \leq r$.  Then:
\begin{itemize}
\item $\HH_i(\bG(k);\St(\bG)) = 0$ for $i \leq \lfloor (n-2)/2 \rfloor$; and
\item letting $\bDelta = \bDeltak(\bG)$, the map
$\HH_i(\bL_{\bDelta[1]}(k);\St(\bL_{\bDelta[1]})) \rightarrow \HH_i(\bG(k);\St(\bG))$
is surjective for $i \leq \lfloor (n-1)/2 \rfloor$.\qedhere
\end{itemize}
\end{definition}

\begin{remark}
\label{remark:alreadyanbcn}
We have already proven Theorem \ref{theorem:strongan}, so we also have available to us vanishing and surjectivity results
in type $\dA$.
\end{remark}

\section{Vanishing and surjectivity for Levi subgroups (types \texorpdfstring{$\dB$}{B} and \texorpdfstring{$\dC$}{C} and \texorpdfstring{$\dBC$}{BC})}
\label{section:levivanishbcn}

In this section, we show how to use the $r$-surjectivity and vanishing hypothesis
in type $\dBC$ to analyze the homology of standard Levi subgroups.  Our main result
is as follows.  Its statement uses the ordering on the simple roots of
of $\dA_{n_{j_0}}$ and $\dX_{n_{j_0}}$ for $\dX \in \{\dB, \dC, \dBC\}$ discussed in \S \ref{section:levian} and \S \ref{section:levibcn}.

\begin{lemma}[Levi vanishing and surjectivity]
\label{lemma:levivanishbcn}
For some $n \geq 3$, assume the $(n-1)$-surjectivity and vanishing hypothesis in type $\dBC$ (Definition \ref{definition:hypothesisbcn}).
Let $\bG$ be a reductive group with $\bPhik(\bG) = \dX_n$ for some $\dX \in \{\dB, \dC, \dBC\}$.  Let $\Delta \subset \bDeltak(\bG)$ be a set
of simple roots with $\Delta \neq \bDeltak(\bG)$.  Write
\[\bPhik(\bL_{\Delta}) = \dA_{n_1} \times \cdots \times \dZ_{n_m} \quad \text{with $\dZ \in \{\dA,\dX\}$}.\]
We thus have $n_1,\ldots,n_{m-1} \geq 1$, and if $\dZ = \dX$ then $n_m \geq 2$ while if $\dZ = \dA$ then $n_m \geq 1$.
Set $b = \bb(\bPhik(\bL_{\Delta}))$, so
\[b = 
\begin{cases}
(m-1) + \lfloor (n_1-1)/2 \rfloor + \cdots + \lfloor(n_{m-1}-1)/2 \rfloor +  \lfloor(n_m-1)/2 \rfloor & \text{if $\dZ = \dA$},\\
(m-1) + \lfloor (n_1-1)/2 \rfloor + \cdots + \lfloor(n_{m-1}-1)/2 \rfloor +  \lfloor(n_m-2)/2 \rfloor & \text{if $\dZ = \dX$}.
\end{cases}\]
Then the following hold:
\begin{itemize}
\item[(i)] We have $\HH_i(\bL_{\Delta}(k);\St(\bL_{\Delta})) = 0$ for $i \leq b$.
\item[(ii)] For some $1 \leq j_0 \leq n$, assume one of the following:
\begin{itemize}
\item $n_{j_0}$ is even and nonzero, $1 \leq j_0 \leq m-1$, and $\Delta' \subset \Delta$
is the set of simple roots obtained by removing either the first or last simple root from $\dA_{n_{j_0}}$.
\item $n_{j_0}$ is even and nonzero, $j_0 = m$ and $\dZ = \dA$, and $\Delta' \subset \Delta$
is the set of simple roots obtained by removing either the first or last simple root from $\dA_{n_{n_m}}$.
\item $n_{j_0}$ is odd, $j_0 = m$ and $\dZ = \dX$, and $\Delta' \subset \Delta$ is the set
of simple roots obtained by removing the first simple root from $\dZ_{n_m}$.
\end{itemize}
Then the map $\HH_{b+1}(\bL_{\Delta'}(k);\St(\bL_{\Delta'})) \rightarrow \HH_{b+1}(\bL_{\Delta}(k);\St(\bG))$
is surjective.
\end{itemize}
\end{lemma}
\begin{proof}
Since $\Delta \neq \bDeltak(\bG)$, we have $n_j \leq n-1$ for $1 \leq j \leq m$.  The
$(n-1)$-surjectivity and vanishing hypothesis in type $\dBC$ thus applies to all reductive
groups $\bH_m$ with $\bPhik(\bH_m) = \dX_{n_m}$.  Theorem \ref{theorem:strongan} also gives a vanishing and surjectivity
result for all reductive groups $\bH_j$ with $\bPhik(\bH_j) = \dA_{n_j}$.  This gives the hypothesis \eqref{vanishing1} in
Lemma \ref{lemma:reduciblevanishing} (reducible vanishing).  Applying Lemma \ref{lemma:reduciblevanishing},
we deduce (i).  Similarly, for $\Delta'$ as in (ii) it gives the hypotheses \eqref{surjectivity1} and \eqref{surjectivity2}
in Lemma \ref{lemma:reduciblesurjectivity} (reducible surjectivity).  Applying Lemma \ref{lemma:reduciblesurjectivity},
we deduce (ii).
\end{proof}

\section{Vanishing region (types \texorpdfstring{$\dB$}{B} and \texorpdfstring{$\dC$}{C} and \texorpdfstring{$\dBC$}{BC})}
\label{section:ssvanishbcn}

Let $\bG$ be a reductive group with
$\bPhik(\bG) = \dX_n$ for some $\dX \in \{\dB_n,\dC_n,\dBC_n\}$ and $n \geq 3$.
Corollary \ref{corollary:spectralsequence} gives a spectral sequence $\ssE^r_{pq}$ converging
to $\HH_{p+q}(\bG(k);\St(\bG))$ with
\[\ssE^1_{pq} \cong \begin{cases}
\bigoplus_{R \in \cL_p(\bG)} \HH_q(\bL_{R}(k);\St(\bL_{R})) & \text{if $0 \leq p \leq n-1$} \\
\HH_q(\bG(k);\St(\bG)^{\otimes 2})                          & \text{if $p = n$},\\
0                                                           & \text{otherwise}.
\end{cases}\]
The following lemma shows that our inductive hypothesis
implies that many terms of this spectral sequence vanish.

\begin{lemma}
\label{lemma:vanishbcn}
Let $\bG$ be a reductive group with 
$\bPhik(\bG) = \dX_n$ for some $\dX \in \{\dB_n,\dC_n,\dBC_n\}$ and $n \geq 3$.
Assume the $(n-1)$-surjectivity and vanishing hypothesis in type $\dBC$ (Definition \ref{definition:hypothesisbcn}).
Let $\ssE^1_{pq}$ be the spectral
sequence from Corollary \ref{corollary:spectralsequence}.  Then the following hold:
\begin{itemize}
\item For $n = 2d+2$ with $d \geq 1$, we have $\ssE^1_{pq} = 0$ for $p+q \leq d$ except for possibly
$\ssE^1_{0d}$. 
\item For $n = 2d+1$ with $d \geq 1$, we have $\ssE^1_{pq} = 0$ for $p+q \leq d$ except for possibly
$\ssE^1_{0d}$ and $\ssE^1_{1,d-1}$.  For $n=3$ (so $d = 1$), we also have $\ssE^1_{1,d-1}=0$.
\end{itemize}
\end{lemma}
\begin{proof}
This is identical to the proof of Lemma \ref{lemma:vanishan} in type $\dA$.  The only difference is
that the standard Levi subgroups can be of two types:
\begin{itemize}
\item $\dA_{n_1} \times \cdots \times \dA_{n_m}$, where the vanishing range for $\HH_i$ given by Theorem \ref{theorem:strongan} is
\[i \leq (m-1) + \lfloor (n_1-1)/2 \rfloor + \cdots + \lfloor (n_{m}-1)/2 \rfloor.\]
Repeatedly applying the inequality $1+\lfloor a/2 \rfloor + \lfloor b/2 \rfloor \geq \lfloor (a+b+1)/2 \rfloor$ from Lemma \ref{lemma:floorinequality},
the right hand side is at least
\begin{equation}
\label{eqn:vanishbcn.1}
\lfloor (n_1 + \cdots + n_m -1)/2 \rfloor.
\end{equation}
\item $\dA_{n_1} \times \cdots \times \dA_{n_{m-1}} \times \dX_{n_m}$, where the vanishing range for $\HH_i$
given by Lemma \ref{lemma:levivanishbcn} (Levi vanishing and surjectivity) is
\begin{equation}
\label{eqn:vanishbcn.21}
i \leq (m-1) + \lfloor (n_1-1)/2 \rfloor + \cdots + \lfloor (n_{m-1}-1)/2 \rfloor + \lfloor (n_m - 2)/2 \rfloor.
\end{equation}
Repeatedly applying the inequality $1+\lfloor a/2 \rfloor + \lfloor b/2 \rfloor \geq \lfloor (a+b+1)/2 \rfloor$ from Lemma \ref{lemma:floorinequality},
the right hand side is at least
\begin{equation}
\label{eqn:vanishbcn.2}
\lfloor (n_1 + \cdots + n_m -2)/2 \rfloor.
\end{equation}
\end{itemize}
The slightly worse vanishing range in \eqref{eqn:vanishbcn.2} (as opposed to \eqref{eqn:vanishbcn.1}, which is the same
vanishing range we found in the proof of Lemma \ref{lemma:vanishan}) accounts for the slightly worse range in the statement
of the lemma.  Finally, the fact that $\ssE^1_{1,d-1} = 0$ for $d=1$ follows from the fact
that $\ssE^1_{1,0}$ is a direct sum of $\HH_0$-groups, and these all vanish by Lemma \ref{lemma:h0}.
\end{proof}

\begin{remark}
\label{remark:needoffset1}
For later use, note that we were lucky that there was only one term of the form $\lfloor (n_j -2)/2 \rfloor$ on the right
hand side of \eqref{eqn:vanishbcn.21}.  If there were more than one, then the vanishing range would not be strong
enough for the lemma.  For instance, the bound
\[i \leq (m-1) + \lfloor (n_1-2)/2 \rfloor + \cdots + \lfloor (n_m-2)/2 \rfloor\]
would not be good enough.  Indeed, if all the $n_j$ were odd then the right hand side equals
\[\lfloor (n_1+\cdots+n_m-m-1)/2 \rfloor,\]
which is too small to give anything like Lemma \ref{lemma:vanishbcn}.
\end{remark}

\section{Remaining tasks (types \texorpdfstring{$\dB$}{B} and \texorpdfstring{$\dC$}{C} and \texorpdfstring{$\dBC$}{BC})} 
\label{section:proofbcn}

Lemma \ref{lemma:vanishbcn} implies many cases of Theorem \ref{theorem:strongbcn}.  To prove the
remaining cases, we need to compute some differentials in our spectral sequence.  We now explain
the structure of the argument, postponing three calculations to the 
next section.  Recall that Theorem \ref{theorem:strongbcn} is:

\theoremstyle{plain}
\newtheorem*{theorem:strongbcn}{Theorem \ref{theorem:strongbcn}}
\begin{theorem:strongbcn}
Let $\bG$ be a reductive group with $\bPhik(\bG) = \dX_n$ for some $\dX \in \{\dB, \dC, \dBC\}$ and $n \geq 2$.  Then:
\begin{itemize}
\item $\HH_i(\bG(k);\St(\bG)) = 0$ for $i \leq \lfloor (n-2)/2 \rfloor$; and
\item letting $\bDelta = \bDeltak(\bG)$, the map
$\HH_i(\bL_{\bDelta[1]}(k);\St(\bL_{\bDelta[1]})) \rightarrow \HH_i(\bG(k);\St(\bG))$
is surjective for $i \leq \lfloor (n-1)/2 \rfloor$.
\end{itemize}
\end{theorem:strongbcn}
\begin{proof}
The proof is by induction
on $n$.  We proved the base case $n = 2$ in Lemma \ref{lemma:typebcnstrongrank2},
so we can assume that $n \geq 3$ and that the result is true for smaller ranks, i.e., that
the $(n-1)$-surjectivity and vanishing hypothesis in type $\dBC$ holds.

Corollary \ref{corollary:spectralsequence} gives a spectral sequence $\ssE^r_{pq}$ converging
to $\HH_{p+q}(\bG(k);\St(\bG))$, and Lemma \ref{lemma:vanishbcn} implies that
$\ssE^1_{pq} = 0$ for $p+q \leq \lfloor (n-1)/2 \rfloor - 1$.  This implies that
$\HH_i(\bG(k);\St(\bG)) = 0$ for $i \leq \lfloor (n-1)/2 \rfloor - 1$.
Since our surjectivity claim is trivial when the target is $0$, all that remains to prove
are the following two claims:

\begin{claim}{1}
Assume that $n = 2d+2$ with $d \geq 1$.  Then $\HH_d(\bG(k);\St(\bG)) = 0$.
\end{claim}

In this case, Lemma \ref{lemma:vanishbcn} says that the only potentially nonzero
term $\ssE^1_{pq}$ in our spectral sequence with $p+q = d$ is $\ssE^1_{0d}$.  We will prove in
Lemma \ref{lemma:differential2bcn} below that the differential $\ssE^1_{1d} \rightarrow \ssE^1_{0d}$
is surjective, so $\ssE^2_{0d} = 0$.  This implies that $\HH_d(\bG(k);\St(\bG)) = 0$, as desired.

\begin{claim}{2}
Assume that $n = 2d+1$ with $d \geq 1$.  Then the map
\[\HH_d(\bL_{\bDelta[1]}(k);\St(\bL_{\bDelta[1]})) \rightarrow \HH_d(\bG(k);\St(\bG))\]
is surjective.
\end{claim}

Lemma \ref{lemma:vanishbcn} says that the only potentially nonzero
terms $\ssE^1_{pq}$ in our spectral sequence with $p+q = d$ are $\ssE^1_{0d}$ and $\ssE^1_{1,d-1}$.  Lemma \ref{lemma:vanishbcn}
also says that $\ssE^1_{1,d-1}=0$ if $d = 1$, and we
will prove in Lemma \ref{lemma:differential3bcn} below that the differential $\ssE^1_{2,d-1} \rightarrow \ssE^1_{1,d-1}$
is surjective for $d \geq 2$.  It follows that in all cases $\ssE^2_{1,d-1} = 0$.  We will also prove in Lemma \ref{lemma:differential1bcn}
below that the summand $\HH_d(\bL_{\bDelta[1]}(k);\St(\bL_{\bDelta[1]}))$ of
\[\ssE^1_{0d} = \bigoplus_{R \in \cL_0(\bG)} \HH_d(\bL_{R}(k);\St(\bL_{R})) = \bigoplus_{j=1}^{n} \HH_d(\bL_{\bDelta[j]}(k);\St(\bL_{\bDelta[j]}))\]
surjects onto the cokernel of the differential $\ssE^1_{1d} \rightarrow \ssE^1_{0d}$.
It follows that $\ssE^2_{0d}$ is a quotient of $\HH_d(\bL_{\bDelta[1]}(k);\St(\bL_{\bDelta[1]}))$.   Since $\ssE^2_{0d}$ is the
only potentially nonzero term of the form $\ssE^2_{pq}$ with $p+q = d$, it follows that $\HH_d(\bL_{\bDelta[1]}(k);\St(\bL_{\bDelta[1]}))$
surjects onto $\HH_d(\bG(k);\St(\bG))$, as desired.
\end{proof}

\section{Differentials (types \texorpdfstring{$\dB$}{B} and \texorpdfstring{$\dC$}{C} and \texorpdfstring{$\dBC$}{BC})}
\label{section:differentialbcn}
    
This final section of this part of the paper
determines the images of three differentials whose calculations were needed in the previous section.

\subsection{Differentials, I (types \texorpdfstring{$\dB$}{B} and \texorpdfstring{$\dC$}{C} and \texorpdfstring{$\dBC$}{BC})}
\label{section:differential1bcn}

Our first differential calculation is:

\begin{lemma}
\label{lemma:differential1bcn}
Let $\bG$ be a reductive group with $\bPhik(\bG) = \dX_{2d+1}$ for some $\dX \in \{\dB_n,\dC_n,\dBC_n\}$ and $d \geq 1$.
Assume the $2d$-surjectivity and vanishing hypothesis in type $\dBC$ (Definition \ref{definition:hypothesisbcn}).
Let $\ssE^1_{pq}$ be the spectral
sequence from Corollary \ref{corollary:spectralsequence}.  Then the summand $\HH_d(\bL_{\bDelta[1]}(k);\St(\bL_{\bDelta[1]}))$
of $\ssE^1_{0d}$ surjects onto the cokernel of the
differential $\ssE^1_{1d} \rightarrow \ssE^1_{0d}$. 
\end{lemma}
\begin{proof}
As notation, for $1 \leq j_1,\ldots,j_{\ell} \leq 2d+1$ let
\[M[j_1, \ldots, j_{\ell}] = \HH_d(\bL_{\bDelta[j_1, \ldots, j_{\ell}]}(k);\St(\bL_{\bDelta[j_1, \ldots, j_{\ell}]})).\]
We have
\[\ssE^1_{0d} = \bigoplus_{1 \leq j \leq 2d+1} M[j] \quad \text{and} \quad 
\ssE^1_{1d} = \bigoplus_{1 \leq j_1 < j_2 \leq 2d+1} M[j_1, j_2].\]
Consider some $1 < j \leq 2d+1$.  We must prove that when we quotient $\ssE^1_{0d}$ by the image of the differential
$\ssE^1_{1d} \rightarrow \ssE^1_{0d}$, the summand
$M[j]$ of $\ssE^1_{0d}$ is identified with a subspace of
$M[1]$.  Note that\noeqref{eqn:differential1bcn.1}\noeqref{eqn:differential1bcn.2}
\begin{align}
\bPhik(\bL_{\bDelta[j]}) &= 
\begin{cases} 
\dA_{j-1} \times \dX_{2d+1-j} & \text{if $j \leq 2d-1$},\\
\dA_{2d-1} \times \dA_1       & \text{if $j = 2d$}, \\
\dA_{2d}                      & \text{if $j = 2d+1$}, 
\end{cases} \label{eqn:differential1bcn.1} \\
\bb(\bPhik(\bL_{\bDelta[j]})) &= 
\begin{cases} 
1 + \lfloor (j-2)/2 \rfloor + \lfloor (2d-j-1)/2 \rfloor = d-1 & \text{if $j \leq 2d-1$}, \\ 
1 + \lfloor (2d-2)/2 \rfloor + \lfloor (1-1)/2 \rfloor   = d   & \text{if $j=2d$},\\
\lfloor (2d-1)/2 \rfloor = d-1                                 & \text{if $j=2d+1$}. 
\end{cases} \label{eqn:differential1bcn.2}
\end{align}
It follows that $M[2d]=0$, so we do not need to deal with that case.  There are three other cases.

The first case is $j$ odd and $j \neq 2d+1$.  Recall that by assumption $j>1$.
The differential $\ssE^1_{1d} \rightarrow \ssE^1_{0d}$
takes the summand $M[1,j]$ of $\ssE^1_{1d}$ to $\ssE^1_{0d}$ via the map
\[\begin{tikzcd}
M[1,j] \arrow{r} & M[j] \oplus M[1] \arrow[hook]{r} & \ssE^1_{0d}.
\end{tikzcd}\]
Since the $\dA_{j-1}$-factor in \eqref{eqn:differential1bcn.1} has $j-1$
even and positive, we can use Lemma \ref{lemma:levivanishbcn} (Levi vanishing and surjectivity)
to see that the map $M[1,j] \rightarrow M[j]$ is surjective.  Here we are using the fact that $\bb(\bPhik(\bL_{\bDelta[j]}))+1 = d$; cf.\ \eqref{eqn:differential1bcn.2}.
Thus when we quotient
$\ssE^1_{0d}$ by the image of the differential, $M[j]$ is identified with a subspace
of $M[1]$, as desired.

The second case is $j=2d+1$.  The differential $\ssE^1_{1d} \rightarrow \ssE^1_{0d}$
takes the summand $M[1,2d+1]$ of $\ssE^1_{1d}$ to $\ssE^1_{0d}$ via the map
\[\begin{tikzcd}
M[1,2d+1] \arrow{r} & M[2d+1] \oplus M[1] \arrow[hook]{r} &\ssE^1_{0d}.
\end{tikzcd}\]
Since the $\dA_{2d}$ in \eqref{eqn:differential1bcn.1} has $2d$
even and positive, we can use
Lemma \ref{lemma:levivanishbcn} (Levi vanishing and surjectivity)
to see that the map $M[1,2d+1] \rightarrow M[2d+1]$ is surjective.  Here we are using the fact that $\bb(\bPhik(\bL_{\bDelta[j]}))+1 = d$; cf.\ \eqref{eqn:differential1bcn.2}.
Thus when we quotient
$\ssE^1_{0d}$ by the image of the differential, $M[2d+1]$ is identified with a subspace
of $M[1]$, as desired.

The third case is $j$ even and $j \neq 2d$.
The differential $\ssE^1_{1d} \rightarrow \ssE^1_{0d}$
takes the summand $M[j,j+1]$ of $\ssE^1_{1d}$ to $\ssE^1_{0d}$ via the map
\[\begin{tikzcd}
M[j,j+1] \arrow{r} & M[j+1] \oplus M[j] \arrow[hook]{r} &  \ssE^1_{0d}.
\end{tikzcd}\]
Since the $\dX_{2d+1-j}$-factor in \eqref{eqn:differential1bcn.1} has $2d+1-j$ odd,
we can use Lemma \ref{lemma:levivanishbcn} (Levi vanishing and surjectivity)
to see that the map $M[j,j+1] \rightarrow M[j]$ is surjective.  Again, we are using the fact that
$\bb(\bPhik(\bL_{\bDelta[j]}))+1 = d$; cf.\ \eqref{eqn:differential1bcn.2}.  Thus when we
quotient $\ssE^1_{0d}$ by the image of the differential, $M[j]$ is identified with
a subspace of $M[j+1]$.  Since $j+1$ is odd, the previous paragraph shows that
this quotient identifies $M[j+1]$ with a subspace of $M[1]$, as desired.
\end{proof}

\subsection{Differentials, II (types \texorpdfstring{$\dB$}{B} and \texorpdfstring{$\dC$}{C} and \texorpdfstring{$\dBC$}{BC})}
\label{section:differential2bcn}

Our second differential calculation is:

\begin{lemma}
\label{lemma:differential2bcn}
Let $\bG$ be a reductive group with
$\bPhik(\bG) = \dX_{2d+2}$ for some $\dX \in \{\dB_n,\dC_n,\dBC_n\}$ and $d \geq 1$.
Assume the $(2d+1)$-surjectivity and vanishing hypothesis in type $\dBC$ (Definition \ref{definition:hypothesisbcn}).
Let $\ssE^1_{pq}$ be the spectral
sequence from Corollary \ref{corollary:spectralsequence}.  Then the
differential $\ssE^1_{1d} \rightarrow \ssE^1_{0d}$ is surjective.
\end{lemma}
\begin{proof}
As notation, for $1 \leq j_1,\ldots,j_{\ell} \leq 2d+2$ let
\[M[j_1, \ldots, j_{\ell}] = \HH_d(\bL_{\bDelta[j_1, \ldots, j_{\ell}]}(k);\St(\bL_{\bDelta[j_1, \ldots, j_{\ell}]})).\]
We have
\[\ssE^1_{0d} = \bigoplus_{1 \leq j \leq 2d+2} M[j] \quad \text{and} \quad 
\ssE^1_{1d} = \bigoplus_{1 \leq j_1 < j_2 \leq 2d+2} M[j_1, j_2].\]
Consider some $1 \leq j \leq 2d+2$.  We must prove that when we quotient $\ssE^1_{0d}$ by the image of the differential
$\ssE^1_{1d} \rightarrow \ssE^1_{0d}$, the summand $M[j]$ is killed.
The first step is to show that many $M[j]$ already vanish:

\begin{claim}{1}
\label{claim:differential2bcn.1}
For $1 \leq j \leq 2d+2$ with $j$ even, we have $M[j] = 0$.
\end{claim}
\begin{proof}[Proof of claim]
Write $j = 2e$, so
\[\bPhik(\bL_{\bDelta[2e]}) = \begin{cases}
\dA_{2e-1} \times \dX_{2d+2-2e} & \text{if $2e \neq 2d+2$}, \\
\dA_{2d+1}                      & \text{if $2e = 2d+2$}.
\end{cases}\]
Lemma \ref{lemma:levivanishbcn} (Levi vanishing and surjectivity) implies that
$\HH_i(\bL_{\bDelta[2e]}(k);\St(\bL_{\bDelta[2e]})) = 0$ for
\[i \leq 
\begin{cases}
1 + \lfloor (2e-2)/2 \rfloor + \lfloor (2d-2e)/2 \rfloor = 1 + (e-1) + (d-e) = d & \text{if $2e \neq 2d+2$}, \\
\lfloor 2d/2 \rfloor = d & \text{if $2e = 2d+2$}.
\end{cases}\]
In particular, $M[2e] = \HH_d(\bL_{\bDelta[2e]}(k);\St(\bL_{\bDelta[2e]})) = 0$.
\end{proof}

Now consider $1 \leq j \leq 2d+2$ with $j$ odd.  In light of Claim \ref{claim:differential2bcn.1}, it is enough
to prove that $M[j]$ is killed when we quotient $\ssE^1_{0d}$ by the image of the differential
$\ssE^1_{1d} \rightarrow \ssE^1_{0d}$.  Since $j$ is odd, we have
\begin{align}
\bPhik(\bL_{\bDelta[j]}) &= \begin{cases}
\dA_{j-1} \times \dX_{2d+2-j} & \text{if $j \neq 1,2d+1$},\\
\dX_{2d+1}                    & \text{if $j=1$},\\
\dA_{2d} \times \dA_1        & \text{if $j=2d+1$},
\end{cases} \label{eqn:differential2bcn.1} \\
\bb(\bPhik(\bL_{\bDelta[j]})) &= \begin{cases}
1+\lfloor (j-2)/2 \rfloor + \lfloor (2d-j)/2 \rfloor = d-1 & \text{if $j \neq 1,2d+1$},\\
\lfloor (2d-2)/2 \rfloor = d-1 & \text{if $j=1$},\\
1+\lfloor (2d-1)/2 \rfloor + \lfloor (1-1)/2 \rfloor = d & \text{if $j=2d+1$}. \label{eqn:differential2bcn.2}
\end{cases}
\end{align}
It follows that $M[2d+1]=0$, so we do not need to deal with that case.  There are two other cases.

The first case is $j \neq 1$.  Remember that $j$ is odd.  
On the summand $M[j-1,j]$, the differential is the map
\[\begin{tikzcd}
M[j-1,j] \arrow{r} & M[j] \oplus M[j-1] \arrow[hook]{r} & \ssE^1_{0d}.
\end{tikzcd}\]
Claim \ref{claim:differential2bcn.1} says that $M[j-1] = 0$, so to show that this differential kills
$M[j]$ it is enough to prove that $M[j-1,j] \rightarrow M[j]$ is surjective.  Since
the $\dA_{j-1}$ in \eqref{eqn:differential2bcn.1} has $j-1$ even and positive, this follows
from Lemma \ref{lemma:levivanishbcn} (Levi vanishing and surjectivity).  Here we are using
the fact that $\bb(\bPhik(\bL_{\bDelta[j]}))+1 = d$; cf.\ \eqref{eqn:differential2bcn.2}.

The second case is $j=1$.  On the summand $M[1,2]$, the differential is the map
\[\begin{tikzcd}[column sep=large]
M[1,2] \arrow{r} & M[2] \oplus M[1] \arrow[hook]{r} & \ssE^1_{0d}.
\end{tikzcd}\]
Claim \ref{claim:differential2bcn.1} says that $M[2] = 0$, and since the $\dX_{2d+1}$ in
\eqref{eqn:differential2bcn.2} has $2d+1$ odd Lemma \ref{lemma:levivanishbcn} (Levi vanishing and surjectivity) shows that
the map $M[1,2] \rightarrow M[1]$ is surjective.
Here we are using
the fact that $\bb(\bPhik(\bL_{\bDelta[1]}))+1 = d$; cf.\ \eqref{eqn:differential2bcn.2}.
The lemma follows.
\end{proof}

\subsection{Differentials, III (types \texorpdfstring{$\dB$}{B} and \texorpdfstring{$\dC$}{C} and \texorpdfstring{$\dBC$}{BC})}
\label{section:differential3bcn}

Our final differential calculation is:

\begin{lemma}
\label{lemma:differential3bcn}
Let $\bG$ be a reductive group with
$\bPhik(\bG) = \dX_{2d+1}$ for some $\dX \in \{\dB_n,\dC_n,\dBC_n\}$ and $d \geq 2$.
Assume the $2d$-surjectivity and vanishing hypothesis in type $\dBC$ (Definition \ref{definition:hypothesisbcn}).
Let $\ssE^1_{pq}$ be the spectral
sequence from Corollary \ref{corollary:spectralsequence}.  Then the differential
$\ssE^1_{2,d-1} \rightarrow \ssE^1_{1,d-1}$ is surjective.
\end{lemma}
\begin{proof}
As notation, for $1 \leq j_1,\ldots,j_{\ell} \leq 2d+1$ let
\[M[j_1, \ldots, j_{\ell}] = \HH_{d-1}(\bL_{\bDelta[j_1, \ldots, j_{\ell}]}(k);\St(\bL_{\bDelta[j_1, \ldots, j_{\ell}]})).\]
We have
\[\ssE^1_{1,d-1} = \bigoplus_{1 \leq j_1 < j_2 \leq 2d+1} M[j_1, j_2] \quad \text{and} \quad 
\ssE^1_{2,d-1} = \bigoplus_{1 \leq j_1 < j_2 < j_3 \leq 2d+1} M[j_1, j_2, j_3].\]
Consider some $1 \leq j_1 < j_2 \leq 2d+1$.  We must prove that when we quotient $\ssE^1_{1,d-1}$ by the image of the differential
$\ssE^1_{2,d-1} \rightarrow \ssE^1_{1,d-1}$, the summand $M[j_1, j_2]$ is killed.
The first step is to show that many $M[j_1, j_2]$ already vanish:

\begin{claim}{1}
\label{claim:differential3bcn.1}
If $1 \leq j_1 < j_2 \leq 2d+1$ with either $j_1$ even or $j_2$ odd or $j_2 = 2d$, then $M[j_1, j_2] = 0$.
\end{claim}
\begin{proof}[Proof of claim]
To simplify our notation, we will let $\dA_0 = \{0\} \subset \R^0$, regarded as a trivial root system of
rank $0$.  This does not affect the bounds for homology vanishing in Lemma \ref{lemma:levivanishbcn} (Levi vanishing and surjectivity).
With this convention,
\[\bPhik(\bL_{\bDelta[j_1,j_2]}) = \begin{cases}
\dA_{j_1-1} \times \dA_{j_2-j_1-1} \times \dX_{2d+1-j_2} & \text{if $j_2 \neq 2d,2d+1$}, \\
\dA_{j_1-1} \times \dA_{2d-j_1-1} \times \dA_1           & \text{if $j_2 = 2d$},\\
\dA_{j_1-1} \times \dA_{2d-j_1}                          & \text{if $j_2 = 2d+1$}.
\end{cases}\]
If $j_2 = 2d$ and $j_1$ is arbitrary, then Lemma \ref{lemma:levivanishbcn} (Levi vanishing and surjectivity)
implies that $\HH_i(\bL_{\bDelta[j_1,j_2]}(k);\St(\bL_{\bDelta[j_1,j_2]})) = 0$ for
\[i \leq 2+\lfloor (j_1-2)/2 \rfloor + \lfloor (2d-j_1-2)/2 \rfloor + \lfloor (1-1)/2 \rfloor = d + \lfloor j_1/2 \rfloor + \lfloor -j_1/2 \rfloor.\]
The right hand side is at least $d-1$, so $M[j_1,j_2] = 0$.  If $j_2 = 2d+1$ and $j_1$ is arbitrary, then Lemma \ref{lemma:levivanishbcn} (Levi vanishing and surjectivity) implies that
$\HH_i(\bL_{\bDelta[j_1,j_2]}(k);\St(\bL_{\bDelta[j_1,j_2]})) = 0$ for
\[i \leq 1 + \lfloor (j_1-2)/2 \rfloor + \lfloor (2d-j_1-1)/2 \rfloor = d + \lfloor j_1/2 \rfloor + \lfloor (-j_1-1)/2 \rfloor = d-1,\]
so $M[j_1,j_2] = 0$.  Finally, if $j_2 \neq 2d,2d+1$ then Lemma \ref{lemma:levivanishbcn} (Levi vanishing and surjectivity) implies that
$\HH_i(\bL_{\bDelta[j_1,j_2]}(k);\St(\bL_{\bDelta[j_1,j_2]})) = 0$ for
\begin{align*}
i \leq &2 + \lfloor (j_1-2)/2 \rfloor + \lfloor (j_2-j_1-2)/2 \rfloor + \lfloor (2d-j_2-1)/2 \rfloor \\
  =    &d + \lfloor j_1/2 \rfloor + \lfloor (j_2-j_1)/2 \rfloor + \lfloor (-j_2-1)/2 \rfloor.
\end{align*}
It follows that $M[j_1,j_2] = 0$ if the right hand side is at least $d-1$, i.e., if
\begin{equation}
\label{eqn:differential3bcn.1.toprove}
\lfloor j_1/2 \rfloor + \lfloor (j_2-j_1)/2 \rfloor + \lfloor (-j_2-1)/2 \rfloor \geq -1.
\end{equation}
The left hand side of \eqref{eqn:differential3bcn.1.toprove} is $-1$ if either $j_1$ is even or $j_2$ is odd, and is $-2$ otherwise.
The claim follows.
\end{proof}

Now consider $1 \leq j_1 < j_2 \leq 2d-2$ with $j_1$ odd and $j_2$ even.  In light of Claim \ref{claim:differential3bcn.1}, it is enough
to prove that $M[j_1, j_2]$ is killed when we quotient $\ssE^1_{1,d-1}$ by the image of the differential
$\ssE^1_{2,d-1} \rightarrow \ssE^1_{1,d-1}$.  We have
\begin{equation}
\label{eqn:differential3bcn.1}
\bPhik(\bL_{\bDelta[j_1, j_2]}) =
\begin{cases}
\dA_{j_1-1} \times \dA_{j_2-j_1-1} \times \dX_{2d+1-j_2} & \text{if $j_1 > 1$ and $j_2>j_1+1$},\\
\dA_{j_2-2} \times \dX_{2d+1-j_2}                        & \text{if $j_1 = 1$ and $j_2>j_1+1$},\\
\dA_{j_1-1} \times \dX_{2d-j_1}                          & \text{if $j_1 > 1$ and $j_2=j_1+1$},\\
\dX_{2d-1}                                               & \text{if $j_1=1$ and $j_2 = j_1+1$}.
\end{cases}
\end{equation}
In all four cases, since $j_1$ is odd and $j_2$ is even we have
\begin{equation}
\label{eqn:differential3bcn.11}
\bb(\bPhik(\bL_{\bDelta[j_1,j_2]})) = d-2.
\end{equation}
On the summand $M[j_1,j_2,j_2+1]$, the differential is the map
\[\begin{tikzcd}
M[j_1,j_2,j_2+1] \arrow{r} & M[j_2,j_2+1] \oplus M[j_1,j_2+1] \oplus M[j_1,j_2] \arrow[hook]{r} & \ssE^1_{1,d-1}.
\end{tikzcd}\]
Since $j_2+1$ is odd, Claim \ref{claim:differential3bcn.1} says that $M[j_2,j_2+1] = M[j_1,j_2+1] = 0$, so to show that this differential kills
$M[j_1, j_2]$ it is enough to prove that $M[j_1,j_2,j_2+1] \rightarrow M[j_1,j_2]$ is surjective.  Since the $\dX$-term that appears in all
four cases of \eqref{eqn:differential3an.1} has an odd subscript, this follows
from Lemma \ref{lemma:levivanishbcn} (Levi vanishing and surjectivity).  Here we are using the fact
that $\bb(\bPhik(\bL_{\bDelta[j_1,j_2]})) + 1 = d-1$; cf.\ \eqref{eqn:differential3bcn.11}.
\end{proof}

\part{Vanishing in type \texorpdfstring{$\dD$}{D} (Theorem \ref{theorem:typedn})}
\label{part:dn}

This part of the paper is devoted to Theorem \ref{theorem:typedn},
which is our vanishing result in type $\dD$.  The proofs follow
the same outline as those for type $\dA$ in Part \ref{part:an}.

\section{Vanishing and surjectivity (type \texorpdfstring{$\dD$}{D})}
\label{section:strongdn}

In this section, we first introduce some notation for the standard Levi factors
of groups of type $\dD_n$.  We then state a stronger version of Theorem \ref{theorem:typedn}.

\subsection{Levi factor notation}
\label{section:levidn}

Let $\bG$ be a reductive group with $\bPhik(\bG) = \dD_n$ for
some $n \geq 4$.  We introduce the following convention:

\begin{convention}
\label{convention:smalld}
Usually $\dD_n$ is only defined for $n \geq 4$, but to allow
uniform statements we will define $\dD_3 = \dA_3$.  We do not
define $\dD_n$ for $n \leq 2$.
\end{convention}

Let $\bDelta = \bDeltak(\bG)$ be the set of simple roots of $\bPhik(\bG)$.  Number the
elements of $\bDelta$ as follows:\\ 
\centerline{\psfig{file=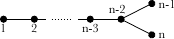,scale=2}}
For $1 \leq j_1,\ldots,j_{\ell} \leq n$, let $\bDelta[j_1, \ldots, j_{\ell}]$ be the result of removing
the simple roots labeled $j_1,\ldots,j_{\ell}$ from $\bDelta$.  We thus have a standard Levi
subgroup $\bL_{\bDelta[j_1, \ldots, j_{\ell}]}$ of $\bG$.

\begin{example}
\label{example:levidn}
We have $\bPhik(\bL_{\bDelta[1]}) = \dD_{n-1}$ and
\[\bPhik(\bL_{\bDelta[n]}) = \bPhik(\bL_{\bDelta[n-1]}) = \dA_{n-1}\]
and $\bPhik(\bL_{\bDelta[n-2]}) = \dA_{n-3} \times \dA_1 \times \dA_1$, while for
$2 \leq j \leq n-3$ we have $\bPhik(\bL_{\bDelta[j]}) = \dA_{j-1} \times \dD_{n-j}$.  In general,
for distinct $1 \leq j_1,\ldots,j_{\ell} \leq n$ we have either
\begin{align*}
&\bPhik(\bL_{\bDelta[j_1, \ldots, j_{\ell}]}) = \dA_{n_1} \times \cdots \times \dA_{n_{m-1}} \times \dD_{n_m} \quad \text{or} \\
&\bPhik(\bL_{\bDelta[j_1, \ldots, j_{\ell}]}) = \dA_{n_1} \times \cdots \times \dA_{n_{m}}
\end{align*}
with $n_1 + \cdots + n_m + \ell = n$.
\end{example}

We have a Reeder map (cf.\ \S \ref{section:reedermap}) of the form $\St(\bL_{\bDelta[j_1, \ldots, j_{\ell}]}) \rightarrow \St(\bG)$, and
thus maps $\HH_i(\bL_{\bDelta[j_1, \ldots, j_{\ell}]}(k);\St(\bL_{\bDelta[j_1, \ldots, j_{\ell}]})) \rightarrow \HH_i(\bG(k);\St(\bG))$.

\subsection{Strong vanishing}

The main result we will prove in this part of the paper is:

\begin{primedtheorem}{theorem:typedn}
\label{theorem:strongdn}
Let $\bG$ be a reductive group with $\bPhik(\bG) = \dD_n$ for some $n \geq 4$.  Then:
\begin{itemize}
\item $\HH_i(\bG(k);\St(\bG)) = 0$ for $i \leq \lfloor (n-3)/2 \rfloor$; and
\item letting $\bDelta = \bDeltak(\bG)$, the map
$\HH_i(\bL_{\bDelta[1]}(k);\St(\bL_{\bDelta[1]})) \rightarrow \HH_i(\bG(k);\St(\bG))$
is surjective for $i \leq \lfloor (n-2)/2 \rfloor$.
\end{itemize}
\end{primedtheorem}

This strengthens Theorem \ref{theorem:typedn} 
 by adding the indicated surjectivity statement. 
Unlike in other types, we do not need to handle the easiest case (that is, $n=4$) separately.  We will
therefore continue to assume that $n \geq 4$.  We will also
assume as an inductive hypothesis that we have already proved Theorem \ref{theorem:strongdn} in smaller ranks.  For this, we make the following
definition:

\begin{definition}
\label{definition:hypothesisdn}
For $r \geq 3$, the {\em $r$-surjectivity and vanishing hypothesis in type $\dD$} is as follows.
If $r \leq 3$, then it is trivial.\footnote{Groups of type $\dD_3 = \dA_3$ are covered by Theorem \ref{theorem:strongan},
so there is no need to assume anything about them.}  Otherwise, if $r \geq 4$ then
let $\bG$ be a reductive group with $\bPhik(\bG) = \dD_n$ for some $4 \leq n \leq r$.  Then:
\begin{itemize}
\item $\HH_i(\bG(k);\St(\bG)) = 0$ for $i \leq \lfloor (n-3)/2 \rfloor$; and
\item letting $\bDelta = \bDeltak(\bG)$, the map
$\HH_i(\bL_{\bDelta[1]}(k);\St(\bL_{\bDelta[1]})) \rightarrow \HH_i(\bG(k);\St(\bG))$
is surjective for $i \leq \lfloor (n-2)/2 \rfloor$.\qedhere
\end{itemize}
\end{definition}

\begin{remark}
\label{remark:alreadyandn}
We have already proven Theorem \ref{theorem:strongan}, so we also have available to us vanishing and surjectivity results
in type $\dA$.  We have also already proven Theorem \ref{theorem:strongbcn} about types $\dB$ and $\dC$ and $\dBC$, but
we will not need this.
\end{remark}

\section{Vanishing and surjectivity for Levi subgroups (type \texorpdfstring{$\dD$}{D})}
\label{section:levivanishdn}

In this section, we show how to use the $r$-surjectivity and vanishing hypothesis
in type $\dD$ to analyze the homology of standard Levi subgroups.  Our main result
is as follows.  Its statement uses the ordering on the simple roots of
of $\dA_{n_{j_0}}$ and $\dD_{n_{j_0}}$ discussed in \S \ref{section:levian} and \S \ref{section:levidn}.

\begin{lemma}[Levi vanishing and surjectivity]
\label{lemma:levivanishdn}
For some $n \geq 4$, assume the $(n-1)$-surjectivity and vanishing hypothesis in type $\dD$ (Definition \ref{definition:hypothesisdn}).
Let $\bG$ be a reductive group with $\bPhik(\bG) = \dD_n$.  Let $\Delta \subset \bDeltak(\bG)$ be a set
of simple roots with $\Delta \neq \bDeltak(\bG)$.  Write
\[\bPhik(\bL_{\Delta}) = \dA_{n_1} \times \cdots \times \dZ_{n_m} \quad \text{with $\dZ \in \{\dA,\dD\}$}.\]
We thus have $n_1,\ldots,n_{m-1} \geq 1$, and if $\dZ = \dD$ then $n_m \geq 4$ while if $\dZ = \dA$ then $n_m \geq 1$.
Set $b = \bb(\bPhik(\bL_{\Delta}))$, so
\[b = 
\begin{cases}
(m-1) + \lfloor (n_1-1)/2 \rfloor + \cdots + \lfloor(n_{m-1}-1)/2 \rfloor +  \lfloor(n_m-1)/2 \rfloor & \text{if $\dZ = \dA$},\\
(m-1) + \lfloor (n_1-1)/2 \rfloor + \cdots + \lfloor(n_{m-1}-1)/2 \rfloor +  \lfloor(n_m-3)/2 \rfloor & \text{if $\dZ = \dD$}.
\end{cases}\]
Then the following hold:
\begin{itemize}
\item[(i)] We have $\HH_i(\bL_{\Delta}(k);\St(\bL_{\Delta})) = 0$ for $i \leq b$.
\item[(ii)] For some $1 \leq j_0 \leq n$, assume one of the following:
\begin{itemize}
\item $n_{j_0}$ is even and nonzero, $1 \leq j_0 \leq m-1$, and $\Delta' \subset \Delta$
is the set of simple roots obtained by removing either the first or last simple root from $\dA_{n_{j_0}}$.
\item $n_{j_0}$ is even and nonzero, $j_0 = m$ and $\dZ = \dA$, and $\Delta' \subset \Delta$
is the set of simple roots obtained by removing either the first or last simple root from $\dA_{n_{n_m}}$.
\item $n_{j_0}$ is even and nonzero, $j_0 = m$ and $\dZ = \dD$, and $\Delta' \subset \Delta$ is the set
of simple roots obtained by removing the first simple root from $\dZ_{n_m}$.
\end{itemize}
Then the map $\HH_{b+1}(\bL_{\Delta'}(k);\St(\bL_{\Delta'})) \rightarrow \HH_{b+1}(\bL_{\Delta}(k);\St(\bG))$
is surjective.
\end{itemize}
\end{lemma}
\begin{proof}
Since $\Delta \neq \bDeltak(\bG)$, we have $n_j \leq n-1$ for $1 \leq j \leq m$.  The
$(n-1)$-surjectivity and vanishing hypothesis in type $\dD$ thus applies to all reductive
groups $\bH_m$ with $\bPhik(\bH_m) = \dD_{n_m}$.  Theorem \ref{theorem:strongan} also gives a vanishing and surjectivity
result for all reductive groups $\bH_j$ with $\bPhik(\bH_j) = \dA_{n_j}$.  This gives the hypothesis \eqref{vanishing1} in
Lemma \ref{lemma:reduciblevanishing} (reducible vanishing).  Applying Lemma \ref{lemma:reduciblevanishing},
we deduce (i).  Similarly, for $\Delta'$ as in (ii) it gives the hypotheses \eqref{surjectivity1} and \eqref{surjectivity2}
in Lemma \ref{lemma:reduciblesurjectivity} (reducible surjectivity).  Applying Lemma \ref{lemma:reduciblesurjectivity},
we deduce (ii).
\end{proof}

\section{Vanishing region (type \texorpdfstring{$\dD$}{D})}
\label{section:ssvanishdn}

Let $\bG$ be a reductive group with $\bPhik(\bG) = \bD_n$ for some $n \geq 4$.
Corollary \ref{corollary:spectralsequence} gives a spectral sequence $\ssE^r_{pq}$ converging
to $\HH_{p+q}(\bG(k);\St(\bG))$ with
\[\ssE^1_{pq} \cong \begin{cases}
\bigoplus_{R \in \cL_p(\bG)} \HH_q(\bL_{R}(k);\St(\bL_{R})) & \text{if $0 \leq p \leq n-1$} \\
\HH_q(\bG(k);\St(\bG)^{\otimes 2})                          & \text{if $p = n$},\\
0                                                           & \text{otherwise}.
\end{cases}\]
The following lemma shows that our inductive hypothesis
implies that many terms of this spectral sequence vanish.

\begin{lemma}
\label{lemma:vanishdn}
Let $\bG$ be a reductive group with $\bPhik(\bG) = \bD_n$ for some $n \geq 4$.
Assume the $(n-1)$-surjectivity and vanishing hypothesis in type $\dD$ (Definition \ref{definition:hypothesisdn}).
Let $\ssE^1_{pq}$ be the spectral sequence from Corollary \ref{corollary:spectralsequence}.  Then the following hold:
\begin{itemize}
\item For $n = 2d+3$ with $d \geq 1$, we have $\ssE^1_{pq} = 0$ for $p+q \leq d$ except for possibly
$\ssE^1_{0d}$.
\item For $n = 2d+2$ with $d \geq 1$, we have $\ssE^1_{pq} = 0$ for $p+q \leq d$ except for possibly
$\ssE^1_{0d}$ and $\ssE^1_{1,d-1}$.  For $d=1$, we also have $\ssE^1_{1,d-1} = 0$.
\end{itemize}
\end{lemma}
\begin{proof}
This is identical to the proof of Lemma \ref{lemma:vanishan} in type $\dA$.  The only difference is
that the standard Levi subgroups can be of two types:
\begin{itemize}
\item $\dA_{n_1} \times \cdots \times \dA_{n_m}$, where the vanishing range for $\HH_i$ given by Theorem \ref{theorem:strongan} is
\[i \leq (m-1) + \lfloor (n_1-1)/2 \rfloor + \cdots + \lfloor (n_{m}-1)/2 \rfloor.\]
\item $\dA_{n_1} \times \cdots \times \dA_{n_{m-1}} \times \dD_{n_m}$, where the vanishing range for $\HH_i$
given by Lemma \ref{lemma:levivanishdn} (Levi vanishing and surjectivity) is
\[i \leq (m-1) + \lfloor (n_1-1)/2 \rfloor + \cdots + \lfloor (n_{m-1}-1)/2 \rfloor + \lfloor (n_m - 3)/2 \rfloor.\]
\end{itemize}
The slightly worse vanishing range in the second case accounts for the slightly worse range in the statement
of the lemma.  Finally, the fact that $\ssE^1_{1,d-1} = 0$ for $d=1$ follows from the fact
that $\ssE^1_{1,0}$ is a direct sum of $\HH_0$-groups, and these all vanish by Lemma \ref{lemma:h0}.
\end{proof}

\section{Remaining tasks (type \texorpdfstring{$\dD$}{D})}
\label{section:proofdn}

Lemma \ref{lemma:vanishdn} implies many cases of Theorem \ref{theorem:strongdn}.  To prove the
remaining cases, we need to compute some differentials in our spectral sequence.  We now explain
the structure of the argument, postponing three calculations to the  
next section.  Recall that Theorem \ref{theorem:strongdn} is:

\theoremstyle{plain}
\newtheorem*{theorem:strongdn}{Theorem \ref{theorem:strongdn}}
\begin{theorem:strongdn}
Let $\bG$ be a reductive group with $\bPhik(\bG) = \dD_n$ for some $n \geq 4$.  Then:
\begin{itemize}
\item $\HH_i(\bG(k);\St(\bG)) = 0$ for $i \leq \lfloor (n-3)/2 \rfloor$; and
\item letting $\bDelta = \bDeltak(\bG)$, the map
$\HH_i(\bL_{\bDelta[1]}(k);\St(\bL_{\bDelta[1]})) \rightarrow \HH_i(\bG(k);\St(\bG))$
is surjective for $i \leq \lfloor (n-2)/2 \rfloor$.
\end{itemize}
\end{theorem:strongdn}
\begin{proof}
The proof is by induction on $n$.  We will prove the base case $n=4$ in exactly the same
way we will prove the inductive step, so assume that $n \geq 4$ and that we have
proved the result in smaller ranks, i.e., that
the $(n-1)$-surjectivity and vanishing hypothesis in type $\dD$ holds.  This assumption
is vacuous if $n=4$.

Corollary \ref{corollary:spectralsequence} gives a spectral sequence $\ssE^r_{pq}$ converging
to $\HH_{p+q}(\bG(k);\St(\bG))$, and Lemma \ref{lemma:vanishdn} implies that
$\ssE^1_{pq} = 0$ for $p+q \leq \lfloor (n-2)/2 \rfloor - 1$.  This implies that
$\HH_i(\bG(k);\St(\bG)) = 0$ for $i \leq \lfloor (n-2)/2 \rfloor - 1$.
Since our surjectivity claim is trivial when the target is $0$, all that remains to prove
are the following two claims:

\begin{claim}{1}
Assume that $n = 2d+3$ with $d \geq 1$.  Then $\HH_d(\bG(k);\St(\bG)) = 0$.
\end{claim}

In this case, Lemma \ref{lemma:vanishdn} says that the only potentially nonzero
term $\ssE^1_{pq}$ in our spectral sequence with $p+q = d$ is $\ssE^1_{0d}$.  We will prove in
Lemma \ref{lemma:differential2dn} below that the differential $\ssE^1_{1d} \rightarrow \ssE^1_{0d}$
is surjective, so $\ssE^2_{0d} = 0$.  This implies that $\HH_d(\bG(k);\St(\bG)) = 0$, as desired.

\begin{claim}{2}
Assume that $n = 2d+2$ with $d \geq 1$.  Then the map
\[\HH_d(\bL_{\bDelta[1]}(k);\St(\bL_{\bDelta[1]})) \rightarrow \HH_d(\bG(k);\St(\bG))\]
is surjective.
\end{claim}

Lemma \ref{lemma:vanishdn} says that the only potentially nonzero
terms $\ssE^1_{pq}$ in our spectral sequence with $p+q = d$ are $\ssE^1_{0d}$ and $\ssE^1_{1,d-1}$.  Lemma \ref{lemma:vanishdn}
also says that $\ssE^1_{1,d-1}=0$ if $d = 1$, and we
will prove in Lemma \ref{lemma:differential3dn} below that the differential $\ssE^1_{2,d-1} \rightarrow \ssE^1_{1,d-1}$
is surjective for $d \geq 2$.  It follows that in all cases $\ssE^2_{1,d-1} = 0$.  We will also prove in Lemma \ref{lemma:differential1dn}
below that the summand $\HH_d(\bL_{\bDelta[1]}(k);\St(\bL_{\bDelta[1]}))$ of
\[\ssE^1_{0d} = \bigoplus_{R \in \cL_0(\bG)} \HH_d(\bL_{R}(k);\St(\bL_{R})) = \bigoplus_{j=1}^{n} \HH_d(\bL_{\bDelta[j]}(k);\St(\bL_{\bDelta[j]}))\]
surjects onto the cokernel of the differential $\ssE^1_{1d} \rightarrow \ssE^1_{0d}$.
It follows that $\ssE^2_{0d}$ is a quotient of $\HH_d(\bL_{\bDelta[1]}(k);\St(\bL_{\bDelta[1]}))$.  Since $\ssE^2_{0d}$ is the
only potentially nonzero term of the form $\ssE^2_{pq}$ with $p+q = d$, it follows that $\HH_d(\bL_{\bDelta[1]}(k);\St(\bL_{\bDelta[1]}))$
surjects onto $\HH_d(\bG(k);\St(\bG))$, as desired.
\end{proof}

\section{Differentials (type \texorpdfstring{$\dD$}{D})} 
\label{section:differentialdn}

This final section of this part of the paper
determines the images of three differentials whose calculations were needed in the previous section.

\subsection{Differentials, I (type \texorpdfstring{$\dD$}{D})}
\label{section:differential1dn}

Our first differential calculation is:

\begin{lemma}
\label{lemma:differential1dn}
Let $\bG$ be a reductive group with $\bPhik(\bG) = \dD_{2d+2}$ for some $d \geq 1$.
Assume the $(2d+1)$-surjectivity and vanishing hypothesis in type $\dD$ (Definition \ref{definition:hypothesisdn}).
Let $\ssE^1_{pq}$ be the spectral
sequence from Corollary \ref{corollary:spectralsequence}.  Then the summand $\HH_d(\bL_{\bDelta[1]}(k);\St(\bL_{\bDelta[1]}))$
of $\ssE^1_{0d}$ surjects onto the cokernel of the
differential $\ssE^1_{1d} \rightarrow \ssE^1_{0d}$.
\end{lemma}
\begin{proof}
As notation, for $1 \leq j_1,\ldots,j_{\ell} \leq 2d+2$ let
\[M[j_1, \ldots, j_{\ell}] = \HH_d(\bL_{\bDelta[j_1, \ldots, j_{\ell}]}(k);\St(\bL_{\bDelta[j_1, \ldots, j_{\ell}]})).\]
We have
\[\ssE^1_{0d} = \bigoplus_{1 \leq j \leq 2d+2} M[j] \quad \text{and} \quad 
\ssE^1_{1d} = \bigoplus_{1 \leq j_1 < j_2 \leq 2d+2} M[j_1, j_2].\]
Consider some $1 < j \leq 2d+2$.  We must prove that when we quotient $\ssE^1_{0d}$ by the image of the differential
$\ssE^1_{1d} \rightarrow \ssE^1_{0d}$, the summand
$M[j]$ of $\ssE^1_{0d}$ is identified with a subspace of
$M[1]$.

We first prove that $M[j] = 0$ for $2d-1 \leq j \leq 2d+2$.  As in Example \ref{example:levidn}, we have
\begin{align*}
\bPhik(\bL_{\bDelta[2d-1]}) &= \dA_{(2d+2)-4} \times \dA_3 = \dA_{2d-2} \times \dA_3,\\
\bPhik(\bL_{\bDelta[2d]})   &= \dA_{(2d+2)-3} \times \dA_1 \times \dA_1 = \dA_{2d-1} \times \dA_1 \times \dA_1,\\
\bPhik(\bL_{\bDelta[2d+1]}) = \bPhik(\bL_{\bDelta[2d+2]}) &= \dA_{(2d+2)-1} = \dA_{2d+1}.
\end{align*}
In all of these cases, we can appeal to Theorem \ref{theorem:strongan} to see that
$\HH_i(\bL_{\bDelta[j]}(k);\St(\bL_{\bDelta[j]})) = 0$ for $i \leq d$.  Indeed, the bound
given by that theorem in the three cases above are
\begin{align*}
&1 + \lfloor (2d-3)/2 \rfloor + \lfloor (3-1)/2 \rfloor = d,\\ 
&2 + \lfloor (2d-2)/2 \rfloor + \lfloor (1-1)/2 \rfloor + \lfloor (1-1)/2 \rfloor = d+1,\\
&\lfloor 2d/2 \rfloor = d.
\end{align*}
The case $i=d$ shows that $M[j] = 0$, as desired.

It remains to prove that for $1 < j < 2d-2$, the summand $M[j]$ of $\ssE^1_{0d}$ is identified with a subspace of
$M[1]$ when you quotient by the image of the differential $\ssE^1_{1d} \rightarrow \ssE^1_{0d}$.  In these
cases, we have
\[\bPhik(\bL_{\bDelta[j]}) = \dA_{j-1} \times \dD_{2d+2-j}.\]
The proof for these cases is identical to the proof of Lemma \ref{lemma:differential1an} in type $\dA$.
\end{proof}

\subsection{Differentials, II (type \texorpdfstring{$\dD$}{D})}
\label{section:differential2dn}

Our second differential calculation is:

\begin{lemma}
\label{lemma:differential2dn}
Let $\bG$ be a reductive group with
$\bPhik(\bG) = \bD_{2d+3}$ for $d \geq 1$.
Assume the $(2d+2)$-surjectivity and vanishing hypothesis in type $\dD$ (Definition \ref{definition:hypothesisdn}).
Let $\ssE^1_{pq}$ be the spectral
sequence from Corollary \ref{corollary:spectralsequence}.  Then the
differential $\ssE^1_{1d} \rightarrow \ssE^1_{0d}$ is surjective.
\end{lemma}
\begin{proof}
As notation, for $1 \leq j_1,\ldots,j_{\ell} \leq 2d+3$ let
\[M[j_1, \ldots, j_{\ell}] = \HH_d(\bL_{\bDelta[j_1, \ldots, j_{\ell}]}(k);\St(\bL_{\bDelta[j_1, \ldots, j_{\ell}]})).\]
We have
\[\ssE^1_{0d} = \bigoplus_{1 \leq j \leq 2d+3} M[j] \quad \text{and} \quad 
\ssE^1_{1d} = \bigoplus_{1 \leq j_1 < j_2 \leq 2d+3} M[j_1, j_2].\]
Consider some $1 \leq j \leq 2d+3$.  We must prove that when we quotient $\ssE^1_{0d}$ by the image of the differential
$\ssE^1_{1d} \rightarrow \ssE^1_{0d}$, the summand $M[j]$ is killed.

Just like in the proof of Lemma \ref{lemma:differential1dn} above, we have
$M[j] = 0$ for $2d \leq j \leq 2d+3$.  We must therefore only deal with $M[j]$ for
$1 \leq j \leq 2d-1$.  For these $j$, we have
\[\bPhik(\bL_{\bDelta[j]}) = \dA_{j-1} \times \dD_{2d+3-j}.\]
The proof that these $M[j]$ are killed when we quotient by the image of the differential is identical to the proof of Lemma \ref{lemma:differential2an} in type $\dA$.
\end{proof}

\subsection{Differentials, III (type \texorpdfstring{$\dD$}{D})}
\label{section:differential3dn}

Our final differential calculation is:

\begin{lemma}
\label{lemma:differential3dn}
Let $\bG$ be a reductive group with
$\bPhik(\bG) = \bD_{2d+2}$ for some $d \geq 2$.
Assume the $(2d+1)$-surjectivity and vanishing hypothesis in type $\dD$ (Definition \ref{definition:hypothesisdn}).
Let $\ssE^1_{pq}$ be the spectral
sequence from Corollary \ref{corollary:spectralsequence}.  Then the differential
$\ssE^1_{2,d-1} \rightarrow \ssE^1_{1,d-1}$ is surjective.
\end{lemma}
\begin{proof}
As notation, for $1 \leq j_1,\ldots,j_{\ell} \leq 2d+2$ let
\[M[j_1, \ldots, j_{\ell}] = \HH_{d-1}(\bL_{\bDelta[j_1, \ldots, j_{\ell}]}(k);\St(\bL_{\bDelta[j_1, \ldots, j_{\ell}]})).\]
We have
\[\ssE^1_{1,d-1} = \bigoplus_{1 \leq j_1 < j_2 \leq 2d+2} M[j_1, j_2] \quad \text{and} \quad 
\ssE^1_{2,d-1} = \bigoplus_{1 \leq j_1 < j_2 < j_3 \leq 2d+2} M[j_1, j_2, j_3].\]
Consider some $1 \leq j_1 < j_2 \leq 2d+2$.  We must prove that when we quotient $\ssE^1_{1,d-1}$ by the image of the differential
$\ssE^1_{2,d-1} \rightarrow \ssE^1_{1,d-1}$, the summand $M[j_1, j_2]$ is killed.
We first prove:

\begin{claim}{1}
\label{claim:differential3dn.1}
If $1 \leq j_1 < j_2 \leq 2d+2$ are such that $2d-1 \leq j_2 \leq 2d+2$, then $M[j_1, j_2] = 0$.
\end{claim}
\begin{proof}[Proof of claim]
In this case,
\[\bPhik(\bL_{\bDelta[j_1, j_2]}) = \dA_{n_1} \times \cdots \times \dA_{n_m} \quad \text{with $n_1 + \cdots + n_m = 2d$}.\]
Theorem \ref{theorem:strongan} therefore says that $\HH_i(\bL_{\bDelta[j_1, j_2]}(k);\St(\bL_{\bDelta[j_1, j_2]})) = 0$
for
\[i \leq (m-1) + \lfloor (n_1-1)/2 \rfloor + \cdots + \lfloor (n_m-1)/2 \rfloor.\]
Lemma \ref{lemma:floorinequality} says that for $a,b \in \Z$ we have
$1+ \lfloor a/2 \rfloor + \lfloor b/2 \rfloor \geq \lfloor (a+b+1)/2 \rfloor$.  Applying this repeatedly,
we see that our vanishing range is at least
$\lfloor (2d-1)/2 \rfloor = d-1$.
We deduce that $M[j_1, j_2] = \HH_{d-1}(\bL_{\bDelta[j_1, j_2]}(k);\St(\bL_{\bDelta[j_1, j_2]})) = 0$.
\end{proof}

In light of Claim \ref{claim:differential3dn.1}, we must show that for $1 \leq j_1 < j_2 \leq 2d-2$, the
summand $M[j_1, j_2]$ is killed by the differential $\ssE^1_{2,d-1} \rightarrow \ssE^1_{1,d-1}$.  Using
the convention that $\dA_0$ is the empty root system, for $1 \leq j_1 < j_2 \leq 2d-2$ we have
\[\bPhik(\bL_{\bDelta[j_1, j_2]}) = \dA_{j_1-1} \times \dA_{j_2-j_1-1} \times \dD_{2d+2-j_2}.\]
The proof is identical to the proof of Lemma \ref{lemma:differential3an} in type $\dA$; in fact, it is even
easier than Lemma \ref{lemma:differential3an} since the extremal case $M[1,2d-2]$ can be treated exactly
the same as the case $M[1,j_2]$ with $1<j_2<2d-2$ even.
\end{proof}

\part{Integral vanishing (Theorem \ref{maintheorem:doublecomvanish})}
\label{part:integral}

We now generalize our vanishing theorems to 
the groups\footnote{We use $\SL_{n+1}$ and $\GL_{n+1}$ since
they have type $\dA_n$ and we want our numerology to match Part \ref{part:an}.}
$\SL_{n+1}(\Z)$ and $\GL_{n+1}(\Z)$, proving Theorem \ref{maintheorem:doublecomvanish}.
This requires a conjectural partial resolution of their Steinberg representations, which 
we discuss in \S \ref{section:partialresolution}.  The existence
of this conjectural resolution is equivalent to the high connectivity of the double
Tits building.  We discuss some results about reducible
Levi subgroups in \S \ref{section:integrallevi} and state a stronger version of
Theorem \ref{maintheorem:doublecomvanish} in \S \ref{section:strongintegral}.  The
rest of this part is devoted to its proof, which follows the outline 
of our proof in type $\dA$ from Part \ref{part:an}.  

\begin{convention}
Throughout this part, $\SL_{n+1}$ and $\GL_{n+1}$ are always taken to be defined over the field $\Q$.  
\end{convention}

\section{Partial resolution of integral Steinberg representation}
\label{section:partialresolution}

The key technical tool underlying our proof of Theorem \ref{maintheorem:fields} was 
the spectral sequence from Corollary \ref{corollary:spectralsequence}, which comes
from the resolution of the Steinberg representation given by Proposition \ref{proposition:resolution}.
This section explains a conjectural integral refinement of this.

\subsection{Notation}

We start by introducing some notation we will use throughout this part of the paper.  The
Tits buildings of $\SL_{n+1}$ and $\GL_{n+1}$ are isomorphic, and we will denote them by
$\Tits(\Q^{n+1})$.  More generally, for a nonzero finite-dimensional $\Q$-vector space $V$ we will
write $\Tits(V)$ for the complex of flags of nonzero proper subspaces of $V$.

Let $d = \dim(V) \geq 1$.  The Solomon--Tits theorem says that $\Tits(V)$ is homotopy equivalent to a wedge
of $(d-2)$-dimensional spheres.  For a commutative ring $\bbF$, we let
$\St(V;\bbF) = \RH_{d-2}(\Tits(V);\bbF)$.  For $\bbF = \Z$, we omit
$\bbF$ from our notation and write $\St(V) = \RH_{d-2}(\Tits(V))$.  The
groups $\GL(V)$ and $\SL(V)$ act on $\St(V;\bbF)$.  With this 
notation, $\St(\Q^{n+1};\bbF)$ is the Steinberg representation of $\GL_{n+1}(\Q)$ and
$\SL_{n+1}(\Q)$.

Now let $W$ be a finite-rank free $\Z$-module.  Write
$\Tits(W)$ for the complex of flags of nonzero proper direct summands of $W$.
Direct summands of $W$ are in bijection with subspaces of $W \otimes \Q$, so
$\Tits(W) = \Tits(W \otimes \Q)$.  Letting $r = \rank(W)$, we define
$\St(W;\bbF) = \RH_{r-2}(\Tits(W);\bbF)$ and $\St(W) = \RH_{r-2}(\Tits(W))$.
We thus have $\St(W;\bbF) \cong \St(W \otimes \Q;\bbF) \cong \St(\Q^r;\bbF)$.

\subsection{Resolution over rationals}

Recall from Example \ref{example:resolutiongl} that the resolution of $\St(\Q^{n+1})$ from
Proposition \ref{proposition:resolution} takes the form
\begin{equation}
\label{eqn:rationalresolution}
0 \rightarrow (\St(\Q^{n+1}))^{\otimes 2} \rightarrow \bR_{n-1} \rightarrow \cdots \rightarrow \bR_0 \rightarrow \St(\Q^{n+1}) \rightarrow 0
\end{equation}
with
\[\bR_i = \bigoplus_{V_1 \oplus \cdots \oplus V_{i+2} = \Q^{n+1}} \St(V_1) \otimes \cdots \otimes \St(V_{i+2})\]
for $0 \leq i \leq n-1$.  Here the direct sum is over decompositions
$V_1 \oplus \cdots \oplus V_{i+2} = \Q^{n+1}$ with $\dim(V_j) \geq 1$ for $1 \leq j \leq i+2$.  

\subsection{Integral refinement}
\label{section:integralrefinement}

Though it is easy to understand the actions of $\GL_{n+1}(\Q)$ and $\SL_{n+1}(\Q)$ on decompositions
$V_1 \oplus \cdots \oplus V_{i+2} = \Q^{n+1}$, the actions of their subgroups $\GL_{n+1}(\Z)$ and
$\SL_{n+1}(\Z)$ are far more complicated.  It would be helpful to restrict to integral decompositions
$W_1 \oplus \cdots \oplus W_{i+2} = \Z^{n+1}$.  If we do this, we are forced to remove
the initial term $(\St(\Q^{n+1}))^{\otimes 2}$.  Let
\begin{equation}
\label{eqn:integralresolution}
\begin{tikzcd}[column sep=small, row sep=small]
\bS_{n-1}(\Z^{n+1}) \arrow{r} & \cdots \arrow{r} & \bS_0(\Z^{n+1}) \arrow{r} & \bS_{-1}(\Z^{n+1}) \arrow[equals]{d} \arrow{r} & 0 \\
                              &                  &                       & \St(\Z^{n+1})                              &
\end{tikzcd}
\end{equation}
be the subcomplex of \eqref{eqn:rationalresolution} with 
\[\bS_i(\Z^{n+1}) = \bigoplus_{W_1 \oplus \cdots \oplus W_{i+2} = \Z^{n+1}} \St(W_1) \otimes \cdots \otimes \St(W_{i+2}) \quad \text{for $-1 \leq i \leq n-1$}.\]
We conjecture that more and more of $\bS_{\bullet}^{n+1}$ becomes
exact as $n$ increases.  The bound in the following more precise
version of this conjecture is exactly what we need to prove
our vanishing theorem:

\begin{conjecture}
\label{conjecture:integralresolution}
For $b \geq 1$, the {\em $b$-integral resolution conjecture} says the following.  Consider some $n \geq 3$.
Let $b' = \min(b,\lfloor n/2 \rfloor)$.  Then the portion
\[\begin{tikzcd}[column sep=small, row sep=small]
\bS_{b'}(\Z^{n+1}) \arrow{r} & \bS_{b'-1}(\Z^{n+1}) \arrow{r} & \cdots \arrow{r} & \bS_0(\Z^{n+1}) \arrow{r} & \bS_{-1}(\Z^{n+1}) \arrow[equals]{d} \arrow{r} & 0 \\
                      &                           &                  &                       & \St(\Z^{n+1})                              &
\end{tikzcd}\]
of the chain complex \eqref{eqn:integralresolution} is exact.
\end{conjecture}

\begin{remark}
The exactness of the entire complex \eqref{eqn:integralresolution} for all $n \geq 1$ is equivalent to the Koszulness of the Steinberg
monoid of the integers as defined in \cite{MillerNagpalPatzt}.  By the results of \cite{MillerPatztWilsonRognes}, this
would imply the Rognes connectivity conjecture \cite[Conjecture 12.3]{RognesConjecture}.
\end{remark}

\subsection{High connectivity of the double Tits building}

We defined the double Tits building $\Tits^2(\Z^{n+1})$ in \S \ref{section:doubletits}.  Since
we will not use it directly, we will not recall its definition here.  The relevance
of $\Tits^2(\Z^{n+1})$ for us is the following theorem of Miller--Patzt--Wilson \cite{MillerPatztWilsonRognes}.

\begin{theorem}
\label{theorem:doubletitshomology}
For all $n \geq 1$, we have $\HH_{i}(\bS_{\bullet}(\Z^{n+1})) = \RH_{i+n}(\Tits^2(\Z^{n+1}))$.
\end{theorem}
\begin{proof}
Here is how to extract this from \cite{MillerPatztWilsonRognes}.  We remark that we will give a self-contained and
somewhat different proof in \cite{MillerPatztPutmanAlternate}.  We will use the notation from \cite{MillerPatztWilsonRognes}.
First, $\bS_{\bullet}(\Z^{n+1})$ is the bar resolution from \cite{MillerNagpalPatzt} computing $\Tor^{\St(\Z)}_{i}(\Z,\Z)_{n+1}$, so
\begin{equation}
\label{eqn:doubletitshomology.1}
\HH_{i}(\bS_{\bullet}(\Z^{n+1})) = \Tor^{\St(\Z)}_{i+2}(\Z,\Z)_{n+1}.
\end{equation}
Here the $i+2$ appears in the Tor term due to a degree shift we are suppressing.
Next, \cite[Lemma 6.2]{MillerPatztWilsonRognes} shows that
\begin{equation}
\label{eqn:doubletitshomology.2}
\Tor^{\St(\Z)}_{i+2}(\Z,\Z)_{n+1} = \RH_{i+n+3}(D^{2,0}(\Z)),
\end{equation}
where $D^{2,0}(\Z)$ is one of the ``higher buildings'' defined in \cite{MillerPatztWilsonRognes}.  Finally, \cite[Lemma 4.22]{MillerPatztWilsonRognes}
proves that $D^{2,0}(\Z) \cong \Sigma^3 \Tits^2(\Z^{n+1})$, so
\begin{equation}
\label{eqn:doubletitshomology.3}
\RH_{i+n+3}(D^{2,0}(\Z)) \cong \RH_{i+n}(\Tits^2(\Z^{n+1})).
\end{equation}
The theorem follows from \eqref{eqn:doubletitshomology.1} and \eqref{eqn:doubletitshomology.2} and \eqref{eqn:doubletitshomology.3}.
\end{proof}

Theorem \ref{theorem:doubletitshomology} has the following immediate corollary:

\begin{corollary}
\label{corollary:doubletitsresolution}
For $b \geq 1$, the $b$-integral resolution conjecture holds if and only if the following holds for all $n \geq 3$:\noeqref{homologyasm}
\begin{equation}
\tag{$\dagger\dagger$}\label{homologyasm} \text{We have $\HH_{i}(\Tits^2(\Z^{n+1})) = 0$ for $i \leq n-1+\min(b,\lfloor n/2 \rfloor)$.}
\end{equation}
\end{corollary}

\subsection{Spectral sequence}

Our next goal is to extract a spectral sequence from Conjecture \ref{conjecture:integralresolution}.  Letting $\bG$ be
either $\GL_{n+1}$ or $\SL_{n+1}$, for each $R \subset \bDeltaQ(\bG)$ there is a standard Levi subgroup $\bL_{R}$.  The
group $\bL_{R}$ is defined over $\Z$, so $\bL_{R}(\Z)$ makes sense.  If $\bG = \GL_{n+1}$ then
\[\bL_{R}(\Z) = \GL_{n_1+1}(\Z) \times \cdots \times \GL_{n_m+1}(\Z)\]
for some $(n_1+1) + \cdots + (n_m+1) = n+1$, while if $\bG = \SL_{n+1}$ then 
\[\bL_{R}(\Z) = \ker(\GL_{n_1+1}(\Z) \times \cdots \times \GL_{n_m+1}(\Z) \stackrel{\det}{\longrightarrow} \Z^{\times}).\]
Recall from Corollary \ref{corollary:spectralsequence} that the resolution
\eqref{eqn:rationalresolution} yields a spectral sequence $\ssE^r_{pq}$ converging to $\HH_{p+q}(\bG(k);\St(\bG))$ with
\[\ssE^1_{pq} \cong \begin{cases}
\bigoplus_{R \in \cL_p(\bG)} \HH_q(\bL_{R}(k);\St(\bL_{R})) & \text{if $0 \leq p \leq n-1$} \\
\HH_q(\bG(k);\St(\bG)^{\otimes 2})                                         & \text{if $p = n$},\\
0                                                                          & \text{otherwise}.
\end{cases}\]
The following shows that Conjecture \ref{conjecture:integralresolution} yields a similar
spectral sequence for $\bG(\Z)$:

\begin{lemma}
\label{lemma:spectralsequenceintegral}
Assume the $b$-integral resolution conjecture (Conjecture \ref{conjecture:integralresolution}) for some $b \geq 1$.
Let $\bG$ be either $\GL_{n+1}$ or $\SL_{n+1}$, let $\bbF$ be a commutative ring, and let $n \geq 3$.  Set $c = \min(b,\lfloor n/2 \rfloor)$.
There is then a spectral sequence $\ssE^r_{pq}$ converging to
$\HH_{p+q}(\bG(\Z);\St(\Z^{n+1};\bbF))$ such that
\[\ssE^1_{pq} \cong \bigoplus_{R \in \cL_p(\bG)} \HH_q(\bL_{R}(\Z);\St(\bL_{R};\bbF)) \quad \text{if $0 \leq p \leq c$}.\]
\end{lemma}
\begin{proof}
Let $\bS_{\bullet} = \bS_{\bullet}(\Z^{n+1})$ be as in Conjecture \ref{conjecture:integralresolution}.  We suppress the $\Z^{n+1}$
to simplify our notation.  Let $\bS'_{c+1}$ be the kernel 
of the differential $\bS_c \rightarrow \bS_{c-1}$.  By assumption, we therefore have a resolution
\[0 \rightarrow \bS'_{c+1} \rightarrow \bS_c \rightarrow \cdots \rightarrow \bS_0 \rightarrow \St(\Z^{n+1}) \rightarrow 0.\]
Each term of this is a free abelian group, so by the universal coefficients theorem we can tensor this with $\bbF$ and
get a resolution 
\[0 \rightarrow \bS'_{c+1} \otimes \bbF \rightarrow \bS_c \otimes \bbF \rightarrow \cdots \rightarrow \bS_0 \otimes \bbF \rightarrow \St(\Z^{n+1};\bbF) \rightarrow 0.\]
As is standard, there is a spectral sequence converging to $\HH_{p+q}(\bG(\Z);\St(\Z^{n+1};\bbF))$ with
\[\ssE^1_{pq} \cong 
\begin{cases}
\HH_q(\bG(\Z);\bS_p \otimes \bbF)      & \text{if $0 \leq p \leq c$},\\
\HH_q(\bG(\Z);\bS'_{c+1} \otimes \bbF) & \text{if $p = c+1$},\\
0                         & \text{otherwise}.
\end{cases}\]
To prove the lemma, it is therefore enough to prove:

\begin{unnumberedclaim}
For $0 \leq p \leq n-1$, we have
$\HH_q(\bG(\Z);\bS_p \otimes \bbF) \cong \bigoplus_{R \in \cL_p(\bG)} \HH_q(\bL_{R}(\Z);\St(\bL_{R};\bbF))$.
\end{unnumberedclaim}

We have
\begin{align*}
\bS_p \otimes \bbF &= \bigoplus_{W_1 \oplus \cdots \oplus W_{p+2} = \Z^{n+1}} \St(W_1;\bbF) \otimes \cdots \otimes \St(W_{p+2};\bbF) \\
                   &= \bigoplus_{(n_1+1) + \cdots + (n_{p+2}+1) = n+1} \left(\bigoplus_{\substack{W_1 \oplus \cdots \oplus W_{p+2} = \Z^{n+1} \\ \rank(W_j) = n_j+1}} \St(W_1;\bbF) \otimes \cdots \otimes \St(W_{p+2};\bbF)\right).
\end{align*}
For a fixed partition $(n_1+1) + \cdots + (n_{p+2}+1) = n+1$, the group $\bG(\Z)$ acts transitively on decompositions
$W_1 \oplus \cdots \oplus W_{p+2} = \Z^{n+1}$ with $\rank(W_j) = n_j+1$ for all $1 \leq j \leq p+2$.  The stabilizer of
the standard decomposition $\Z^{n_1+1} \oplus \cdots \oplus \Z^{n_{p+2}+1} = \Z^{n+1}$ is the integer points of one of the standard
Levi subgroups $\bL_{R}$ of $\bG$: if $\bG = \GL_{n+1}$ the stabilizer is
\[\bL_{R}(\Z) = \GL_{n_1+1}(\Z) \times \cdots \times \GL_{n_{p+2}+1}(\Z),\]
while if $\bG = \SL_{n+1}$ the stabilizer is
\[\bL_{R}(\Z) = \ker(\GL_{n_1+1}(\Z) \times \cdots \times \GL_{n_{p+2}+1}(\Z) \stackrel{\det}{\longrightarrow} \Z^{\times}).\]
The term $\St(\Z^{n_1+1};\bbF) \otimes \cdots \otimes \St(\Z^{n_{p+2}+1};\bbF)$ is exactly
$\St(\bL_{R};\bbF)$.  We deduce that
\[\bS_p \otimes \bbF = \bigoplus_{R \in \cL_p(\bG)} \Ind_{\bL_{R}(\Z)}^{\bG(\Z)} \St(\bL_{R};\bbF).\]
The claim now follows from Shapiro's lemma.
\end{proof}

\section{Reducible Levi subgroups over the integers}
\label{section:integrallevi}

When we proved Theorem \ref{maintheorem:fields} for groups of type $\dA_n$ in Part \ref{part:an}, we
had to understand their standard Levi subgroups.  These subgroups have reducible relative
root systems of type $\dA_{n_1} \times \cdots \times \dA_{n_m}$, and we analyzed them
using general results about reducible root systems from \S \ref{section:reducible}.
This section discusses the analogue of these results for Levi subgroups of
$\GL_{n+1}(\Z)$ and $\SL_{n+1}(\Z)$.

\subsection{Setup}

The standard Levi factors of $\GL_{n+1}$ and $\SL_{n+1}$ are of the forms
\begin{align*}
\GL_{n_1+1} \times \cdots \times \GL_{n_m+1}& \quad\text{and} \\
\ker(\GL_{n_1+1} \times \cdots \times \GL_{n_m+1} \stackrel{\det}{\longrightarrow} \GL_1)&
\end{align*}
with $(n_1+1)+\cdots+(n_m+1) = n+1$.  We will call groups $\bG$ of these
two forms {\em standard $\dA$-integral Levi subgroups}.  Since such $\bG$ are defined
over $\Z$, the group $\bG(\Z)$ makes sense.  
To simplify our notation, we will let $\dA_0 = \{0\} \subset \R^0$, regarded as a trivial root system of 
rank $0$.  With this convention, a standard $\dA$-integral Levi subgroup $\bG$ is split with root
system
\[\bPhiQ(\bG) = \dA_{n_1} \times \cdots \times \dA_{n_m}.\]
The $\dA_{n_i}$ with $n_i = 0$ could be deleted from this.  

\subsection{Reduction to products}

Our first result about these groups is as follows.  It will serve as a substitute for
Lemmas  \ref{lemma:reducetosemisimple} and \ref{lemma:reducetoproduct} from \S \ref{section:reducible}.

\begin{lemma}
\label{lemma:reducetoproductintegral}
Let $\bG$ be a standard $\dA$-integral Levi subgroup.  Write
\[\bPhiQ(\bG) = \dA_{n_1} \times \cdots \times \dA_{n_m} \quad \text{with $n_1,\ldots,n_m \geq 1$}.\]
Set $\bH = \SL_{n_1+1} \times \cdots \times \SL_{n_m+1}$.  Let $\bbF$ be a commutative ring.  Then:
\begin{itemize}
\item[(i)] the group $\bH$ is a subgroup of $\bG$ and the inclusion $\bH \rightarrow \bG$ induces a bijection between parabolic subgroups, and thus
isomorphisms $\Tits(\bH) \cong \Tits(\bG)$ and $\St(\bH;\bbF) \cong \Res^{\bG(\Z)}_{\bH(\Z)} \St(\bG;\bbF)$.
\end{itemize}
If in addition for some $b \geq -1$ we have $\HH_i(\bH(\Z);\St(\bH))=0$ for $i \leq b$, then:
\begin{itemize}
\item[(ii)] $\HH_i(\bG(\Z);\St(\bG;\bbF)) = 0$ for $i \leq b$ and 
the map 
\[\HH_{b+1}(\bH(\Z);\St(\bH;\bbF)) \rightarrow \HH_{b+1}(\bG(\Z);\St(\bG;\bbF))\] 
is a surjection.
\end{itemize}
\end{lemma}
\begin{proof}
Recall that we can add or delete copies of $\dA_0$ from our root system without changing it.
Since $\SL_1$ is trivial, we can add corresponding $\SL_{0+1}$-factors to $\bH$ without changing it.  Adding
an appropriate number of $\dA_0$-factors to our root system,\footnote{We thus no longer have $n_j \geq 1$ for all $j$, but
this does not affect the truth of the lemma.} we can thus assume that
\[\bH = \SL_{n_1+1} \times \cdots \times \SL_{n_m+1} \subset \bG \subset \GL_{n_1+1} \times \cdots \times \GL_{n_m+1}.\]
This shows that $\bH$ is a subgroup of $\bG$.  The rest of (ii) is obvious.

For (ii), assume that for some $b \geq -1$ we have $\HH_i(\bH(\Z);\St(\bH))=0$ for $i \leq b$.  By construction, we have a short exact sequence
\[1 \longrightarrow \bH(\Z) \longrightarrow \bG(\Z) \longrightarrow A \longrightarrow 1\]
with $A$ abelian.  The associated Hochschild--Serre
spectral sequence with coefficients in $\St(\bH;\bbF) = \St(\bG;\bbF)$ is of the form
\[\ssE^2_{pq} = \HH_p(A;\HH_q(\bH(\Z);\St(\bH;\bbF))) \Rightarrow \HH_{p+q}(\bG(\Z);\St(\bG;\bbF)).\]
Our vanishing assumption implies that $\ssE^2_{pq} = 0$ for $q \leq b$, so
$\HH_i(\bG(\Z);\St(\bG;\bbF)) = 0$ for $i \leq b$ and
$\HH_{b+1}(\bG(\Z);\St(\bG;\bbF)) = \ssE^2_{0,b+1}$.  We deduce that
\[\HH_{b+1}(\bG(\Z);\St(\bG;\bbF)) = \HH_0(A;\HH_{b+1}(\bH(\Z);\St(\bH;\bbF))) = \HH_{b+1}(\bH(\Z);\St(\bH;\bbF))_A,\]
where the subscript indicates that we are taking coinvariants.  For any group $\Gamma$ and any
$\Gamma$-module $M$, the map $M \rightarrow M_{\Gamma}$ is a surjection.  In light of the above
identity, we deduce that the map $\HH_{b+1}(\bH(\Z);\St(\bH;\bbF)) \rightarrow \HH_{b+1}(\bG(\Z);\St(\bG;\bbF))$
is a surjection, as desired.  
\end{proof}

Lemma \ref{lemma:reducetoproductintegral} has the following corollary:

\begin{corollary}
\label{corollary:reducetoslintegral}
Let $n \geq 1$ and let $\bbF$ be a commutative ring.  For some $b \geq -1$, assume
that $\HH_i(\SL_{n+1}(\Z);\St(\Z^{n+1};\bbF)) = 0$ for $i \leq b$.  Then $\HH_i(\GL_{n+1}(\Z);\St(\Z^{n+1};\bbF)) = 0$.
\end{corollary}
\begin{proof}
Immediate from the case $m=1$ of Lemma \ref{lemma:reducetoproductintegral}.
\end{proof}

\subsection{Reducible root systems and vanishing}

The following result will serve as a substitute for Lemma \ref{lemma:reduciblevanishing} from \S \ref{section:reducible}:

\begin{lemma}[Reducible vanishing]
\label{lemma:reduciblevanishingintegral}
Let $n_1,\ldots,n_m \geq 1$ and let $\bbF$ be a PID.  For each $1 \leq j \leq m$, assume that there is some $b_j \geq -1$ such that
the following holds:\noeqref{vanishing1integral}
\begin{equation}
\tag{$\varheartsuit$}\label{vanishing1integral}
\parbox{35em}{We have $\HH_i(\SL_{n_j+1}(\Z);\St(\Z^{n_j+1};\bbF)) = 0$ for $i \leq b_j$.}
\end{equation}
Then for all standard $\dA$-integral Levi subgroups $\bG$ with
$\bPhiQ(\bG) = \dA_{n_1} \times \cdots \times \dA_{n_m}$,
we have $\HH_i(\bG(\Z);\St(\bG;\bbF)) = 0$ for $i \leq (m-1) + b_1 + \cdots + b_m$.
\end{lemma}
\begin{proof}
By Lemma \ref{lemma:reducetoproductintegral}, it is enough to prove this vanishing for\footnote{Note that for $m \geq 2$ this $\bG$ is not a standard $\dA$-integral Levi subgroup.}
$\bG = \SL_{n_1+1} \times \cdots \times \SL_{n_m+1}$.
The proof will be by induction on $m$.  The base case $m=1$ follows immediately from
\eqref{vanishing1integral}, so assume that $m \geq 2$ and that the result is true
whenever $m$ is smaller.

As notation, set $\bA = \SL_{n_1+1} \times \cdots \times \SL_{n_{m-1}+1}$ and
$\bB = \SL_{n_m}$, so $\bG = \bA \times \bB$.  By Lemma \ref{lemma:steinbergproduct}, we have
$\St(\bA \times \bB;\bbF) = \St(\bA;\bbF) \otimes \St(\bB;\bbF)$.
Since $\bbF$ is a PID and $\St(\bA;\bbF)$ and $\St(\bB;\bbF)$ are free $\bbF$-modules, the K\"{u}nneth formula applies and
shows that $\HH_i(\bA(\Z) \times \bB(\Z);\St(\bA \times \bB;\bbF))$ fits
into a short exact sequence with the following kernel and cokernel:
\begin{itemize}
\item $\bigoplus_{i_1 + i_2 = i} \HH_{i_1}(\bA(\Z);\St(\bA;\bbF)) \otimes \HH_{i_2}(\bB(\Z);\St(\bB;\bbF))$.
\item $\bigoplus_{i_1 + i_2 = i-1} \Tor(\HH_{i_1}(\bA(\Z);\St(\bA;\bbF)),\HH_{i_2}(\bB(\Z);\St(\bB;\bbF)))$.
\end{itemize}
By our induction hypothesis and \eqref{vanishing1integral}, we have
\begin{itemize}
\item $\HH_{i_1}(\bA(\Z);\St(\bA;\bbF)) = 0$ for $i_1 \leq (m-2) + b_1 + \cdots + b_{m-1}$; and
\item $\HH_{i_2}(\bB(\Z);\St(\bB;\bbF)) = 0$ for $i_2 \leq b_m$.
\end{itemize}
Note that if
\[i_1 + i_2 \leq ((m-2) + b_1 + \cdots + b_{m-1}) + b_m + 1 = (m-1) + b_1 + \cdots + b_m,\]
then either
\[i_1 \leq (m-2) + b_1 + \cdots + b_{m-1} \quad \text{or} \quad i_2 \leq b_m.\]
It follows that $\HH_i(\bA(\Z) \times \bB(\Z);\St(\bA \times \bB;\bbF)) = 0$
for $i \leq (m-1) + b_1 + \cdots + b_m$, as desired.
\end{proof}

\subsection{Reducible root systems and surjectivity}

The following result will serve as a substitute for Lemma \ref{lemma:reduciblesurjectivity} from \S \ref{section:reducible}:

\begin{lemma}[Reducible surjectivity]
\label{lemma:reduciblesurjectivityintegral}
Let $n_1,\ldots,n_m \geq 1$ and let $\bbF$ be a PID.  For each 
$1 \leq j \leq m$, assume that there is some $b_j \geq -1$ such that the following holds:\noeqref{surjectivity1integral}
\begin{equation} 
\tag{$\varheartsuit$}\label{surjectivity1integral}
\parbox{35em}{We have $\HH_i(\SL_{n_j+1}(\Z);\St(\Z^{n_j+1};\bbF)) = 0$ for $i \leq b_j$.}
\end{equation}
Additionally, for some $1 \leq j_0 \leq m$ assume there exists a set $\Delta'$ of simple roots
of $\dA_{j_0}$ such that the following holds:\noeqref{surjectivity2integral}
\begin{equation} 
\tag{$\varheartsuit\varheartsuit$}\label{surjectivity2integral}
\parbox{35em}{Let $\bL^{j_0}_{\Delta'}$ be the
standard Levi subgroup of $\SL_{n_{j_0}+1}$ corresponding to $\Delta'$.  Then the map
$\HH_{b_{j_0}+1}(\bL^{j_0}_{\Delta'}(\Z);\St(\bL^{j_0}_{\Delta'};\bbF)) \rightarrow \HH_{b_{j_0}+1}(\SL_{n_{j_0}+1}(\Z);\St(\Z^{n_{j_0}+1};\bbF))$
is surjective.}
\end{equation}
Let $\bG$ be a standard $\dA$-integral Levi subgroup with
$\bPhiQ(\bG) = \dA_{n_1} \times \cdots \times \dA_{n_m}$.
Let
\[\oDelta' = \Delta' \sqcup \bigsqcup_{\substack{1 \leq j \leq m \\ j \neq j_0}} \bDeltaQ(\bPhi_j),\]
and let $\bL^{\bG}_{\oDelta'}$ be the corresponding
standard Levi subgroup of $\bG$.  Then for
\[b = (m-1) + b_1 + \cdots + b_m\]
the map $\HH_{b+1}(\bL^{\bG}_{\oDelta'}(\Z);\St(\bL^{\bG}_{\oDelta'};\bbF)) \rightarrow \HH_{b+1}(\bG(\Z);\St(\bG;\bbF))$
is surjective.
\end{lemma}
\begin{proof}
Let $\bH = \SL_{n_1+1} \times \cdots \times \SL_{n_m+1}$.  The proof of Lemma \ref{lemma:reduciblevanishingintegral} shows
that we have $\HH_i(\bH(\Z);\St(\bH;\bbF)) = 0$ for $i \leq b$.  Applying Lemma \ref{lemma:reducetoproductintegral}, we deduce:
\begin{itemize}
\item[(i)] the group $\bH$ is a subgroup of $\bG$ and the inclusion $\bH \rightarrow \bG$ induces a bijection between parabolic subgroups, and thus
isomorphisms $\Tits(\bH) \cong \Tits(\bG)$ and $\St(\bH;\bbF) \cong \Res^{\bG(\Z)}_{\bH(\Z)} \St(\bG;\bbF)$.
\item[(ii)] $\HH_i(\bG(\Z);\St(\bG;\bbF)) = 0$ for $i \leq b$ and 
the map 
\[\HH_{b+1}(\bH(\Z);\St(\bH;\bbF)) \rightarrow \HH_{b+1}(\bG(\Z);\St(\bG;\bbF))\] 
is a surjection.
\end{itemize}
Let $\bL_{\oDelta'}^{\bH}$ and $\bL_{\Delta'}^{j_0}$ be the standard Levi subgroups of $\bH$ and $\SL_{n_{j_0}+1}$ corresponding to
the simple roots $\oDelta'$ and $\Delta'$.  We therefore have
\[\bL_{\oDelta'}^{\bH} = \SL_{n_1+1} \times \cdots \times \bL_{\Delta'}^{\bH_{j_0}} \times \cdots \times \SL_{n_m+1}.\]
By construction, $\bL_{\oDelta'}^{\bH}$ is a subgroup of $\bL^{\bG}_{\oDelta'}$ satisfying $\St(\bL_{\oDelta'}^{\bH};\bbF) = \St(\bL^{\bG}_{\oDelta'};\bbF)$.
We therefore have a commutative diagram
\[\begin{tikzcd}[column sep=small]
\HH_{b+1}(\bL_{\oDelta'}^{\bH}(\Z);\St(\bL_{\oDelta'}^{\bH};\bbF)) \arrow{r} \arrow{d}{\tf} & \HH_{b+1}(\bL^{\bG}_{\oDelta'}(\Z);\St(\bL^{\bG}_{\oDelta'};\bbF)) \arrow{d}{f} \\
\HH_{b+1}(\bH(\Z);\St(\bH;\bbF))                                   \arrow[two heads]{r}     & \HH_{b+1}(\bG(\Z);\St(\bG;\bbF))
\end{tikzcd}\]
From this, we see that to prove that $f$ is surjective, it is enough to prove that $\tf$ is surjective.

In light of \eqref{surjectivity1integral}, we can use the K\"{u}nneth formula just like we did in the
proof of Lemma \ref{lemma:reduciblevanishingintegral} (reducible vanishing) to see that
\[\HH_{b+1}(\bH(\Z);\St(\bH;\bbF)) = \bigotimes_{j=1}^m \HH_{b_j+1}(\SL_{n_j+1}(\Z);\St(\Z^{n_j+1};\bbF)).\]
We thus see that it is enough to prove that the map
\begin{align*}
&\left(\bigotimes_{j=1}^{j_0-1} \HH_{b_j+1}(\SL_{n_j+1}(\Z);\St(\Z^{n_j+1};\bbF))\right) \otimes \HH_{b_{j_0}+1}(\bL_{\Delta'}^{j_0}(\Z);\St(\bL_{\Delta'}^{j_0};\bbF)) \\
&\quad\quad\otimes \left(\bigotimes_{j=j_0+1}^{m} \HH_{b_j+1}(\SL_{n_j+1}(\Z);\St(\Z^{n_j+1};\bbF))\right)
\rightarrow \bigotimes_{j=1}^m \HH_{b_j+1}(\SL_{n_j+1}(\Z);\St(\Z^{n_j+1};\bbF))
\end{align*}
is surjective.  Since the maps on all but one tensor factor are the identity, this is equivalent
to the surjectivity of the map
\[\HH_{b_{j_0}+1}(\bL_{\Delta'}^{j_0}(\Z);\St(\bL_{\Delta'}^{j_0};\bbF)) \rightarrow \HH_{b_{j_0}+1}(\SL_{j_0+1}(\Z);\St(\Z^{n_{j_0}+1};\bbF))\]
on the remaining tensor factor, which is exactly \eqref{surjectivity2integral}.
\end{proof}

\section{Vanishing and surjectivity (integral)}
\label{section:strongintegral}

In this section, we first recall our  notation for the standard Levi factors
of groups of type $\dA_n$.  We then state our conditional vanishing theorems and prove
them for ranks at most $2$.  We remark that Theorem \ref{maintheorem:doublecomvanish} is proved 
immediately after the statement of Theorem \ref{maintheorem:doublecomvanishprime}.

\subsection{Levi factor notation}
\label{section:leviintegral}

Let $\bG$ be either $\GL_{n+1}$ or $\SL_{n+1}$, so $\bPhiQ(\bG) = \dA_n$.  We recall the
notation from Part \ref{part:an} for its Levi factors.
Let $\bDelta = \bDeltaQ(\bG)$ be the set of simple roots of $\bPhiQ(\bG) = \dA_n$.  Number the
elements of $\bDelta$ from left to right as in the usual Dynkin diagram:\\
\centerline{\psfig{file=NumberAn,scale=2}}
For $1 \leq j_1,\ldots,j_{\ell} \leq n$, let $\bDelta[j_1, \ldots, j_{\ell}]$ be the result of removing
the simple roots labeled $j_1,\ldots,j_{\ell}$ from $\bDelta$.  We have a standard Levi
subgroup $\bL_{\bDelta[j_1, \ldots, j_{\ell}]}$ of $\bG$.

\subsection{Strong vanishing}

The central results in this part of the paper are as follows.  

\begin{numberedtheorem}{maintheorem:doublecomvanishprime}{1}[{$2$ and $3$ invertible}]
\label{theorem:strongintegral23}
Assume the $b$-integral resolution conjecture (Conjecture \ref{conjecture:integralresolution}) for some $b \geq 1$.
Let $\bbF = \Z[1/2,1/3]$.  Then:
\begin{itemize}
\item $\HH_i(\SL_{n+1}(\Z);\St(\Z^{n+1};\bbF)) = 0$ for $i \leq \min(b,\lfloor (n-1)/2 \rfloor)$; and
\item letting $\bDelta = \bDeltaQ(\SL_{n+1})$, the maps
\begin{align*}
&\HH_i(\bL_{\bDelta[1]}(\Z);\St(\bL_{\bDelta[1]};\bbF)) \rightarrow \HH_i(\SL_{n+1}(\Z);\St(\Z^{n+1};\bbF)) \quad \text{and} \\
&\HH_i(\bL_{\bDelta[n]}(\Z);\St(\bL_{\bDelta[n]};\bbF)) \rightarrow \HH_i(\SL_{n+1}(\Z);\St(\Z^{n+1};\bbF))
\end{align*}
are both surjective for $i \leq \min(b,\lfloor n/2 \rfloor)$.
\end{itemize}
\end{numberedtheorem}

\begin{numberedtheorem}{maintheorem:doublecomvanishprime}{2}[Integral]
\label{theorem:strongintegral}
Assume the $b$-integral resolution conjecture (Conjecture \ref{conjecture:integralresolution}) for some $b \geq 1$.
Then:
\begin{itemize}
\item $\HH_i(\SL_{n+1}(\Z);\St(\Z^{n+1})) = 0$ for $i \leq \min(b,\lfloor (n-1)/3 \rfloor)$; and
\item letting $\bDelta = \bDeltaQ(\SL_{n+1})$, the maps
\begin{align*}
&\HH_i(\bL_{\bDelta[1]}(\Z);\St(\bL_{\bDelta[1]})) \rightarrow \HH_i(\SL_{n+1}(\Z);\St(\Z^{n+1})) \quad \text{and} \\
&\HH_i(\bL_{\bDelta[n]}(\Z);\St(\bL_{\bDelta[n]})) \rightarrow \HH_i(\SL_{n+1}(\Z);\St(\Z^{n+1}))
\end{align*}
are both surjective for $i \leq \min(b,\lfloor n/3 \rfloor)$.
\end{itemize}
\end{numberedtheorem}

We will discuss the proofs of these theorems soon, but first we show they imply:

\begin{primedtheorem}{maintheorem:doublecomvanish}
\label{maintheorem:doublecomvanishprime}
Assume the $b$-integral resolution conjecture (Conjecture \ref{conjecture:integralresolution}) for some $b \geq 1$.
Then $\HH_i(\GL_{n+1}(\Z);\St(\Z^{n+1};\bbF)) = \HH_i(\SL_{n+1}(\Z);\St(\Z^{n+1};\bbF)) = 0$
for all commutative rings $\bbF$ and all $n,i \geq 0$ such that:
\begin{itemize}
\item $i \leq \min(b,\lfloor (n-1)/2 \rfloor)$ if $2$ and $3$ are invertible in $\bbF$; and
\item $i \leq \min(b,\lfloor (n-1)/3 \rfloor)$ in general.
\end{itemize}
\end{primedtheorem}

Before proving Theorem \ref{maintheorem:doublecomvanishprime}, note that 
the only difference between it and Theorem \ref{maintheorem:doublecomvanish} is that Theorem \ref{maintheorem:doublecomvanishprime} assumes the
$b$-integral resolution conjecture, while Theorem \ref{maintheorem:doublecomvanish} assumes:\noeqref{eqn:doublecomasmrestate}
\begin{equation}
\tag{$\dagger$}\label{eqn:doublecomasmrestate} \text{The space
$\Tits^2(\Z^n)$ is $n-2+\min(b,\lfloor (n-1)/2 \rfloor)$-connected.} 
\end{equation}
Corollary \ref{corollary:doubletitsresolution} says that \eqref{eqn:doublecomasmrestate} implies
the $b$-integral resolution conjecture, so Theorem \ref{maintheorem:doublecomvanishprime} implies
Theorem \ref{maintheorem:doublecomvanish}.

\begin{proof}[Proof of Theorem \textnormal{\ref{maintheorem:doublecomvanishprime}}, assuming Theorems \textnormal{\ref{theorem:strongintegral23}} and \textnormal{\ref{theorem:strongintegral}}]
By Corollary \ref{corollary:reducetoslintegral}, it is enough to prove Theorem \ref{maintheorem:doublecomvanishprime} for $\SL_{n+1}(\Z)$.  
For $\SL_{n+1}(\Z)$, the only difference between the vanishing result in Theorem \ref{maintheorem:doublecomvanishprime} and those of
Theorems \ref{theorem:strongintegral23} and \ref{theorem:strongintegral} are the coefficients $\bbF$, which are
more general in Theorem \ref{maintheorem:doublecomvanishprime}.  Lemma \ref{lemma:reductionlemmaintegral} below 
(which generalizes Lemma \ref{lemma:reductionlemma}) shows that Theorems \ref{theorem:strongintegral23} and 
\ref{theorem:strongintegral} imply the slightly more general vanishing in Theorem \ref{maintheorem:doublecomvanishprime}.
\end{proof}

\begin{lemma}
\label{lemma:reductionlemmaintegral}
Let $\bG$ be a standard $\dA$-integral Levi subgroup, let $\cP$ be a set of primes in $\Z$, and let
$\Z_{\cP} = \Z[\text{$1/p$ $|$ $p \in \cP$}]$.  Let $b \geq 0$ be such that
$\HH_i(\bG(\Z);\St(\bG;\Z_{\cP})) = 0$ for $i \leq b$.  Then for all commutative rings $\bbF$ such that
each $p \in \cP$ is invertible in $\bbF$ we have $\HH_i(\bG(\Z);\St(\bG;\bbF)) = 0$ for $i \leq b$.
\end{lemma} 
\begin{proof}
Letting $F_{\bullet} \rightarrow \Z$ be a free resolution of $\bG(\Z)$-modules, the homology groups
of $F_{\bullet} \otimes_{\bG(\Z)} \St(\bG;\Z_{\cP})$ are $\HH_i(\bG(\Z);\St(\bG;\Z_{\cP}))$.  Since
$\St(\bG;\Z_{\cP})$ is a free $\Z_{\cP}$-module, each $F_i \otimes_{\bG(\Z)} \St(\bG;\Z_{\cP})$
is also a free $\Z_{\cP}$-module.  Since $\Z_{\cP}$ is a PID, we can apply the universal coefficients
theorem to
\[F_{\bullet} \otimes_{\bG(\Z)} \St(\bG;\Z_{\cP}) \otimes_{\Z_{\cP}} \bbF = F_{\bullet} \otimes_{\bG(\Z)} \St(\bG;\bbF),\]
which computes $\HH_i(\bG(\Z);\St(\bG;\bbF))$.  The result is a short exact sequence
\[0 \rightarrow \HH_i(\bG(\Z);\St(\bG;\Z_{\cP})) \otimes \bbF \rightarrow \HH_i(\bG(\Z);\St(\bG;\bbF)) \rightarrow \Tor(\HH_{i-1}(\bG(\Z);\St(\bG)),\Z_{\cP}) \rightarrow 0.\]
The lemma follows.
\end{proof}
 
\subsection{Initial results}

Before discussing the proofs of Theorems \ref{theorem:strongintegral23} and \ref{theorem:strongintegral}, we prove two preliminary results.

\begin{lemma}
\label{lemma:conjugateintegral}
Let $\bG = \SL_{n+1}$ for some $n \geq 1$.  For all commutative rings $\bbF$, the images of the maps
\begin{align*}
&\HH_i(\bL_{\bDelta[1]}(\Z);\St(\bL_{\bDelta[1]};\bbF)) \rightarrow \HH_i(\bG(\Z);\St(\bG;\bbF)) \quad \text{and} \\
&\HH_i(\bL_{\bDelta[n]}(\Z);\St(\bL_{\bDelta[n]};\bbF)) \rightarrow \HH_i(\bG(\Z);\St(\bG;\bbF))
\end{align*}
are the same.
\end{lemma}
\begin{proof}
We have
\begin{align*}
&\bL_{\bDelta[1]}(\Z) = \ker(\GL_1(\Z) \times \GL_n(\Z) \stackrel{\det}{\longrightarrow} \Z^{\times}), \\
&\bL_{\bDelta[n]}(\Z) = \ker(\GL_n(\Z) \times \GL_1(\Z) \stackrel{\det}{\longrightarrow} \Z^{\times}).
\end{align*}
These are conjugate subgroups of $\SL_{n+1}(\Z)$, with the conjugating matrix an appropriate signed permutation
matrix.  The lemma now follows from the fact that 
inner automorphisms act trivially on group homology.
\end{proof}

\begin{lemma}
\label{lemma:strongintegralrank2}
Theorems \textnormal{\ref{theorem:strongintegral23}} and \textnormal{\ref{theorem:strongintegral}} hold for $n \leq 2$.
\end{lemma}
\begin{proof}
For $n=0$, these two theorems assert nothing.  We therefore
only need to handle $\SL_2$ and $\SL_3$.  Examining the conclusions of
Theorems \ref{theorem:strongintegral23} and \ref{theorem:strongintegral} for these groups, it is enough to prove
the following two things:

\begin{claim}{1}
\label{claim:strongintegralrank2.1}
$\HH_0(\SL_{n+1}(\Z);\St(\Z^{n+1};\bbF)) = 0$ for all commutative rings $\bbF$ and all $n \geq 1$.
\end{claim}

This was proved by Lee--Szczarba \cite[Theorem 1.3]{LeeSzczarba}.

\begin{claim}{2}
\label{claim:strongintegralrank2.3}
For all commutative rings $\bbF$ such that $2$ and $3$ are invertible in $\bbF$, we have
$\HH_1(\SL_3(\Z);\St(\Z^3;\bbF)) = 0$.  We remark that this makes the surjectivity statement
from Theorem \textnormal{\ref{theorem:strongintegral23}} trivially true for $\SL_3(\Z)$.
\end{claim}

For a field $\bbF$ of characteristic $0$, the paper \cite{ChurchPutmanCodimOne} proves that
$\HH_1(\SL_n(\Z);\St(\Z^n;\bbF)) = 0$ for $n \geq 3$.  
In fact, as was pointed out in \cite{KupersKTheory} the
proof in \cite{ChurchPutmanCodimOne} shows more generally that this holds
if $\bbF$ is a commutative ring in which $p$ is invertible for all primes $p \leq n$, giving the claim.
Since \cite{KupersKTheory} does not explain the argument, we give the details.  We remark
that \cite[Theorem 1.10]{MillerNagpalPatzt} proves
that for an arbitrary commutative ring $\bbF$ we have $\HH_1(\SL_n(\Z);\St(\Z^n;\bbF)) = 0$ for $n \geq 6$,
but our main interest is in the case $n=3$, so this result does not suffice for us.

The main theorem of \cite{ChurchPutmanCodimOne} gives a presentation
\[\bX_1 \rightarrow \bX_0 \rightarrow \St(\Z^n) \rightarrow 0.\]
See \S \ref{section:threestepmodular} for more about this presentation.  All the terms here are free abelian groups, so we can tensor this with $\bbF$ to
get a presentation
\[\bY_1 \rightarrow \bY_0 \rightarrow \St(\Z^n;\bbF) \rightarrow 0.\]
Let $\bY_2$ be the kernel of the map $\bY_1 \rightarrow \bY_0$, so we have a resolution
\[0 \rightarrow \bY_2 \rightarrow \bY_1 \rightarrow \bY_0 \rightarrow \St(\Z^n;\bbF) \rightarrow 0.\]
This gives a spectral sequence of the form
\[\ssE^1_{pq} \cong \HH_q(\SL_n(\Z);\bY_p) \Rightarrow \HH_{p+q}(\SL_n(\Z);\St(\Z^n;\bbF)).\]
To prove that $\HH_1(\SL_n(\Z);\St(\Z^n;\bbF)) = 0$, it is enough to prove that
$\ssE^1_{10} = \ssE^1_{01} = 0$.

The term $\ssE^1_{10} = \HH_0(\SL_n(\Z);\bY_1)$ equals the coinvariants of the action of $\SL_n(\Z)$ on
$\bY_1$.  In \cite{ChurchPutmanCodimOne}, the authors prove that $\bX_1$ is generated by 
elements $\sigma$ such that there exists $f \in \SL_n(\Z)$ with $f(\sigma) = -\sigma$.  In
the coinvariants, we thus have $2 \sigma = 0$.  Since $2$ is invertible in $\bbF$, this
implies that the image of $\sigma$ in $(\bY_1)_{\SL_n(\Z)} = (\bX_1 \otimes \bbF)_{\SL_n(\Z)}$ is
$0$.  Thus $\ssE^1_{01} = 0$, as desired.

For $\ssE^1_{01} = \HH_1(\SL_n(\Z);\bY_0)$, in \cite{ChurchPutmanCodimOne} the authors
describe $\bX_0$ as an induced representation, which we now describe.  Let $\Lambda_n$ be the subgroup
of $\SL_n(\Z)$ consisting of signed permutation matrices with determinant $1$.  Let $\Z_{\sign}$ and $\bbF_{\sign}$ be
$\Z$ and $\bbF$ with the $\Lambda_n$-actions given by multiplication by the sign of the permutation.\footnote{This
is not the determinant, which is always $1$.  The signs in the signed permutation matrices play no role in it.}
The paper \cite{ChurchPutmanCodimOne} then proves that
\[\bX_0 = \Ind_{\Lambda_n}^{\SL_n(\Z)} \Z_{\sign}, \quad \text{so $\bY_0 = \Ind_{\Lambda_n}^{\SL_n(\Z)} \bbF_{\sign}$}.\]
By Shapiro's lemma, we have
\[\ssE^1_{01} = \HH_1(\SL_n(\Z);\bY_0) \cong \HH_1(\Lambda_n;\bbF_{\sign}).\]
The group $\Lambda_n < \SL_n(\Z)$ is a subgroup of the whole signed permutation group, and thus
has order dividing $2^n n!$.  In particular, since all primes $p \leq n$ are invertible in $\bbF$ it
follows that multiplication by $|\Lambda_n|$ is invertible in $\bbF$.  This implies that $\ssE^1_{01} = 0$, as desired.
\end{proof}

\begin{remark}
For $\SL_3(\Z)$ we could only verify surjectivity\footnote{We actually proved vanishing
for $\SL_3(\Z)$.  Knowing this does not seem to improve our ultimate result.} directly after inverting $2$ and $3$.  Integrally,
we will be forced to start by proving surjectivity for $\SL_4(\Z)$.  This is the
ultimate source of the $1/3$ slope in Theorem \ref{maintheorem:doublecomvanishprime}.  
One might hope that we could prove a weaker theorem with a $1/2$ slope; for instance, that
$\HH_i(\SL_n(\Z);\St(\SL_n)) = 0$ for $i \leq \lfloor (n-2)/2 \rfloor$ and that surjectivity
holds for $i = \lfloor (n-1)/2 \rfloor$.  However, these ranges would not give a sufficient
vanishing range in our spectral sequence argument; cf.\ Remark \ref{remark:needoffset1}.
\end{remark}

\subsection{Proof when \texorpdfstring{$2$}{2} and \texorpdfstring{$3$}{3} are invertible}

We now turn to Theorems \ref{theorem:strongintegral23} and \ref{theorem:strongintegral}.  
We will give full details for Theorem \ref{theorem:strongintegral} below, and to get a sense
as to how the argument works a reader might want to read that first.
For Theorem \ref{theorem:strongintegral23}, the ranges in it for vanishing and surjectivity
are exactly the same as those we proved in type $\dA_n$ in Part \ref{part:an}.  The proof
of Theorem \ref{theorem:strongintegral23} follows Part \ref{part:an} closely, with the following changes:
\begin{itemize}
\item Throughout, work with $\bbF = \Z[1/2,1/3]$ instead of $\bbF = \Z$.  
Since $\Z[1/2,1/3]$ is a PID, this causes no problems in the proof (e.g., in the uses of the K\"{u}nneth formula).
\item The spectral sequence from Lemma \ref{lemma:spectralsequenceintegral} replaces the spectral sequence for general
reductive groups from Corollary \ref{corollary:spectralsequence}.  The existence of this spectral
sequence is where we use the $b$-integral resolution conjecture (Conjecture \ref{conjecture:integralresolution}).
\item To handle reducible root systems, substitute the results about reducible Levi subgroups from \S \ref{section:integrallevi} for
the corresponding results from \S \ref{section:reducible}.
\end{itemize}
Because the arguments are so similar, we will not give the full details of the proof.

\subsection{Proof over the integers}

The rest of this part of the paper is devoted to the proof of Theorem \ref{theorem:strongintegral}.
For the proof, we will assume that we have already proved Theorem \ref{theorem:strongintegral} in smaller ranks.  
For this, we make the following definition:

\begin{definition}
\label{definition:hypothesisintegral}
For $b \geq 1$ and $r \geq 0$, the {\em integral $(b,r)$-surjectivity and vanishing hypothesis} is as follows:
\begin{itemize}
\item The $b$-integral resolution (Conjecture \ref{conjecture:integralresolution}) holds.
\item For all $n \leq r$, we have $\HH_i(\SL_{n+1}(\Z);\St(\Z^{n+1})) = 0$ for $i \leq \min(b,\lfloor (n-1)/3 \rfloor)$.
\item For all $n \leq r$, letting $\bDelta = \bDeltaQ(\SL_{n+1})$ the maps
\begin{align*}
&\HH_i(\bL_{\bDelta[1]}(\Z);\St(\bL_{\bDelta[1]})) \rightarrow \HH_i(\SL_{n+1}(\Z);\St(\Z^{n+1})) \quad \text{and} \\
&\HH_i(\bL_{\bDelta[n]}(\Z);\St(\bL_{\bDelta[n]})) \rightarrow \HH_i(\SL_{n+1}(\Z);\St(\Z^{n+1}))
\end{align*}
are both surjective for $i \leq \min(b,\lfloor n/3 \rfloor)$.\qedhere
\end{itemize}
\end{definition}

\section{Vanishing and surjectivity for Levi subgroups (integral)}
\label{section:levivanishintegral}

In this section, we show how to use the integral $(b,r)$-surjectivity and vanishing hypothesis
to analyze the homology of standard Levi subgroups.  As notation, for 
$\bPhi = \dA_{n_1} \times \cdots \times \dA_{n_m}$ define
\[\hbb(\bPhi) = (m-1) + \lfloor (n_1-1)/3 \rfloor + \cdots + \lfloor(n_m-1)/3 \rfloor.\]
Our main result is as follows.  Its statement uses the ordering on the simple roots of
of $\dA_{n_{j_0}}$ discussed in \S \ref{section:leviintegral}.

\begin{lemma}[Levi vanishing and surjectivity]
\label{lemma:levivanishintegral}
Assume the integral $(b,n-1)$-surjectivity and vanishing hypothesis (Definition \ref{definition:hypothesisintegral}).
Let $\Delta \subset \bDeltaQ(\SL_{n+1})$ be a set of simple roots with $\Delta \neq \bDeltaQ(\SL_{n+1})$.  Write
\[\bPhiQ(\bL_{\Delta}) = \dA_{n_1} \times \cdots \times \dA_{n_m}.\]
Set\footnote{In other versions of this lemma we used $b$ for this bound, but we change to $c$ since $b$ is being used
for something else.}
\[c = \hbb(\bPhiQ(\bL_{\Delta})) = (m-1) + \lfloor (n_1-1)/3 \rfloor + \cdots + \lfloor(n_m-1)/3 \rfloor.\]
Then the following hold:
\begin{itemize}
\item[(i)] We have $\HH_i(\bL_{\Delta}(\Z);\St(\bL_{\Delta})) = 0$ for $i \leq \min(b,c)$.
\item[(ii)] For some $1 \leq j_0 \leq n$, assume that $n_j$ is divisible by $3$.  Let $\Delta' \subset \Delta$
be the set of simple roots obtained by removing either the first or last simple root from $\dA_{n_{j_0}}$, so
\[\bPhiQ(\bL_{\Delta'}) = \dA_{n_1} \times \cdots \times \dA_{n_{j_0}-1} \times \cdots \times \dA_{n_m}.\]
Finally, assume that $c+1 \leq b$.  Then the map 
\[\HH_{c+1}(\bL_{\Delta'}(\Z);\St(\bL_{\Delta'})) \rightarrow \HH_{c+1}(\bL_{\Delta}(\Z);\St(\bG))\]
is surjective.
\end{itemize}
\end{lemma}
\begin{proof}
Since $\Delta \neq \bDeltaQ(\bG)$, we have $n_j \leq n-1$ for $1 \leq j \leq m$.  The
integral $(b,n-1)$-surjectivity and vanishing hypothesis thus applies to all the
groups $\SL_{n_j+1}$.  This gives the hypothesis \eqref{vanishing1integral} in
Lemma \ref{lemma:reduciblevanishingintegral} (reducible vanishing).  Applying Lemma \ref{lemma:reduciblevanishingintegral},
we deduce (i).  Similarly, for $\Delta'$ as in (ii) it gives the hypotheses \eqref{surjectivity1integral} and \eqref{surjectivity2integral}
in Lemma \ref{lemma:reduciblesurjectivityintegral} (reducible surjectivity).  Applying Lemma \ref{lemma:reduciblesurjectivityintegral},
we deduce (ii).
\end{proof}

\section{Vanishing region (integral)}
\label{section:ssvanishintegral}

If the $b$-integral resolution conjecture (Conjecture \ref{conjecture:integralresolution}) holds,
then for $n \geq 3$ Lemma \ref{lemma:spectralsequenceintegral} gives a spectral sequence $\ssE^r_{pq}$ converging to $\HH_{p+q}(\SL_{n+1}(\Z);\St(\Z^{n+1}))$
with\footnote{Here we switched from $\Delta$ to $R$ to make our notation match that of the corresponding sections
of Part \ref{part:an}.}
\[\ssE^1_{pq} \cong \bigoplus_{R \in \cL_p(\bG)} \HH_q(\bL_{R}(\Z);\St(\bL_{R})) \quad \text{if $0 \leq p \leq \min(b,\lfloor n/2 \rfloor)$}\]
The following lemma shows that our inductive hypothesis implies that many terms of this
spectral sequence vanish.

\begin{lemma}
\label{lemma:vanishintegral}
Assume the integral $(b,n-1)$-surjectivity and vanishing hypothesis (Definition \ref{definition:hypothesisintegral})
for some $b \geq 1$ and $n \geq 3$.
Let $\ssE^1_{pq}$ be the spectral sequence converging to
$\HH_{p+q}(\SL_{n+1}(\Z);\St(\Z^{n+1}))$ from Lemma \ref{lemma:spectralsequenceintegral}.  Then the following hold:
\begin{itemize}
\item Let $d = \lfloor n/3 \rfloor$.  Then $\ssE^1_{pq} = 0$ for $p+q \leq \min(b,d)$, except for possibly
$\ssE^1_{0d}$ when $n$ is congruent to $0$ or $1$ modulo $3$.  There are thus no exceptions if $b<d$ or $n$ is congruent to $2$ modulo $3$.
\end{itemize}
\end{lemma}
\begin{proof}
Our goal is to prove a vanishing result for $\ssE^1_{pq}$.  The terms in question all have $p \leq \min(b,d)$, so they all satisfy
\[\ssE^1_{pq} \cong \bigoplus_{R \in \cL_p(\bG)} \HH_q(\bL_{R}(\Z);\St(\bL_{R})).\]
Consider some $R \in \cL_p(\bG)$.  We will prove that the
integral $(b,n-1)$-surjectivity and vanishing hypothesis implies
that $\HH_q(\bL_{R}(\Z);\St(\bL_{R})) = 0$ whenever $q \leq b$ and $p+q \leq d$, with the possible
exception of $(p,q) = (0,d)$ if $n$ is congruent to $0$ or $1$ modulo $3$.  This will give the lemma.

Since $R$ is obtained by deleting $p+1$ simple roots from the Dynkin diagram of $\dA_n$, we have
\[\bPhiQ(\bL_{R}) = \dA_{n_1} \times \cdots \times \dA_{n_m} \quad \text{with $n_1+\cdots+n_m+p+1 = n$}.\]
Lemma \ref{lemma:levivanishintegral} (Levi vanishing and surjectivity) implies that
$\HH_q(\bL_R(\Z);\St(\bL_R)) = 0$ for $q$ such that $q \leq b$ and
\begin{equation}
\label{eqn:integralvanishqbound}
q \leq \hbb(\bPhiQ(\bL_{R})) = (m-1) + \lfloor (n_1-1)/3 \rfloor + \cdots + \lfloor (n_m-1)/3 \rfloor.
\end{equation}
For $a_1,a_2\in \Z$, Lemma \ref{lemma:floorinequality} implies that
$1+ \lfloor a_1/3 \rfloor + \lfloor a_2/3 \rfloor \geq \lfloor (a_1+a_2+1)/3 \rfloor$.
Applying this repeatedly, we deduce that the right hand side of \eqref{eqn:integralvanishqbound} is at least
\[\left\lfloor (n_1+\cdots+n_m-1)/3 \right\rfloor = \left\lfloor (n-p-2)/3 \right\rfloor.\]
It follows that $\HH_q(\bL_{R}(\Z);\St(\bL_{R})) = 0$ for $q$ such that $q \leq b$ and
$q \leq \lfloor (n-p-2)/3 \rfloor$, or equivalently
\begin{equation}
\label{eqn:integralvanishqbound2}
p+q \leq p+\lfloor (n-p-2)/3 \rfloor = \lfloor (n+2p-2)/3 \rfloor.
\end{equation}
There are now three cases:
\begin{itemize}
\item If $p \geq 1$, then the right hand side of \eqref{eqn:integralvanishqbound2} is at least $\lfloor n/3 \rfloor = d$.
\item If $p=0$ and $n$ is congruent to $2$ modulo $3$, then the right hand side of \eqref{eqn:integralvanishqbound2} is
$\lfloor (n-2)/3 \rfloor = \lfloor n/3 \rfloor = d$.
\item If $p=0$ and $n$ is congruent to $0$ or $1$ modulo $3$, then the right hand side of
\eqref{eqn:integralvanishqbound2} is $\lfloor (n-2)/3 \rfloor = \lfloor n/3\rfloor - 1 = d-1$.
\end{itemize}
All of this implies what we are trying to show: that $\HH_q(\bL_{R}(\Z);\St(\bL_{R})) = 0$ for $q$ such that $q \leq b$ and
$p+q \leq \lfloor n/3 \rfloor$, with the possible exception of $(p,q) = (0,d)$ if $n$ is congruent to $0$ or $1$ modulo $3$.
\end{proof}

\section{Remaining tasks (integral)}
\label{section:proofintegral}

Lemma \ref{lemma:vanishintegral} implies many cases of Theorem \ref{theorem:strongintegral}.  To prove the
remaining cases, we need to compute some differentials in our spectral sequence.  We now explain
the structure of the argument, postponing three calculations to the next
section.  Recall that Theorem \ref{theorem:strongintegral} is:

\theoremstyle{plain}
\newtheorem*{theorem:strongintegral}{Theorem \ref{theorem:strongintegral}}
\begin{theorem:strongintegral}
Assume the $b$-integral resolution conjecture (Conjecture \ref{conjecture:integralresolution}) for some $b \geq 1$.
Then:
\begin{itemize}
\item $\HH_i(\SL_{n+1}(\Z);\St(\Z^{n+1})) = 0$ for $i \leq \min(b,\lfloor (n-1)/3 \rfloor)$; and
\item letting $\bDelta = \bDeltaQ(\SL_{n+1})$, the maps
\begin{align*}
&\HH_i(\bL_{\bDelta[1]}(\Z);\St(\bL_{\bDelta[1]})) \rightarrow \HH_i(\SL_{n+1}(\Z);\St(\Z^{n+1})) \quad \text{and} \\
&\HH_i(\bL_{\bDelta[n]}(\Z);\St(\bL_{\bDelta[n]})) \rightarrow \HH_i(\SL_{n+1}(\Z);\St(\Z^{n+1}))
\end{align*}
are both surjective for $i \leq \min(b,\lfloor n/3 \rfloor)$.
\end{itemize}
\end{theorem:strongintegral}
\begin{proof}
The proof is by induction on $n$.  We proved the base cases $n \leq 2$ in Lemma \ref{lemma:strongintegralrank2},
so we can assume that $n \geq 3$ and that the result is true for smaller ranks, i.e., that
the integral $(b,n-1)$-surjectivity and vanishing hypothesis holds.

Lemma \ref{lemma:spectralsequenceintegral} gives a spectral sequence $\ssE^r_{pq}$ converging
to $\HH_{p+q}(\SL_{n+1}(\Z);\St(\Z^{n+1}))$.  Letting $d = \lfloor n/3 \rfloor$, Lemma \ref{lemma:vanishintegral} implies that
$\ssE^1_{pq} = 0$ for $p+q \leq \min(b,d)$ except for
possibly $\ssE^1_{0d}$ when $n$ is congruent to $0$ or $1$ modulo $3$.  This implies that
$\HH_i(\SL_{n+1}(\Z);\St(\Z^{n+1})) = 0$ for 
\[i \leq \begin{cases}
b   & \text{if $b < d$},\\
d   & \text{if $b \geq d$ and $n$ is congruent to $2$ modulo $3$}, \\
d-1 & \text{if $b \geq d$ and $n$ is congruent to $0$ or $1$ modulo $3$}.
\end{cases}\]
Since our surjectivity claim is trivial when the target is $0$, all that remains to prove
are the following two claims:

\begin{claim}{1}
Assume that $b \geq d$ and that $n$ is congruent to $1$ modulo $3$, so $d = \lfloor (n-1)/3 \rfloor$.
Then $\HH_d(\SL_{n+1}(\Z);\St(\Z^{n+1})) = 0$.
\end{claim}

In this case, Lemma \ref{lemma:vanishintegral} says that the only potentially nonzero
term $\ssE^1_{pq}$ in our spectral sequence with $p+q = d$ is $\ssE^1_{0d}$.  We will prove in
Lemma \ref{lemma:differential2integral} below that the differential $\ssE^1_{1d} \rightarrow \ssE^1_{0d}$
is surjective, so $\ssE^2_{0d} = 0$.  This implies that $\HH_d(\SL_{n+1}(\Z);\St(\Z^{n+1})) = 0$, as desired.

\begin{claim}{2}
Assume that $b \geq d$ and that $n$ is congruent to $0$ modulo $3$, so $d-1 = \lfloor (n-1)/3 \rfloor$.  Then the maps
\begin{align*}
&\HH_d(\bL_{\bDelta[1]}(\Z);\St(\bL_{\bDelta[1]})) \rightarrow \HH_d(\SL_{n+1}(\Z);\St(\Z^{n+1})) \quad \text{and} \\
&\HH_d(\bL_{\bDelta[n]}(\Z);\St(\bL_{\bDelta[n]})) \rightarrow \HH_d(\SL_{n+1}(\Z);\St(\Z^{n+1}))
\end{align*}
are both surjective.
\end{claim}

Lemma \ref{lemma:conjugateintegral} says that these maps have the same image, so it is enough to prove that
\begin{equation}
\label{eqn:sumoftwointegral}
\HH_d(\bL_{\bDelta[1]}(\Z);\St(\bL_{\bDelta[1]})) \oplus \HH_d(\bL_{\bDelta[n]}(\Z);\St(\bL_{\bDelta[n]}))
\end{equation}
surjects onto $\HH_d(\SL_{n+1}(\Z);\St(\Z^{n+1}))$.
Lemma \ref{lemma:vanishintegral} says that the only potentially nonzero
term $\ssE^1_{pq}$ in our spectral sequence with $p+q = d$ is $\ssE^1_{0d}$.  
We will prove in Lemma \ref{lemma:differential1integral} below that the summand \eqref{eqn:sumoftwointegral} of
\[\ssE^1_{0d} = \bigoplus_{R \in \cL_0(\bG)} \HH_d(\bL_{R}(\Z);\St(\bL_{R})) = \bigoplus_{j=1}^{n} \HH_d(\bL_{\bDelta[j]}(\Z);\St(\bL_{\bDelta[j]}))\]
surjects onto the cokernel of the differential $\ssE^1_{1d} \rightarrow \ssE^1_{0d}$.
It follows that $\ssE^2_{0d}$ is a quotient of \eqref{eqn:sumoftwointegral}.  Since $\ssE^2_{0d}$ is the
only potentially nonzero term of the form $\ssE^2_{pq}$ with $p+q = d$, it follows that \eqref{eqn:sumoftwointegral}
surjects onto $\HH_d(\SL_{n+1}(\Z);\St(\Z^{n+1}))$, as desired.
\end{proof}

\section{Differentials (integral)}
\label{section:differentialintegral}

This final section of this part of the paper
determines the images of two differentials whose calculations were needed in the previous section.

\subsection{Differentials, I (integral)}
\label{section:differential1integral}

Our first differential calculation is:

\begin{lemma}
\label{lemma:differential1integral}
Consider $\SL_{n+1}$ with $n = 3d$ for some $d \geq 1$.  
Assume the integral $(b,n-1)$-surjectivity and vanishing hypothesis (Definition \ref{definition:hypothesisintegral}) for some $b \geq d$.
Let $\ssE^1_{pq}$ be the spectral
sequence from Lemma \ref{lemma:spectralsequenceintegral}.  Then the summand
\[\HH_d(\bL_{\bDelta[1]}(\Z);\St(\bL_{\bDelta[1]})) \oplus \HH_d(\bL_{\bDelta[2d]}(\Z);\St(\bL_{\bDelta[2d]}))\]
of $\ssE^1_{0d}$ surjects onto the cokernel of the
differential $\ssE^1_{1d} \rightarrow \ssE^1_{0d}$.
\end{lemma}
\begin{proof}
As notation, for $1 \leq j_1,\ldots,j_{\ell} \leq 3d$ let
\[M[j_1, \ldots, j_{\ell}] = \HH_d(\bL_{\bDelta[j_1, \ldots, j_{\ell}]}(\Z);\St(\bL_{\bDelta[j_1, \ldots, j_{\ell}]})).\]
We have
\[\ssE^1_{0d} = \bigoplus_{1 \leq j \leq 3d} M[j] \quad \text{and} \quad 
\ssE^1_{1d} = \bigoplus_{1 \leq j_1 < j_2 \leq 3d} M[j_1, j_2].\]
Consider some $1 < j < 3d$.  We must prove that when we quotient $\ssE^1_{0d}$ by the image of the differential
$\ssE^1_{1d} \rightarrow \ssE^1_{0d}$, the summand
$M[j]$ of $\ssE^1_{0d}$ is identified with a subspace of
$M[1] \oplus M[3d]$.  The first step is to show that many $M[j]$ already vanish:

\begin{claim}{1}
\label{claim:differential1integral.1}
For $1 < j < 3d$ with $j=3e+2$ for some $0 \leq e < d$, we have $M[j] = 0$.
\end{claim}
\begin{proof}[Proof of claim]
Since $\bPhiQ(\bL_{\bDelta[3e+2]}) = \dA_{3e+1} \times \dA_{3d-3e-2}$, Lemma \ref{lemma:levivanishintegral} (Levi vanishing and surjectivity) 
implies that $\HH_i(\bL_{\bDelta[3e+2]}(\Z);\St(\bL_{\bDelta[3e+2]})) = 0$ for
\[i \leq \hbb(\bPhiQ(\bL_{\bDelta[3e+2]})) = 1 + \lfloor 3e/3 \rfloor + \lfloor (3d-3e-3)/3 \rfloor = 1 + e + (d-e-1) = d.\]
In particular, $M[3e+2] = \HH_d(\bL_{\bDelta[3e+2]}(\Z);\St(\bL_{\bDelta[3e+2]})) = 0$.
\end{proof}

In light of Claim \ref{claim:differential1integral.1}, there are two remaining cases.
The first is that $j = 3e+1$ for some $1 \leq e < d$, so
\begin{align}
\bPhiQ(\bL_{\bDelta[j]}) &= \dA_{3e} \times \dA_{3d-3e-1}, \label{eqn:differential1integral.1.1}\\
\hbb(\bPhiQ(\bL_{\bDelta[j]})) &= 1 + \lfloor (3e-1)/3 \rfloor + \lfloor (3d-3e-2)/3 \rfloor \label{eqn:differential1integral.1.2}\\
                               &= 1 + (e-1)+(d-e-1) = d-1. \notag
\end{align}
The differential $\ssE^1_{1d} \rightarrow \ssE^1_{0d}$ takes
the summand $M[1,j]$ of $\ssE^1_{1d}$ to $\ssE^1_{0d}$ via the map
\[\begin{tikzcd}[column sep=large]
M[1,j] \arrow{r}{f \oplus (-g)} & M[j] \oplus M[1] \arrow[hook]{r} & \ssE^1_{0d}.
\end{tikzcd}\]
To prove that this differential identifies $M[j]$ with a subspace of $M[1]$,
we must prove that $f\colon M[1,j] \rightarrow M[j]$ is surjective.  Since
the
$\dA_{3e}$ in \eqref{eqn:differential1integral.1.1} has $3e$ a positive multiple of $3$, this follows
from Lemma \ref{lemma:levivanishintegral} (Levi vanishing and surjectivity).  Here
we are using the fact that $\hbb(\bPhiQ(\bL_{\bDelta[j]}))+1 = d$; cf.\ \eqref{eqn:differential1integral.1.2}.

The second case is that $j = 3e$ for some $1 \leq e < d$, so
\begin{align}
\bPhiQ(\bL_{\bDelta[j]}) &= \dA_{3e-1} \times \dA_{3d-3e}, \label{eqn:differential1integral.2.1} \\
\hbb(\bPhiQ(\bL_{\bDelta[j]})) &= 1 + \lfloor (3e-2)/3 \rfloor + \lfloor (3d-3e-1)/3 \rfloor \label{eqn:differential1integral.2.2}\\
                               &= 1+(e-1)+(d-e-1) = d-1. \notag
\end{align}
The differential $\ssE^1_{1d} \rightarrow \ssE^1_{0d}$ takes
the summand $M[j,3d]$ of $\ssE^1_{1d}$ to $\ssE^1_{0d}$ via the map
\[\begin{tikzcd}[column sep=large]
M[j,3d] \arrow{r}{f \oplus (-g)} & M[3d] \oplus M[j] \arrow[hook]{r} & \ssE^1_{0d}.
\end{tikzcd}\]
To prove that this differential identifies $M[j]$ with a subspace of $M[3d]$,
we must prove that $f\colon M[j,3d] \rightarrow M[j]$ is surjective.  Since
the $\dA_{3d-3e}$ in \eqref{eqn:differential1integral.2.1} has $3d-3e$ a positive multiple
of $3$, this follows from 
Lemma \ref{lemma:levivanishintegral} (Levi vanishing and surjectivity).
Here
we are using the fact that $\hbb(\bPhiQ(\bL_{\bDelta[j]}))+1 = d$; cf.\ \eqref{eqn:differential1integral.2.2}.
\end{proof}

\subsection{Differentials, II (integral)}
\label{section:differential2integral}

Our second and final differential calculation is:

\begin{lemma}
\label{lemma:differential2integral}
Consider $\SL_{n+1}$ with $n = 3d+1$ for some $d \geq 1$.  
Assume the integral $(b,n-1)$-surjectivity and vanishing hypothesis (Definition \ref{definition:hypothesisintegral})
for some $b \geq d$.
Let $\ssE^1_{pq}$ be the spectral
sequence from Lemma \ref{lemma:spectralsequenceintegral}.  Then the
differential $\ssE^1_{1d} \rightarrow \ssE^1_{0d}$ is surjective.
\end{lemma}
\begin{proof}
As notation, for $1 \leq j_1,\ldots,j_{\ell} \leq 3d+1$ let
\[M[j_1, \ldots, j_{\ell}] = \HH_d(\bL_{\bDelta[j_1, \ldots, j_{\ell}]}(\Z);\St(\bL_{\bDelta[j_1, \ldots, j_{\ell}]})).\]
We have
\[\ssE^1_{0d} = \bigoplus_{1 \leq j \leq 3d+1} M[j] \quad \text{and} \quad 
\ssE^1_{1d} = \bigoplus_{1 \leq j_1 < j_2 \leq 3d+1} M[j_1, j_2].\]
Consider some $1 \leq j \leq 3d+1$.  We must prove that when we quotient $\ssE^1_{0d}$ by the image of the differential
$\ssE^1_{1d} \rightarrow \ssE^1_{0d}$, the summand $M[j]$ is killed.
The first step is to show that many $M[j]$ already vanish:

\begin{claim}{1}
\label{claim:differential2integral.1}
For $1 \leq j \leq 3d+1$ with $j$ equal to either $3e+2$ with $0 \leq e < d$ or $3e$ with $1 \leq e \leq d$, we have $M[j] = 0$.
\end{claim}
\begin{proof}[Proof of claim]
Since $\bPhiQ(\bL_{\bDelta[j]}) = \dA_{j-1} \times \dA_{3d+1-j}$, Lemma \ref{lemma:levivanishintegral} (Levi vanishing and surjectivity)
implies that
$\HH_i(\bL_{\bDelta[j]}(\Z);\St(\bL_{\bDelta[j]})) = 0$ for
\[i \leq \hbb(\bPhiQ(\bL_{\bDelta[j]})) = 1 + \lfloor (j-2)/3 \rfloor + \lfloor (3d-j+1)/3 \rfloor.\]
We must prove that the right hand side of this is $d$ for $j$ of the indicated forms:
\begin{itemize}
\item If $j = 3e+2$ with $0 \leq e < d$, then the right hand side is
\[1 + \lfloor 3e/3 \rfloor + \lfloor 3d-3e-1 \rfloor = 1 + e + (d-e-1) = d.\]
\item If $j = 3e$ with $1 \leq e \leq d$, then the right hand side is
\[1 + \lfloor (3e-2) \rfloor + \lfloor (3d-3e+1)/3 \rfloor = 1 + (e-1) + (d-e) = d.\qedhere\]
\end{itemize}
\end{proof}

Now consider $1 \leq j \leq 3d+1$ with $j = 3e+1$ for some $0 \leq e \leq d$.
In light of Claim \ref{claim:differential2integral.1}, it is enough
to prove that $M[j]$ is killed when we quotient $\ssE^1_{0d}$ by the image of the differential
$\ssE^1_{1d} \rightarrow \ssE^1_{0d}$.  Assume first that $j \neq 1$, so
\begin{equation}
\label{eqn:differential2integral.1.1}
\bPhiQ(\bL_{\bDelta[j]}) = 
\begin{cases}
\dA_{3e} \times \dA_{3d-3e} & \text{if $e \neq d$},\\
\dA_{3e}                    & \text{if $e=d$}.
\end{cases}
\end{equation}
In both cases, we have
\begin{equation}
\label{eqn:differential2integral.1.2}
\hbb(\bPhiQ(\bL_{\bDelta[j]}) = d-1.
\end{equation}
We will show that $M[j]$ is killed by the image of the summand $M[j-1,j]$ of $\ssE^1_{1d}$. 
On the summand $M[j-1,j]$, the differential is the map
\[\begin{tikzcd}[column sep=large]
M[j-1,j] \arrow{r}{f \oplus (-g)} & M[j] \oplus M[j-1] \arrow[hook]{r} & \ssE^1_{0d}.
\end{tikzcd}\]
Claim \ref{claim:differential2integral.1} says that $M[j-1] = 0$, so to show that this differential kills
$M[j]$ it is enough to prove that $f\colon M[j-1,j] \rightarrow M[j]$ is surjective.  
Since the $\dA_{3e}$ in both cases of \eqref{eqn:differential2integral.1.1} has $3e$ a positive multiple of $3$,
this follows from Lemma \ref{lemma:levivanishintegral} (Levi vanishing and surjectivity).
Here we are using the fact that $\hbb(\bPhiQ(\bL_{\bDelta[j]})+1 = d$; cf.\ \eqref{eqn:differential2integral.1.2}.

It remains to deal with the case $j = 1$, so
\begin{align}
\bPhiQ(\bL_{\bDelta[j]}) &= \dA_{3d}, \label{eqn:differential2integral.2.1} \\
\hbb(\bPhiQ(\bL_{\bDelta[j]})) &= \lfloor (3d-1)/3 \rfloor = d-1. \label{eqn:differential2integral.2.2}
\end{align} 
In this case, we will use the summand $M[1,2]$ of $\ssE^1_{1d}$.  Just like
above, on this summand this differential takes the form
\[\begin{tikzcd}[column sep=large]
M[1,2] \arrow{r} & M[2] \oplus M[1] \arrow[hook]{r} & \ssE^1_{0d}.
\end{tikzcd}\]
Claim \ref{claim:differential2integral.1} says that $M[2] = 0$, and since
the $\dA_{3d}$ in \eqref{eqn:differential2integral.2.1} has $3d$ a positive multiple of $3$
Lemma \ref{lemma:levivanishintegral} (Levi vanishing and surjectivity) implies that 
the map $M[1,2] \rightarrow M[1]$ is surjective.  
Here we are using the fact that $\hbb(\bPhiQ(\bL_{\bDelta[j]})+1 = d$; cf.\ \eqref{eqn:differential2integral.2.2}.
The lemma follows.
\end{proof}

\part{The double Tits building (Theorem \ref{maintheorem:doubletits})}
\label{part:doubletits}

We defined the double Tits building $\Tits^2(\Z^{n})$ in \S \ref{section:doubletits}.
In this final part of the paper, we prove Theorem \ref{maintheorem:doubletits}, which says that
$\Tits^2(\Z^{n})$ is $n$-connected for $n \geq 4$.  
See the introductory \S \ref{section:doubletitsintro} for an outline.

\section{Outline of connectivity proof}
\label{section:doubletitsintro}

In this introductory section, we first introduce some notation, then translate Theorem \ref{maintheorem:doubletits} into a homological
statement, and finally outline our proof of it.

\subsection{Steinberg module and lines}

Let $V$ be a rank-$r$ free $\Z$-module with $r \geq 1$.  Recall that
$\Tits(V) \cong \Tits(\GL_r)$ is the complex of flags of nonzero proper direct summands of $V$
and $\St(V) = \RH_{r-2}(\Tits(V)) \cong \St(\GL_r)$.
A {\em line} in $V$ is a one-dimensional direct summand.  Each line $L$ can be written as
$L = \Span{x}$ for an $x \in V$ that is {\em primitive}, that is, only divisible by $\pm 1$.
The $x$ in $L = \Span{x}$ is unique up to multiplication by $\pm 1$.  

\subsection{Apartments}
\label{section:apartmentclasses}

For lines $L_1,\ldots,L_r \subset V$, the
{\em apartment class} $\Apart{L_1,\ldots,L_r} \in \St(V)$ is as follows.  
Let $X_r$ be the simplicial complex with vertices
$[r] = \{1,\ldots,r\}$ and simplices nonempty proper subsets of $[r]$.  The
complex $X_r$ is isomorphic to the barycentric subdivision of the boundary of an $(r-1)$-simplex.
Let $f\colon X_r \rightarrow \Tits(V)$ be the simplicial map taking the simplex corresponding
to $I \subsetneq [r]$ to $\SpanSet{$L_i$}{$i \in I$}$.  The ordering on $[r]$ gives $X_r \cong \bbS^{r-2}$
an orientation.  Let $[X_r] \in \RH_{r-2}(X_r) \cong \Z$ be the fundamental class.  Then
\[\Apart{L_1,\ldots,L_r} = f_{\ast}([X_r]) \in \RH_{r-2}(\Tits(V)) = \St(V).\]
Permuting the $L_i$ changes the orientation, so for $\sigma$ in the symmetric
group $\fS_r$ we have
\[\Apart{L_{\sigma(1)},\ldots,L_{\sigma(r)}} = (-1)^{|\sigma|} \Apart{L_1,\ldots,L_r}.\]
We have $\Apart{L_1,\ldots,L_r} = 0$ if the $L_i$ are linearly dependent, i.e., if they
do not span $V \otimes \Q$.
The Solomon--Tits theorem says that $\St(V)$ is spanned by apartment classes.  In
fact, Ash--Rudolph \cite{AshRudolph} proved that $\St(V)$ is spanned by {\em integral apartment classes}, i.e.,
apartment classes $\Apart{L_1,\ldots,L_r}$ such that the lines $L_i$ satisfy
$V = L_1 \oplus \cdots \oplus L_r$.

\begin{remark}
Assume that the $L_i$ are linearly independent.  Identifying $V$ with $\Z^r$ and $\St(V)$ with $\St(\GL_r)$, in 
the notation of \S \ref{section:apartments} the apartment $\Apart{L_1,\ldots,L_r}$ is the apartment
$\bbA_g$ with $g \in \GL_r(\Q)$ the matrix whose columns are $v_i \in \Z^r$ with $L_i = \Span{v_i}$.
\end{remark}

\subsection{Steinberg multiplication}
\label{section:steinbergmultiplication}

Let $V$ and $W$ be $\Z$-modules with $V \cong \Z^r$ and $W \cong \Z^s$.  There is a natural bilinear
multiplication $\St(V) \otimes \St(W) \rightarrow \St(V \oplus W)$.  For apartment classes
$\Apart{L_1,\ldots,L_r} \in \St(V)$ and $\Apart{L'_1,\ldots,L'_s} \in \St(W)$, their product is
\[\Apart{L_1,\ldots,L_r} \Cdot \Apart{L'_1,\ldots,L'_s} = \Apart{L_1,\ldots,L_r,L'_1,\ldots,L'_s}.\]
Here we are identifying $V$ and $W$ with the corresponding subspaces of $V \oplus W$.  This product
is classical.  One place where it appears explicitly is \cite{ChurchFarbPutmanConjecture}.  See
\cite{AshMillerPatzt} for recent work on it.
This product appeared earlier in our paper in a different guise.  Indeed, identify $V$ with $\Z^r$ and $W$ with $\Z^s$.  We
can then identify our product with a map
\[\St(\GL_r \times \GL_s) = \St(\GL_r) \otimes \St(\GL_s) = \St(\Z^r) \otimes \St(\Z^s) \longrightarrow \St(\Z^{r+s}) = \St(\GL_{r+s}).\]
This product map $\St(\GL_r \times \GL_s) \rightarrow \St(\GL_{r+s})$ 
is exactly the Reeder map described in \S \ref{section:reedermap}.

\subsection{Chain complex for Steinberg}
\label{section:barsteinberg}

For $i \geq -1$, recall from \S \ref{section:integralrefinement} that
\[\bS_i(\Z^n) = \bigoplus_{V_1 \oplus \cdots \oplus V_{i+2} = \Z^n} \St(V_1) \otimes \cdots \otimes \St(V_{i+2}).\]
Here and throughout the rest of the paper the direct sum is over all decompositions
$V_1 \oplus \cdots \oplus V_{i+1} = \Z^n$ with $V_j \neq 0$ for all $1 \leq j \leq i+2$.  As a special
case, we have $\bS_{-1}(\Z^n) = \St(\Z^n)$.  As we discussed in \S \ref{section:integralrefinement}, these
fit into a chain complex
\[\begin{tikzcd}[column sep=small, row sep=small]
0 \arrow{r} & \bS_{n-2}(\Z^{n}) \arrow{r}{\partial} & \cdots \arrow{r}{\partial} & \bS_0(\Z^{n}) \arrow{r}{\partial} & \bS_{-1}(\Z^{n}) \arrow[equals]{d} \arrow{r} & 0. \\
            &                             &                  &                         & \St(\Z^{n})                                  &
\end{tikzcd}\]
For a decomposition $V_1 \oplus \cdots \oplus V_{i+2} = \Z^n$ and $\theta_j \in \St(V_j)$ for $1 \leq j \leq i+2$, denote
$\theta_1 \otimes \cdots \otimes \theta_{i+2} \in \bS_i(\Z^n)$ by
$[\theta_1 | \ldots | \theta_{i+2}]$.  The boundary map then has the familiar form
\[\partial[\theta_1 | \ldots | \theta_{i+2}] = \sum_{j=1}^{i+1} (-1)^{j-1} [\theta_1 | \ldots | \theta_j \Cdot \theta_{j+1} | \ldots | \theta_{i+2}] \in \bS_{i-1}(\Z^n).\]

\subsection{Main theorem}

The main theorem we will prove in this part of the paper is:

\begin{primedtheorem}{maintheorem:doubletits}
\label{theorem:exactness}
For $n \geq 4$, the chain complex
\[\begin{tikzcd}
\bS_{2}(\Z^n) \arrow{r}{\partial} & \bS_1(\Z^n) \arrow{r}{\partial} & \bS_0(\Z^n) \arrow{r}{\partial} & \St(\Z^n) \arrow{r} & 0
\end{tikzcd}\]
is exact.  
\end{primedtheorem}

Before outlining the proof of Theorem \ref{theorem:exactness}, we explain why it implies
Theorem \ref{maintheorem:doubletits}:

\theoremstyle{plain}
\newtheorem*{maintheorem:doubletits}{Theorem \ref{maintheorem:doubletits}}
\begin{maintheorem:doubletits}
The complex $\Tits^2(\Z^n)$ is $n$-connected for $n \geq 4$.
\end{maintheorem:doubletits}
\begin{proof}[Proof, assuming Theorem \ref{theorem:exactness}]
Let $n \geq 4$.  Theorem \ref{theorem:doubletitshomology} says that
\[\RH_{i+n-1}(\Tits^2(\Z^n)) = \HH_i(\bS_{\bullet}(\Z^n)) \quad \text{for all $i$}.\]
Theorem \ref{theorem:exactness} implies that $\HH_i(\bS_{\bullet}(\Z^n)) = 0$ for $i \leq 1$, so
$\RH_i(\Tits^2(\Z^n)) = 0$ for $i \leq n$.
To complete the proof, it is enough to prove that
the fundamental group of $\Tits^2(\Z^n)$ is trivial.
Let $f\colon \bbS^1 \rightarrow \Tits^2(\Z^n)$
be a map.  By definition, $\Tits^2(\Z^{n})$ contains a copy of
$\Tits(\Z^{n})$.  It follows from the proof of \cite[Proposition 5.7]{MillerPatztWilsonRognes} that
$f\colon \bbS^1 \rightarrow \Tits^2(\Z^{n})$ can be homotoped such that its
image lies in $\Tits(\Z^n)$.  The Solomon--Tits theorem says that $\Tits(\Z^n)$ is $(n-3)$-connected.
Since $n \geq 4$, it is in particular $1$-connected, so $f$ is homotopic to a constant map, as desired.
\end{proof}

\subsection{Outline of rest of paper}

The rest of this paper is devoted to proving Theorem \ref{theorem:exactness}.  Our
proof has three steps.  In \S \ref{section:threestepmodular} we recall a three-step
resolution of $\St(\Z^r)$ that was proved by Br\"{u}ck--Miller--Patzt--Sroka--Wilson \cite{BruckMillerPatztSrokaWilson},
generalizing classical work on modular symbols and \cite{Bykovskii, ChurchPutmanCodimOne}.  Next, in \S \ref{section:doublecomplex}
we construct a double complex out of this $3$-step resolution and prove
that its homology agrees with that of $\bS_{\bullet}(\Z^n)$ in a range of degrees.
Finally, in \S \ref{section:threestepbar} we use this double complex to prove
Theorem \ref{theorem:exactness}.

\section{Three-step partial resolution of Steinberg from modular symbols}
\label{section:threestepmodular}

In this section, we describe a three-step resolution of the Steinberg representation that
is different from the one in Theorem \ref{theorem:exactness}.  Let $V$ be a free $\Z$-module
of rank $r \geq 1$.

\subsection{Frames and augmentations}
\label{section:framesaug}

A {\em frame} of $V$ is an ordered collection $(L_1,\ldots,L_r)$ of lines in $V$ such that
$V = L_1 \oplus \cdots \oplus L_r$.  Associated to this frame is an integral apartment class
$\Apart{L_1,\ldots,L_r} \in \St(V)$.  
A {\em partial frame} of $V$ is an ordered collection of lines in $V$
that can be extended to a frame.  Fix a partial frame $F = (L_1,\ldots,L_s)$ for
$V$.  For each $1 \leq j \leq s$, write $L_j = \Span{v_j}$ for some $v_j \in V$.  Let
$[s] = \{1,\ldots,s\}$.  An {\em augmentation} of $F$ is a set
$F' = \{L'_1,\ldots,L'_{s'}\}$ of lines in $V$ constructed as follows:
\begin{itemize}
\item Choose disjoint subsets $I_1,\ldots,I_m \subset [s]$ such that each $I_j$ has either $2$ or $3$ elements.
\item For each $2$-element set $I_j = \{p_1,p_2\}$, pick signs $\epsilon_1,\epsilon_2 \in \{\pm 1\}$ and
add the line $\Span{\epsilon_1 v_{p_1} +\epsilon_2 v_{p_2}}$ to $F'$.
\item For each $3$-element set $I_j = \{p_1,p_2,p_3\}$, pick signs $\epsilon_1,\epsilon_2,\epsilon_3 \in \{\pm 1\}$
and add to $F'$ either:
\begin{itemize}
\item the line $\Span{\epsilon_1 v_{p_1} +\epsilon_2 v_{p_2} + \epsilon_3 v_{p_3}}$; or
\item the two lines $\Span{\epsilon_1 v_{p_1} +\epsilon_2 v_{p_2}}$ and $\Span{\epsilon_1 v_{p_1} +\epsilon_2 v_{p_2} + \epsilon_3 v_{p_3}}$.
\end{itemize}
\end{itemize}
An {\em augmented partial frame} is an ordered collection of lines that as an unordered set is the
union of a partial frame $F$ and some augmentation $F'$ of $F$.  An {\em unordered augmented partial frame}
is the unordered set underlying an augmented partial frame. 
Here is an example:

\begin{example}
Assume that $V$ has rank $15$, and let $\{v_1,\ldots,v_{15}\}$ be a basis of $V$.  The following is then an augmented
partial frame of $V$:
\begin{align*}
&(\Span{v_1}, \Span{v_2}, \Span{v_1 - v_2}, \Span{v_3}, \Span{v_4},\Span{v_5}, \Span{-v_3+v_4-v_5}, \\
&\quad \quad \Span{v_6},\Span{v_7},\Span{v_8},\Span{-v_6+v_7},\Span{-v_6+v_7+v_8}, \Span{v_9},\Span{v_{10}},\Span{v_9+v_{10}},\Span{v_{11}}). \qedhere
\end{align*}
\end{example}

By construction, the collection of augmented partial frames of $V$ is closed under re-ordering and passing to subsets.  Also,
the span of the lines in an augmented partial frame of $V$ is a direct summand of $V$.

\subsection{Chain complex}
\label{section:chaincomplex}

Recall that $V$ is a free $\Z$-module of rank $r \geq 1$.
For $i \geq 0$, define $\bX_i(V)$ to be the free abelian group given by generators
and relations as follows.  The generators of $\bX_i(V)$ are formal symbols
$\Gen{L_1,\ldots,L_{r+i}}$, where $(L_1,\ldots,L_{r+i})$ is an $(r+i)$-element
augmented partial frame for $V$.  The relations of $\bX_i(V)$ are:
\begin{itemize}
\item For a generator $\Gen{L_1,\ldots,L_{r+i}}$ such that the $L_j$ do
not span\footnote{Since the span of the $L_j$ is a direct summand of $V$, it would be equivalent
to require that the $L_j$ do not span $V$.} $V \otimes \Q$, the relation $\Gen{L_1,\ldots,L_{r+i}} = 0$.
\item For a generator $\Gen{L_1,\ldots,L_{r+i}}$ and $\sigma \in \fS_{r+i}$, the relation
$\Gen{L_{\sigma(1)},\ldots,L_{\sigma(r+i)}} = (-1)^{|\sigma|} \Gen{L_1,\ldots,L_{r+i}}$.
\end{itemize}
There is a differential $\delta\colon \bX_i(V) \rightarrow \bX_{i-1}(V)$ defined in
the usual way:\footnote{Here we are using the fact that
the set of augmented partial frames is closed under reordering and passing to subsets, which we observed at the end of \S \ref{section:framesaug}.}
\[\delta\Gen{L_1,\ldots,L_{r+i}} = \sum_{j=1}^{r+i} (-1)^{j-1} \Gen{L_1,\ldots,\widehat{L_j},\ldots,L_{r+i}}.\]
This makes sense since as is easily checked $\delta$ takes relations to relations.
We therefore have a chain complex $(\bX_{\bullet}(V),\delta)$.  The low-degree
terms of this chain complex have the following descriptions:

\begin{example}
\label{example:x0}
The nonzero generators of $\bX_0(V)$ are exactly the
$\Gen{\Span{v_1},\ldots,\Span{v_r}}$ such that $\{v_1,\ldots,v_r\}$
is a basis for $V \cong \Z^r$.
\end{example}

\begin{example}
\label{example:x1}
If $r = 1$, then $\bX_1(V) = 0$.  If $r \geq 1$, then
the nonzero generators of $\bX_1(V)$ are exactly those
that after reordering their entries (and thus introducing a
sign) are of one of the following two forms for some
basis $\{v_1,\ldots,v_r\}$ of $V$:
\begin{itemize}
\item $\Gen{\Span{v_1},\ldots,\Span{v_r},\Span{\epsilon_1 v_1 + \epsilon_2 v_2}}$
for signs $\epsilon_1,\epsilon_2 \in \{\pm 1\}$.
\item $\Gen{\Span{v_1},\ldots,\Span{v_r},\Span{\epsilon_1 v_1 + \epsilon_2 v_2+\epsilon_3 v_3}}$
for signs $\epsilon_1,\epsilon_2,\epsilon_3 \in \{\pm 1\}$.  This second kind of generator
only appears when $r \geq 3$.\qedhere
\end{itemize}
\end{example}

\begin{example}
\label{example:x2}
If $r \leq 2$, then $\bX_2(V) = 0$.  If $r \geq 3$, then
the nonzero generators of $\bX_2(V)$ are exactly those
that after reordering their entries (and thus introducing a
sign) are of one of the following four forms for some
basis $\{v_1,\ldots,v_r\}$ of $V$:
\begin{itemize}
\item $\Gen{\Span{v_1},\ldots,\Span{v_r},\Span{\epsilon_1 v_1+\epsilon_2 v_2},\Span{\epsilon_1 v_1 + \epsilon_2 v_2+\epsilon_3 v_3}}$
for signs $\epsilon_1,\epsilon_2,\epsilon_3 \in \{\pm 1\}$.
\item $\Gen{\Span{v_1},\ldots,\Span{v_r},\Span{\epsilon_1 v_1+\epsilon_2 v_2},\Span{\epsilon_3 v_3 + \epsilon_4 v_4}}$
for signs $\epsilon_1,\ldots,\epsilon_4 \in \{\pm 1\}$.  This second kind of generator
only appears when $r \geq 4$.
\item $\Gen{\Span{v_1},\ldots,\Span{v_r},\Span{\epsilon_1 v_1+\epsilon_2 v_2+\epsilon_3 v_3},\Span{\epsilon_4 v_4 + \epsilon_5 v_5}}$
for signs $\epsilon_1,\ldots,\epsilon_5 \in \{\pm 1\}$.  This third kind of generator
only appears when $r \geq 5$.
\item $\Gen{\Span{v_1},\ldots,\Span{v_r},\Span{\epsilon_1 v_1+\epsilon_2 v_2+\epsilon_3 v_3},\Span{\epsilon_4 v_4 + \epsilon_5 v_5 + \epsilon_6 v_6}}$
for signs $\epsilon_1,\ldots,\epsilon_6 \in \{\pm 1\}$.  This fourth kind of generator
only appears when $r \geq 6$.\qedhere
\end{itemize}
\end{example}

\subsection{Resolution of Steinberg}

There is a map $\psi\colon \bX_0(V) \rightarrow \St(V)$ taking a generator
$\Gen{L_1,\ldots,L_r}$ to the corresponding apartment class $\Apart{L_1,\ldots,L_r}$.  This
map makes sense since the two relations in $\bX_0(V)$ go to relations between apartment
classes in $\St(V)$; cf.\ the relations in \S \ref{section:apartmentclasses}.
Generalizing classical work of Manin about modular symbols, Bykovskii \cite{Bykovskii}
proved\footnote{Actually, Bykovskii had fewer relations than we do, but adding more
relations is no problem for a presentation.} that the image of the differential $\delta\colon \bX_1(V) \rightarrow \bX_0(V)$ is
contained in $\ker(\psi)$ and that the resulting chain complex
\begin{equation}
\label{eqn:bykovskii}
\begin{tikzcd}
\bX_1(V) \arrow{r}{\delta} & \bX_0(V) \arrow{r}{\psi} & \St(V) \arrow{r} & 0
\end{tikzcd}
\end{equation}
is exact.  In other words, \eqref{eqn:bykovskii} is a presentation of $\St(V)$.  See
\cite{SteinbergPolylogarithms, ChurchPutmanCodimOne} for alternate proofs.  This was generalized
by Br\"{u}ck--Miller--Patzt--Sroka--Wilson \cite{BruckMillerPatztSrokaWilson}.  The
following is a small variant on their result:

\begin{theorem}
\label{theorem:twostepmodular}
For a free $\Z$-module $V$ of rank $r \geq 1$, we have an exact sequence
\[\begin{tikzcd}
\bX_2(V) \arrow{r}{\delta} & \bX_1(V) \arrow{r}{\delta} & \bX_0(V) \arrow{r}{\psi} & \St(V) \arrow{r} & 0.
\end{tikzcd}\]
\end{theorem}
\begin{proof}
In light of Examples \eqref{example:x0} -- \eqref{example:x2}, the indicated resolution
has the same generators $\bX_0(V)$ and relations $\bX_1(V)$ as the resolution from \cite{BruckMillerPatztSrokaWilson},
but $\bX_2(V)$ is larger than the corresponding term from \cite{BruckMillerPatztSrokaWilson}.  However, since
we are not claiming anything about the kernel of $\delta\colon \bX_2(V) \rightarrow \bX_1(V)$, this does
not affect the result we are trying to prove.
\end{proof}

\begin{remark}
The elements of $\bX_2(V)$ that do not appear in \cite{BruckMillerPatztSrokaWilson} are there
to allow the multiplication we define in \S \ref{section:complexmultiplication} below.
\end{remark}

\section{Double complex of resolutions}
\label{section:doublecomplex}

We now assemble the $\bX_{\bullet}(V)$ into a double complex by imitating the construction
of the bar resolution $\bS_{\bullet}(\Z^n)$ of $\St(\Z^n)$ discussed in \S \ref{section:barsteinberg}.

\subsection{Complex multiplication}
\label{section:complexmultiplication}

Let $V$ and $W$ be $\Z$-modules with $V \cong \Z^r$ and $W \cong \Z^s$.  Just 
like for the Steinberg representation (cf.\ \S \ref{section:steinbergmultiplication}), there is a natural bilinear
multiplication $\bX_i(V) \otimes \bX_j(W) \rightarrow \bX_{i+j}(V \oplus W)$ given as follows.
Consider generators $\Gen{L_1,\ldots,L_{r+i}}$ and $\Gen{L'_1,\ldots,L'_{s+j}}$ of
$\bX_i(V)$ and $\bX_j(W)$.  We then define
\[\Gen{L_1,\ldots,L_{r+i}} \Cdot \Gen{L'_1,\ldots,L'_{s+j}} = \Gen{L_1,\ldots,L_{r+i},L'_1,\ldots,L'_{s_j}} \in \bX_{i+j}(V \oplus W).\]
Here we are identifying $V$ and $W$ with the corresponding subspaces of $V \oplus W$.  We remark that
this product is why we defined augmented partial frames like we did instead of just including
the elements needed by \cite{BruckMillerPatztSrokaWilson}.  This product is compatible with the differentials
$\delta$ on our chain complexes in the sense that 
\begin{equation}
\label{eqn:deltaderivation}
\delta(\kappa \Cdot \kappa') = \delta(\kappa) \Cdot \kappa' + (-1)^{r+i} \kappa \Cdot \delta(\kappa') \quad \text{for $\kappa \in \bX_i(V)$ and $\kappa' \in \bX_j(W)$}.
\end{equation}

\subsection{Double complex}
\label{section:doublecomplexdef}

For $p \geq -1$ and $q \geq 0$, define
\[\bX_{pq}(\Z^n) = \bigoplus_{V_1 \oplus \cdots \oplus V_{p+2} = \Z^n} \left(\bigoplus_{i_1 + \cdots + i_{p+2} = q} \bX_{i_1}(V_1) \otimes \cdots \otimes \bX_{i_{p+2}}(V_{p+2})\right).\]
Here the direct sums are over decompositions $V_1 \oplus \cdots \oplus V_{p+2} = \Z^n$ and $i_1 + \cdots + i_{p+2} = q$ such that
$V_j \neq 0$ and $i_j \geq 0$ for all $1 \leq j \leq p+2$.  Fixing such decompositions $V_1 \oplus \cdots \oplus V_{p+2} = \Z^n$ and $i_1 + \cdots + i_{p+2} = q$, consider elements
$\kappa_j \in \bX_{i_j}(V_j)$ for $1 \leq j \leq p+2$.  We will denote the element of
$\bX_{pq}(\Z^n)$ corresponding to $\kappa_1 \otimes \cdots \otimes \kappa_{p+2}$ by $\Gen{\kappa_1| \ldots | \kappa_{p+2}}$.
When we want to write specific elements of $\bX_{pq}(\Z^n)$, we will use the following convention:

\begin{convention}
\label{convention:omitbrackets}
For generators $\kappa_j$ of $\bX_{i_j}(V_j)$, we will omit the inner brackets
when writing $\Gen{\kappa_1| \ldots | \kappa_{p+2}}$.  For instance, for $p=1$ if
$\kappa_1 = \Gen{L_1,L_2}$ and $\kappa_2 = \Gen{L_3,L_4,L_5}$ and $\kappa_3 = \Gen{L_6}$, then
we will write $\Gen{L_1,L_2|L_3,L_4,L_5|L_6}$ instead of
$\Gen{\Gen{L_1,L_2}|\Gen{L_3,L_4,L_5}|\Gen{L_6}}$.
\end{convention}

Let $r_j = \rank(V_j)$.  Define two differentials $\partial\colon \bX_{pq}(\Z^n) \rightarrow \bX_{p-1,q}(\Z^n)$ and 
$\delta\colon \bX_{pq}(\Z^n) \rightarrow \bX_{p,q-1}(\Z^n)$
via the formulas
\begin{align*}
\partial \Gen{\kappa_1| \ldots | \kappa_{p+2}} &= \sum_{j=1}^{p+1} (-1)^{j-1} \Gen{\kappa_1| \ldots | \kappa_j \Cdot \kappa_{j+1} | \ldots | \kappa_{p+2}} \\
\delta   \Gen{\kappa_1| \ldots | \kappa_{p+2}} &= \sum_{j=1}^{p+2} (-1)^{r_1+i_1+\cdots+r_{j-1}+i_{j-1}} \Gen{\kappa_1 | \ldots | \delta(\kappa_j) | \ldots | \kappa_{p+2}}.
\end{align*}
Using \eqref{eqn:deltaderivation}, these differentials make $\bX_{\bullet,\bullet}(\Z^n)$ into a double complex.

\subsection{Double complex and Steinberg}

We now return to the chain complex
\[\begin{tikzcd}[column sep=small, row sep=small]
0 \arrow{r} & \bS_{n-2}(\Z^{n}) \arrow{r}{\partial} & \cdots \arrow{r}{\partial} & \bS_0(\Z^{n}) \arrow{r}{\partial} & \bS_{-1}(\Z^{n}) \arrow[equals]{d} \arrow{r} & 0. \\
            &                             &                  &                         & \St(\Z^{n})                                  &
\end{tikzcd}\]
Our main result in this section is:

\begin{lemma}
\label{lemma:doublecomplexsteinberg}
For all $n \geq 1$, we have $\HH_i(\bX_{\bullet,\bullet}(\Z^n)) \cong \HH_i(\bS_{\bullet}(\Z^n))$ for $i \leq 0$.  Moreover,
$\HH_1(\bX_{\bullet,\bullet}(\Z^n))$ surjects onto $\HH_1(\bS_{\bullet}(\Z^n))$.
\end{lemma}
\begin{proof}
Consider the spectral sequence computing $\HH_i(\bX_{\bullet,\bullet}(\Z^n))$ where we take
homology with respect to the differential $\delta\colon \bX_{pq}(\Z^n) \rightarrow \bX_{p,q-1}(\Z^n)$ first,
so
\[\ssE^1_{pq} = \HH_q(\bX_{p,\bullet}(\Z^n)) \Rightarrow \HH_{p+q}(\bX_{\bullet,\bullet}(\Z^n)).\]
To prove the lemma, it is enough to prove that the bottom rows of the $\ssE^1$-page of this spectral sequence
are as follows: 
\begin{center}
\begin{tblr}{vline{2}={1-Y}{}}
$q=2$ & $\ast$           & $\leftarrow$                      & $\ast$        & $\leftarrow$                      & $\ast$        & $\leftarrow$                      & $\ast$        & $\leftarrow$                      & $\cdots$ \\
$q=1$ & $0$              & $\leftarrow$                      & $\ast$        & $\leftarrow$                      & $\ast$        & $\leftarrow$                      & $\ast$        & $\leftarrow$                      & $\cdots$ \\
$q=0$ & $\bS_{-1}(\Z^n)$ & $\stackrel{\partial}{\leftarrow}$ & $\bS_0(\Z^n)$ & $\stackrel{\partial}{\leftarrow}$ & $\bS_1(\Z^n)$ & $\stackrel{\partial}{\leftarrow}$ & $\bS_2(\Z^n)$ & $\stackrel{\partial}{\leftarrow}$ & $\cdots$ \\
\cline{2-10}
      & $p=-1$           &                                   & $p=0$         &                                   & $p=1$         &                                   & $p=2$         &                                   &
\end{tblr}
\end{center}
The $\ast$ entries are unspecified.  To do this, it is enough to prove the following claim.  This claim will show that the entries have
the indicated form (and in fact all the entries on the $q=1$ line vanish).  That the differentials are correct will be immediate from the definitions.

\begin{unnumberedclaim}
Fix $p \geq -1$.  For $q \geq 0$, we have
$\HH_q(\bX_{p,\bullet}(\Z^n)) \cong
\begin{cases}
\bS_p(\Z^n) & \text{if $q=0$},\\
0           & \text{if $q=1$}.
\end{cases}$
\end{unnumberedclaim}

For a decomposition $V_1 \oplus \cdots \oplus V_{p+2} = \Z^n$, let
\[\bX_{pq}(V_1,\ldots,V_{p+2}) = \bigoplus_{i_1 + \cdots + i_{p+2} = q} \bX_{i_1}(V_1) \otimes \cdots \otimes \bX_{i_{p+2}}(V_{p+2}).\]
This forms a subcomplex of the chain complex $\bX_{p,\bullet}(\Z^n)$, and
\[\bX_{p,\bullet}(\Z^n) = \bigoplus_{V_1 \oplus \cdots \oplus V_{p+2} = \Z^n} \bX_{p,\bullet}(V_1,\ldots,V_{p+2}).\]
Since
\[\bS_p(\Z^n) = \bigoplus_{V_1 \oplus \cdots \oplus V_{p+2} = \Z^n} \St(V_1) \otimes \cdots \otimes \St(V_{p+2}),\]
it is enough to prove that for a fixed decomposition $V_1 \oplus \cdots \oplus V_{p+2} = \Z^n$, we have
\begin{equation}
\label{eqn:doublecomplextoprove}
\HH_q(\bX_{p,\bullet}(V_1,\ldots,V_{p+2})) \cong
\begin{cases}
\St(V_1) \otimes \cdots \otimes \St(V_{p+2}) & \text{if $q=0$},\\
0                                            & \text{if $q=1$}.
\end{cases}
\end{equation}
For each $1 \leq j \leq p+2$, we have a chain complex $\bX_{\bullet}(V_j)$, and Theorem \ref{theorem:twostepmodular} says that
\begin{equation}
\label{eqn:baseoftensorproduct}
\HH_q(\bX_{\bullet}(V_j)) \cong \begin{cases}
\St(V_j) & \text{if $q=0$},\\
0        & \text{if $q=1$}.
\end{cases}
\end{equation}
It is almost the case that
\begin{equation}
\label{eqn:tensorproductchain}
\bX_{p,\bullet}(V_1,\ldots,V_{p+2}) = \bX_{\bullet}(V_1) \otimes \cdots \otimes \bX_{\bullet}(V_{p+2}).
\end{equation}
The only issue is that while this tensor product has the same terms as $\bX_{p,\bullet}(V_1,\ldots,V_{p+2})$, its differentials
are not right.  To fix this, let $r_j = \rank(V_j)$ and for each $2 \leq j \leq p+2$ change
$\bX_{\bullet}(V_j)$ by multiplying its differential by $(-1)^{r_1 + \cdots + r_{j-1}}$.  This does not affect
\eqref{eqn:baseoftensorproduct}, and now \eqref{eqn:tensorproductchain} holds.  Since each $\bX_q(V_j)$ is a free $\Z$-module,
the desired result \eqref{eqn:doublecomplextoprove} now follows from \eqref{eqn:baseoftensorproduct} and \eqref{eqn:tensorproductchain}
along with the K\"{u}nneth formula.
\end{proof}

\section{Three-step bar resolution of Steinberg}
\label{section:threestepbar}

In this final section we first establish some notation and then
prove Theorem \ref{theorem:exactness}.

\subsection{Restricted bar resolution of a set}

Let $S$ be a nonempty finite set.  Fix a set $\fR$ of
disjoint nonempty subsets of $S$.
An {\em ordered subset} of $S$ is a subset of $S$ equipped with a linear order.
We will talk about ordered subsets of $S$ using ordinary
set notation and terminology; for instance, if $A$ is an
ordered subset of $S$ and $R \in \fR$ then we will say that $R \subset A$
if each $r \in R$ appears in $A$.  For disjoint ordered
subsets $A$ and $A'$ of $S$, write $A \Cdot A'$ for the
ordered subset of $S$ obtained by concatenating $A$ and $A'$.

For $i \geq -1$, define 
$\bZ_i(S,\fR)$ to be the free abelian group given by generators
and relations as follows.  The generators of $\bZ_i(S,\fR)$ are formal symbols
$[A_1|\ldots|A_{i+2}]$ where:
\begin{itemize}
\item the $A_j$ are disjoint nonempty ordered subsets of $S$; and
\item each $s \in S$ appears exactly once in $A_1 \Cdot \ldots \Cdot A_{i+2}$; and
\item for each $R \in \fR$ there exists some $1 \leq j \leq i+2$ such that $R \subset A_j$.
\end{itemize}
The relations of $\bZ_i(S,\fR)$ are:
\begin{itemize}
\item Let $[A_1|\ldots|A_{i+2}]$ be a generator of $\bZ_i(S,\fR)$.  For each
$1 \leq j \leq i+2$, let $A'_j$ be an ordered set obtained by permuting
the ordered set $A_j$ and let $\lambda_j \in \{\pm 1\}$ be the sign
of this permutation.  Then
\[[A'_1|\ldots|A'_{i+2}] = (-1)^{\lambda_1 + \cdots + \lambda_{i+2}} [A_1|\ldots|A_{i+2}].\]
\end{itemize}
There is a differential $\partial\colon \bZ_i(S,\fR) \rightarrow \bZ_{i-1}(S,\fR)$ defined in
the usual way:
\[\partial[A_1|\ldots|A_{i+2}] = \sum_{j=1}^{i+1} (-1)^{j-1} [A_1|\ldots|A_j \Cdot A_{j+1}| \ldots |A_{i+2}].\]
This makes $\bZ_{\bullet}(S,\fR)$ into a chain complex.  We have:

\begin{lemma}
\label{lemma:barset}
Let $S$ be a nonempty finite set and $\fR$ be a set of
disjoint nonempty subsets of $S$.  Set 
\[d = |S \setminus \bigsqcup_{R \in \fR} R| + |\fR|.\]
Then
\[\HH_i(\bZ_{\bullet}(S,\fR)) = \begin{cases}
\Z & \text{if $i=d-2$},\\
0  & \text{otherwise}.
\end{cases}\]
\end{lemma}
\begin{proof}
To simplify our notation, for each $s \in S$ with $s \notin \sqcup_{R \in \fR} R$
add the one-element set $\{s\}$ to $\fR$.  This does not change
$\bZ_{\bullet}(S,\fR)$ or $d$, and ensures that $S = \sqcup_{R \in \fR} R$
and $d = |\fR|$.

Let $\bW(\fR)$ be the poset of nonempty proper subsets of $\fR$.
Viewed as a semisimplicial set, $\bW(\fR)$ is isomorphic to the barycentric
subdivision of the boundary of the simplex with $d=|\fR|$ vertices.  This
implies that its geometric realization $|\bW(\fR)|$ is homeomorphic to the
sphere $\bbS^{d-2}$.  Letting $\RC_{\bullet}(\bW(\fR))$ be the reduced chain
complex of the semisimplicial set $\bW(\fR)$, it follows that
\[\HH_i(\RC_{\bullet}(\bW(\fR))) = \begin{cases}
\Z & \text{if $i=d-2$},\\
0  & \text{otherwise}.
\end{cases}\]
It is therefore enough to prove that $\bZ_{\bullet}(S,\fR)$ is isomorphic
to $\RC_{\bullet}(\bW(\fR))$.

Fix a total ordering on $S$.  The desired isomorphism
$\phi\colon \bZ_{\bullet}(S,\fR) \rightarrow \RC_{\bullet}(\bW(\fR))$ is as follows.
Consider a generator $[A_1|\ldots|A_{i+2}]$ of $\bZ_i(S,\fR)$.  The ordered
set $A_1 \Cdot \ldots \Cdot A_{i+2}$ is a permutation of the ordered set $S$.
Let $\epsilon \in \{\pm 1\}$ be the sign of this permutation.  Also, for
$1 \leq j \leq i+2$ set $B_j = \Set{$R \in \fR$}{$R \subset A_j$}$.  Let
$\sigma$ be the following $i$-simplex of $\bW(\fR)$:
\[B_1 \subsetneq (B_1 \sqcup B_2) \subsetneq \ldots \subsetneq (B_1 \sqcup \ldots \sqcup B_{i+1}).\]
Note that if $i=-1$ then $\sigma$ is the empty simplex, which we view as the unique $(-1)$-simplex used
to form the reduced chain complex $\RC_{\bullet}(\bW(\fR))$.
We then define $\phi([A_1|\ldots|A_{i+2}]) = \epsilon \sigma \in \RC_{i}(\bW(\fR))$.  It is clear
that this is a well-defined isomorphism of chain complexes.
\end{proof}

\subsection{Proof of main theorem}

We close this paper by proving Theorem \ref{theorem:exactness}:

\theoremstyle{plain}
\newtheorem*{theorem:exactness}{Theorem \ref{theorem:exactness}}
\begin{theorem:exactness}
For $n \geq 4$, the chain complex
\[\begin{tikzcd}
\bS_{2}(\Z^n) \arrow{r}{\partial} & \bS_1(\Z^n) \arrow{r}{\partial} & \bS_0(\Z^n) \arrow{r}{\partial} & \St(\Z^n) \arrow{r} & 0
\end{tikzcd}\]
is exact.  
\end{theorem:exactness}
\begin{proof}
This theorem asserts that $\HH_i(\bS_{\bullet}(\Z^n)) = 0$ for $i \leq 1$. 
By Lemma \ref{lemma:doublecomplexsteinberg}, it is enough to prove
that $\HH_i(\bX_{\bullet,\bullet}(\Z^n)) = 0$ for $i \leq 1$.
Consider the spectral sequence computing $\HH_i(\bX_{\bullet,\bullet}(\Z^n))$ where we take
homology with respect to the differential $\partial\colon \bX_{pq}(\Z^n) \rightarrow \bX_{p-1,q}(\Z^n)$ first,
so
\[\ssE^1_{qp} = \HH_p(\bX_{\bullet,q}(\Z^n)) \Rightarrow \HH_{p+q}(\bX_{\bullet,\bullet}(\Z^n)).\]
We will prove the bottom rows of the $\ssE^1$-page of this spectral sequence are as follows:
\begin{center}
\begin{tblr}{vline{2}={1-Y}{}} 
$p=1$  & $0$  & $\leftarrow$ & $\ast$         & $\leftarrow$                    & $\ast$        & $\leftarrow$ & $\cdots$ \\ 
$p=0$  & $0$  & $\leftarrow$ & $\ssE^1_{10}$  & $\stackrel{\delta}{\leftarrow}$ & $\ssE^1_{20}$ & $\leftarrow$ & $\cdots$ \\                
$p=-1$ & $0$  & $\leftarrow$ & $0$            & $\leftarrow$                    & $0$           & $\leftarrow$ & $\cdots$ \\   
\cline{2-10} 
      & $q=0$ &              & $q=1$          &                                 & $q=2$         &              &         
\end{tblr}
\end{center}
Here the $\ast$ entries are unspecified and the entry $\ssE^1_{10}$ is $0$ for $n \geq 5$.  This
will prove the theorem for $n \geq 5$.  For $n=4$, we will also prove that the indicated differential
$\delta\colon \ssE^1_{20} \rightarrow \ssE^1_{10}$ is surjective, so $\ssE^2_{10} = 0$.  This will prove
the theorem for $n=4$.

For an unordered\footnote{Since $S$ is not ordered, we write it like a set rather than an ordered set.} augmented partial frame $S = \{L_1,\ldots,L_{n+q}\}$ of $\Z^n$ that
spans $\Z^n$, denote by $\bX_{pq}[S]$ the subspace of $\bX_{pq}(\Z^n)$ spanned
by nonzero generators $\kappa$ such that you get $\pm \Gen{L_1,\ldots,L_{n+q}}$
if you delete the $(p+1)$ occurrences of the bar symbol $|$ from $\kappa$.
Note that in such $\kappa$ the $L_j$ might appear in a different
order than in $\Gen{L_1,\ldots,L_{n+q}}$.  These form a chain subcomplex
$\bX_{\bullet,q}[S]$ of $\bX_{\bullet,q}(\Z^n)$, and
$\bX_{\bullet,q}(\Z^n)$ is the direct sum of
its subcomplexes $\bX_{\bullet,q}[S]$.  We must prove
the following four claims:

\begin{claim}{0}
\label{claim:exactness.0}
Let $S = \{L_1,\ldots,L_n\}$ be an unordered augmented partial frame of $\Z^n$
that spans $\Z^n$.  Then 
$\HH_i(\bX_{\bullet,0}[S]) = 0$ for $i \leq 1$ and $n \geq 4$.
\end{claim}

By definition, we have
$\bX_{\bullet,0}[S] \cong \bZ_{\bullet}(S,\emptyset)$.  It therefore
follows from Lemma \ref{lemma:barset} that $\HH_i(\bX_{\bullet,0}[S]) = 0$
for $i \leq n-3$.  For $n \geq 4$ this vanishing holds for $i \leq 1$, as desired.

\begin{claim}{1}
\label{claim:exactness.1}
Let $S = \{L_1,\ldots,L_{n+1}\}$ be an unordered augmented partial frame of $\Z^n$
that spans $\Z^n$.  Then $\HH_i(\bX_{\bullet,1}[S]) = 0$ in the following two situations:
\begin{itemize}
\item if $i \leq 0$ and $n \geq 5$; or
\item if $i \leq -1$ and $n = 4$.
\end{itemize}
\end{claim}

As we described in Example \ref{example:x1}, there exists a basis
$\{v_1,\ldots,v_n\}$ for $\Z^n$ such that $S$
is of one of the following two forms:
\begin{itemize}
\item[(i)] we have $S =\{\Span{v_1},\ldots,\Span{v_n},\Span{\epsilon_1 v_1 + \epsilon_2 v_2}\}$
with $\epsilon_1,\epsilon_2 \in \{\pm 1\}$.
\item[(ii)] we have $S =\{\Span{v_1},\ldots,\Span{v_n},\Span{\epsilon_1 v_1 + \epsilon_2 v_2+\epsilon_3 v_3}\}$
with $\epsilon_1,\epsilon_2,\epsilon_3 \in \{\pm 1\}$.
\end{itemize}
Assume first that we are in case (i), so
$S = \{\Span{v_1},\ldots,\Span{v_n},\Span{\epsilon_1 v_1 + \epsilon_2 v_2}\}$
with $\epsilon_1,\epsilon_2 \in \{\pm 1\}$.  In the generators
of $\bX_{\bullet,1}[S]$, the bar symbols $|$ cannot
separate the three terms in
$R = \{\Span{v_1},\Span{v_2},\Span{\epsilon_1 v_1+\epsilon_2 v_2}\}$.
It follows that
$\bX_{\bullet,1}[S] \cong \bZ_{\bullet}(S,\{R\})$.
Lemma \ref{lemma:barset} thus implies that
$\HH_i(\bX_{\bullet,1}[S]) = 0$
for 
\[i \leq d-3 \quad \text{for $d = (n+1 - 3) + 1 = n-1$}.\]
For $n \geq 4$ this vanishing holds for $i \leq 0$, giving the claim.

Assume next that we are in case (ii), so
$S = \{\Span{v_1},\ldots,\Span{v_n},\Span{\epsilon_1 v_1 + \epsilon_2 v_2+\epsilon_3 v_3}\}$
with $\epsilon_1,\epsilon_2,\epsilon_3 \in \{\pm 1\}$.  Letting
\[R' = \{\Span{v_1},\Span{v_2},\Span{v_3},\Span{\epsilon_1 v_1 + \epsilon_2 v_2+\epsilon_3 v_3}\},\]
just like in the previous paragraph we have
$\bX_{\bullet,1}[S] \cong \bZ_{\bullet}(S,\{R'\})$.
It therefore follows from Lemma \ref{lemma:barset} that $\HH_i(\bX_{\bullet,1}[S]) = 0$
for 
\[i \leq d-3 \quad \text{for $d = (n+1 - 4) + 1 = n-2$}.\]
For $n=4$ this vanishing holds for $i=-1$, while for $n \geq 5$ this vanishing
holds for $i \leq 0$.  

Assume now that $n=4$.  For later use, note that Lemma \ref{lemma:barset} also says that
\begin{equation}
\label{eqn:badcase}
\HH_0(\bX_{\bullet,1}[S]) \cong \Z.
\end{equation}
We can write down a generator for $\HH_0(\bX_{\bullet,1}[S])$ as follows.  For $1 \leq j \leq 4$, 
let $L_j = \Span{v_j}$.  Define $L_{123} = \Span{\epsilon_1 v_1 + \epsilon_2 v_2+\epsilon_3 v_3}$.
Using Convention \S \ref{convention:omitbrackets}, let
\[\kappa = \Gen{L_4|L_1, L_2, L_3, L_{123}} - \Gen{L_1,L_2,L_3,L_{123} | L_4} \in \bX_{01}(\Z^4).\]
We have
\begin{align*}
\partial \kappa &= \Gen{L_4,L_1,L_2,L_3,L_{123}} - \Gen{L_1,L_2,L_3,L_{123},L_4} \\
                &= \Gen{L_1,L_2,L_3,L_{123},L_4} - \Gen{L_1,L_2,L_3,L_{123},L_4} = 0,
\end{align*}
so we have $[\kappa] \in \HH_0(\bX_{\bullet,1}(\Z^4))$.  It follows from the
proof of Lemma \ref{lemma:barset} that $[\kappa]$ generates $\HH_0(\bX_{\bullet,1}[S]) \cong \Z$.

\begin{claim}{2}
\label{claim:exactness.2}
Let $S=\{L_1,\ldots,L_{n+2}\}$ be an augmented partial frame of $\Z^n$
that spans $\Z^n$.  Then
$\HH_{-1}(\bX_{\bullet,2}[S]) = 0$ for $n \geq 4$.
\end{claim}

Assume that $n \geq 4$.  As we described in Example \ref{example:x2}, there exists a basis
$\{v_1,\ldots,v_n\}$ for $\Z^n$ such that the unordered set $S$
is of one of the following four forms:
\begin{itemize}
\item[(i)] we have 
$S = \{\Span{v_1},\ldots,\Span{v_n},\Span{\epsilon_1 v_1+\epsilon_2 v_2},\Span{\epsilon_1 v_1 + \epsilon_2 v_2+\epsilon_3 v_3}\}$
with $\epsilon_1,\epsilon_2,\epsilon_3 \in \{\pm 1\}$.
\item[(ii)] we have 
$S=\{\Span{v_1},\ldots,\Span{v_n},\Span{\epsilon_1 v_1+\epsilon_2 v_2},\Span{\epsilon_3 v_3 + \epsilon_4 v_4}\}$
with $\epsilon_1,\ldots,\epsilon_4 \in \{\pm 1\}$.  
\item[(iii)] we have
$S=\{\Span{v_1},\ldots,\Span{v_n},\Span{\epsilon_1 v_1+\epsilon_2 v_2+\epsilon_3 v_3},\Span{\epsilon_4 v_4 + \epsilon_5 v_5}\}$
with $\epsilon_1,\ldots,\epsilon_5 \in \{\pm 1\}$.  This third case
only appears when $n \geq 5$.
\item[(iv)] we have
$S=\{\Span{v_1},\ldots,\Span{v_n},\Span{\epsilon_1 v_1+\epsilon_2 v_2+\epsilon_3 v_3},\Span{\epsilon_4 v_4 + \epsilon_5 v_5 + \epsilon_6 v_6}\}$
with $\epsilon_1,\ldots,\epsilon_6 \in \{\pm 1\}$.  This fourth case
only appears when $n \geq 6$.
\end{itemize}
The analysis of each of these cases uses Lemma \ref{lemma:barset} just like in Claim \ref{claim:exactness.1}.
In each case, there is some $\fR$ such that
\[\bX_{\bullet,2}[S] \cong \bZ_{\bullet}(S,\fR).\]
We summarize the $\fR$ and the resulting vanishing range
\[i \leq ((n+2)-\left|\bigsqcup_{R \in \fR} R\right|) + |\fR| - 3\] 
for $\HH_i(\bX_{\bullet,2}[S]) \cong \HH_i(\bZ_{\bullet}(S,\fR))$ in the following table:
\begin{center}
\begin{tblr}{l|l|l}
      & {\bf $\mathbf{\fR = \{R_1,\ldots,R_k\}}$ with}  & {\bf vanishing range for $\mathbf{\HH_i}$} \\
\hline
(i)   & $R_1=\{\Span{v_1},\Span{v_2},\Span{v_3},\Span{\epsilon_1 v_1+\epsilon_2 v_2},$     & $((n+2)-5)+1-3$ \\
      & \hspace{\widthof{$R_2=$\ }}$\Span{\epsilon_1 v_1 + \epsilon_2 v_2+\epsilon_3 v_3}\}$ & $=n-5$ \\
\hline
(ii)  & $R_1 = \{\Span{v_1},\Span{v_2},\Span{\epsilon_1 v_1+\epsilon_2 v_2}\}$   & $((n+2)-6)+2-3$\\
      & $R_2 = \{\Span{v_3},\Span{v_4},\Span{\epsilon_3 v_3 + \epsilon_4 v_4}\}$ & $=n-5$ \\
\hline
(iii) & $R_1 = \{\Span{v_1},\Span{v_2},\Span{v_3},\Span{\epsilon_1 v_1+\epsilon_2 v_2+\epsilon_3 v_3}\}$ & $((n+2)-7)+2-3$\\
      & $R_2 = \{\Span{v_4},\Span{v_5},\Span{\epsilon_4 v_4 + \epsilon_5 v_5}\}$                         & $n-6$ \\
\hline
(iv)  & $R_1 = \{\Span{v_1},\Span{v_2},\Span{v_3},\Span{\epsilon_1 v_1+\epsilon_2 v_2+\epsilon_3 v_3}\}$ & $((n+2)- 8)+2-3$\\
      & $R_2 = \{\Span{v_4},\Span{v_5},\Span{v_6},\Span{\epsilon_4 v_4+\epsilon_5 v_5+\epsilon_6 v_6}\}$ & $=n-7$
\end{tblr}
\end{center}
Since $n \geq 4$ in all cases by assumption and additionally $n \geq 5$ in case (iii) and $n \geq 6$ in case (iv), we conclude that
in all cases $\HH_{-1}(\bX_{\bullet,2}[S]) = 0$.

\begin{claim}{(surjectivity)}
For $n=4$, the differential $\delta\colon \ssE^1_{20} \rightarrow \ssE^1_{10}$ is surjective.
\end{claim}

As we observed at the end of the proof Claim \ref{claim:exactness.1} (see the paragraph after \eqref{eqn:badcase}), the term
$\ssE^1_{10}$ is generated by elements
$[\kappa]$ constructed as follows.  Let $v_1,\ldots,v_4$ be
a basis for $\Z^4$.  For $1 \leq j \leq 4$, set $L_j = \Span{v_j}$.  Next, 
let $\epsilon_1,\epsilon_2,\epsilon_3 \in \{\pm 1\}$ and set $L_{123} = \Span{\epsilon_1 v_1+\epsilon_2 v_2+\epsilon_3 v_3}$.
Define
\[\kappa = \Gen{L_4 | L_1, L_2, L_3, L_{123}} - \Gen{L_1,L_2,L_3,L_{123} | L_4} \in \bX_{01}(\Z^4).\]
We then have $\partial \kappa = 0$, so there is an element
$[\kappa] \in \HH_0(\bX_{\bullet,0}(\Z^4)) = \ssE^1_{10}$.

We must prove that $[\kappa]$ is in the image of the differential $\delta\colon \ssE^1_{20} \rightarrow \ssE^1_{10}$.
Set $L_{12} = \Span{\epsilon_1 v_1+\epsilon_2 v_2}$.  Define
\[\eta = \Gen{L_{12},L_1,L_2,L_3,L_{123} | L_4} + \Gen{L_4 | L_{12},L_1, L_2, L_3, L_{123}} \in \bX_{02}(\Z^4).\]
We have
\begin{align*}
\partial \eta &= \Gen{L_{12},L_1,L_2,L_3,L_{123},L_4} + \Gen{L_4, L_{12},L_1, L_2, L_3, L_{123}} \\
              &= \Gen{L_{12},L_1,L_2,L_3,L_{123},L_4} - \Gen{L_{12},L_1,L_2,L_3,L_{123},L_4} = 0.
\end{align*}
It follows that we have $[\eta] \in \HH_0(\bX_{\bullet,2}(\Z^4)) = \ssE^1_{20}$.
The differential $\delta\colon \ssE^1_{20} \rightarrow \ssE^1_{10}$ is induced by the differential
$\delta\colon \bX_{02}(\Z^4) \rightarrow \bX_{01}(\Z^4)$ discussed in \S \ref{section:doublecomplexdef}, which
in turn is induced by the differential discussed in \S \ref{section:chaincomplex}.  To prove the claim,
it is enough to prove that $[\delta \eta] = [\kappa]$. 

The element $\delta \eta$ is an alternating sum of the result of deleting terms from $\eta$.  If we delete
$L_4$ then the result is $0$.  From this, we see that $\delta \eta = \kappa+\kappa_2 + \cdots + \kappa_5$ with
\begin{align*}
\kappa   &= \Gen{L_1,L_2,L_3,L_{123} | L_4}     - \Gen{L_4 | L_1,L_2,L_3,L_{123}} \\
\kappa_2 &= -\Gen{L_{12},L_2,L_3,L_{123} | L_4} + \Gen{L_4 | L_{12},L_2,L_3,L_{123}} \\
\kappa_3 &= \Gen{L_{12},L_1,L_3,L_{123} | L_4}  - \Gen{L_4 | L_{12},L_1,L_3,L_{123}} \\
\kappa_4 &= -\Gen{L_{12},L_1,L_2,L_{123} | L_4} + \Gen{L_4 | L_{12},L_1,L_2,L_{123}} \\
\kappa_5 &= \Gen{L_{12},L_1,L_2,L_3 | L_4}      - \Gen{L_4 | L_{12},L_1,L_2,L_3}
\end{align*}
Each $\kappa_j$ satisfies $\partial \kappa_j = 0$.  For instance,
\begin{align*}
\partial \kappa_2 &= -\Gen{L_{12},L_2,L_3,L_{123},L_4} + \Gen{L_4,L_{12},L_2,L_3,L_{123}} \\
                  &= -\Gen{L_{12},L_2,L_3,L_{123},L_4} + \Gen{L_{12},L_2,L_3,L_{123},L_4} = 0.
\end{align*}
We therefore have elements $[\kappa_j] \in \HH_0(\bX_{\bullet,0}(\Z^4)) = \ssE^1_{10}$, and can
write
\[[\delta \eta] = [\kappa] + [\kappa_2] + \cdots + [\kappa_5].\]
To prove that $[\delta \eta] = [\kappa]$, it is enough to prove that all the $\kappa_j$ satisfy
$[\kappa_j] = 0$.  In fact, all of them are terms that we proved are zero during the proof
of Claim \ref{claim:exactness.1}.  For example, to see that $[\kappa_2] = 0$, let $\{v'_1,\ldots,v'_4\}$ be the
basis
\[(v'_1,v'_2,v'_3,v'_4) = (\epsilon_1 v_1+\epsilon_2 v_2,\epsilon_3 v_3,v_2,v_4).\]
The collection of lines $S = \{L_{12},L_2,L_3,L_{123},L_4\}$ used to form $\kappa_2$ are then of the form
\[S = \{\Span{v'_1},\Span{v'_2},\Span{v'_3},\Span{v'_4},\Span{v'_1 + v'_2}\}.\]
This makes it clear that $\kappa_1$ is one of the terms proved to vanish in case (i) of the
proof of Claim \ref{claim:exactness.1}.
\end{proof}


\begin{thebibliography}{99}

\bibitem{AshMillerPatzt}
A. Ash, J. Miller, \& P. Patzt,
Hopf algebras, Steinberg modules, and the unstable cohomology of $\SL_n(\Z)$ and $\GL_n(\Z)$,
preprint 2024.
\arxiv{2404.13776}

\bibitem{AshPutmanSam}
A. Ash, A. Putman, \& S. Sam,
Homological vanishing for the Steinberg representation,
Compos. Math. 154 (2018), no. 6, 1111--1130.

\bibitem{AshRudolph}
A. Ash \& L. Rudolph, 
The modular symbol and continued fractions in higher dimensions, 
Invent. Math. 55 (1979), no. 3, 241--250.

\bibitem{BorelBook}
A. Borel, 
Linear algebraic groups, 
second edition, Graduate Texts in Mathematics, 126, Springer-Verlag, New York, 1991.

\bibitem{BorelSerreCorners}
A. Borel \& J.-P. Serre, 
Corners and arithmetic groups, 
Comment. Math. Helv. 48 (1973), 436--491.

\bibitem{BorelTitsReductive}
A. Borel \& J. Tits, Groupes r\'{e}ductifs, Publ. Math. I.H.E.S. 27 (1965), 55--152.

\bibitem{BorelTitsReductive2}
A. Borel \& J. Tits, Compl\'ements \`a{} l'article: ``Groupes r\'eductifs'', Inst. Hautes \'Etudes Sci. Publ. Math. No. 41 (1972), 253--276.

\bibitem{BrownBuildings}
K. S. Brown, Buildings, Springer, New York, 1989.

\bibitem{BruckMillerPatztSrokaWilson}
B. Br\"{u}ck, J. Miller, P. Patzt, R. Sroka, \& J. Wilson,
On the codimension-two cohomology of $\SL_n(\Z)$,
Adv. Math. 451 (2024), 109795.

\bibitem{BruckPatztSroka}
B. Br\"{u}ck, P. Patzt, \& R. Sroka,
A presentation of symplectic Steinberg modules and cohomology of $\Sp_{2g}(\Z)$,
preprint 2023.
\arxiv{2306.03180}

\bibitem{BruckSantosRegoSroka}
B. Br\"{u}ck, Y. Santos Rego, \& R. Sroka,
On the top-dimensional cohomology of arithmetic Chevalley groups,
Proc. Amer. Math. Soc. 152 (2024), 4131--4139.

\bibitem{Bykovskii}
V.A. Bykovskii, 
Generating elements of the annihilating ideal for modular symbols,
Funktsional. Anal. i Prilozhen. 37 (2003), no. 4, 27--38, 95; 
translation in Funct. Anal. Appl. 37 (2003), no. 4, 263--272.

\bibitem{SteinbergPolylogarithms}
S. Charlton, D. Radchenko, \& D. Rudenko,
Multiple polylogarithms and the Steinberg module,
preprint 2025.
\arxiv{2505.02202}

\bibitem{ChurchFarbPutmanConjecture}
T. Church, B. Farb, \& A. Putman,
A stability conjecture for the unstable cohomology of $\SL_n(\Z)$, mapping class groups, and $\Aut(F_n)$,
in ``Algebraic Topology: Applications and New Directions'', 55--70,
Contemp. Math. 620 (2014), Amer. Math. Soc., Providence, RI.

\bibitem{ChurchPutmanCodimOne}
T. Church \& A. Putman,
The codimension-one cohomology of $\SL_n \Z$,
Geom. Topol. 21 (2017), no. 2, 999--1032.

\bibitem{CurtisBN}
C. W. Curtis, The Steinberg character of a finite group with a $(B,\,N)$-pair, J. Algebra 4 (1966), 433--441.

\bibitem{HumphreysSurvey}
J. E. Humphreys, 
The Steinberg representation, 
Bull. Amer. Math. Soc. (N.S.) 16 (1987), no. 2, 247--263.

\bibitem{LeeSzczarba}
R. Lee \& R.H. Szczarba, On the homology and cohomology of congruence subgroups, Invent. Math. 33 (1976) 1, 15--53.

\bibitem{MillerNagpalPatzt}
J. Miller, R. Nagpal, \& P. Patzt, 
Stability in the high-dimensional cohomology of congruence subgroups, 
Compos. Math. 156 (2020), no.~4, 822--861.

\bibitem{MillerNagpalPatztCorrection}
J. Miller, R. Nagpal, \& P. Patzt,
Corrigendum to ``Stability in the high-dimensional cohomology of congruence subgroups''
[Compos. Math. 156 (2020), no.~4, 822--861]
\arxiv{2508.14945}

\bibitem{MillerPatztPutmanAlternate}
J. Miller, P. Patzt, \& A. Putman,
The double Tits building and the bar resolution of the Steinberg representation,
in preparation.

\bibitem{MillerPatztWilsonRognes}
J. Miller, P. Patzt, \& J. Wilson,
On rank filtrations of algebraic K-theory and Steinberg modules,
to appear in J. Eur. Math. Soc.
\arxiv{2303.00245}.

\bibitem{MilneBook}
J.~S. Milne, {\it Algebraic groups}, Cambridge Studies in Advanced Mathematics, 170, Cambridge Univ. Press, Cambridge, 2017.

\bibitem{PutmanSnowden}
A. Putman \& A. Snowden,
The Steinberg representation is irreducible,
Duke Math. J. 172 (2023), no. 4, 775--808.

\bibitem{Reeder}
M. Reeder, 
The Steinberg module and the cohomology of arithmetic groups, 
J. Algebra 141 (1991), no. 2, 287--315.

\bibitem{RognesConjecture}
J. Rognes, 
A spectrum level rank filtration in algebraic K-theory, 
Topology 31 (1992), no. 4, 813--845.

\bibitem{KupersKTheory}
M. Sikiri\'{c}, P. Elbaz-Vincent, A. Kupers, \& J. Martinet, Voronoi complexes in higher dimensions, cohomology of $\GL_N(\Z)$ for $N \geq 8$, and the triviality of $K_8(\Z)$,
preprint 2019. \arxiv{1910.11598}

\bibitem{SolomonTits}
L. Solomon, The Steinberg character of a finite group with $BN$-pair, in {\it Theory of Finite Groups (Symposium, Harvard Univ., Cambridge, Mass., 1968)}, 213--221, Benjamin, New York.

\bibitem{Steinberg1}
R. Steinberg, A geometric approach to the representations of the full linear group over a Galois field, Trans. Amer. Math. Soc. 71 (1951), 274--282.

\bibitem{Steinberg2}
R. Steinberg, Prime power representations of finite linear groups, Canadian J. Math. 8 (1956), 580--591.

\bibitem{Steinberg3}
R. Steinberg, Prime power representations of finite linear groups. II, Canadian J. Math. 9 (1957), 347--351.

\bibitem{SteinbergSurvey}
R. Steinberg, 
Comments on the Papers, 
in Robert Steinberg, Collected Works, 7, American Mathematical Society, Providence, RI, 1997.

\bibitem{TitsSurvey}
J. Tits,
Classification of algebraic semisimple groups,
in {\it Algebraic Groups and Discontinuous Subgroups} (Proc. Sympos. Pure Math., Boulder, Colo., 1965), 
pp. 33--62 American Mathematical Society, Providence, RI, 1966.

\bibitem{TitsBN}
J. Tits, 
Buildings of spherical type and finite BN-pairs, 
Lecture Notes in Mathematics, Vol. 386, Springer-Verlag, Berlin, 1974.

\end{thebibliography}
\end{document}